\documentclass[reqno,a4paper,11pt]{amsart}

\usepackage[english]{babel}

\usepackage{amsfonts}
\usepackage{amsmath}
\usepackage{amsthm}
\usepackage{amssymb}
\usepackage{indentfirst}
\usepackage{xcolor}
\usepackage{tabularx}
\usepackage{graphicx}
\usepackage{esint}
\usepackage{dsfont}

\allowdisplaybreaks

\usepackage{soul}

\usepackage[a4paper,top=3cm,bottom=3cm,left=2.5cm,right=2.5cm]{geometry}

\usepackage{subcaption}

\usepackage{mathrsfs}
\usepackage{mathtools}
\usepackage{enumitem}
\usepackage[normalem]{ulem}

\usepackage{empheq}
\usepackage[most]{tcolorbox}

\usepackage{framed}
\newenvironment{mathframed}{\framed%
\allowdisplaybreaks
\vspace*{-\abovedisplayskip}\noindent}{%
\vspace*{-\dimexpr\baselineskip+\topsep}\endframed}

\usepackage[utf8]{inputenc}
\usepackage[T1]{fontenc}
\usepackage[english]{babel}

\numberwithin{equation}{section}

\theoremstyle{plain}
\begingroup
\theoremstyle{plain}
\newtheorem{theorem}{Theorem}[section]
\newtheorem{corollary}[theorem]{Corollary}
\newtheorem{proposition}[theorem]{Proposition}
\newtheorem{lemma}[theorem]{Lemma}
\theoremstyle{definition}
\newtheorem{definition}[theorem]{Definition}

\theoremstyle{remark}
\newtheorem{remark}[theorem]{Remark}
\endgroup

\theoremstyle{definition}
\theoremstyle{remark}

\mathchardef\emptyset="001F

\newcommand{\R}{\mathbb{R}}
\newcommand{\E}{\mathcal{E}}

\newcommand{\dive}{\mathrm{div}}

\newcommand{\spt}{\mathrm{spt}}

\newcommand{\G}{\mathcal{G}}

\newcommand{\di}{\mathrm{d}}

\newcommand{\Om}{\Omega}

\newcommand{\F}{\mathcal{F}}

\newcommand{\J}{\mathcal{J}}

\newcommand{\y}{\boldsymbol{y}}

\newcommand{\de}{\mathrm{d}}
\newcommand{\z}{\boldsymbol{z}}

\newcommand{\bX}{\boldsymbol{X}}

\newcommand{\cP}{\mathcal{P}}

\newcommand{\cW}{\mathcal{W}}

\newcommand{\uu}{\boldsymbol{u}}
\newcommand{\ev}{\mathrm{ev}}
\newcommand{\bmu}{\boldsymbol{\mu}}

\definecolor{dred}{rgb}{.8,0,0}
\definecolor{ddmagenta}{rgb}{0.7,0,0.9}
\definecolor{ddcyan}{rgb}{0,0.2,1.0}
\definecolor{Orchid}{rgb}{0.7,0.4,0}
\definecolor{blue_links}{RGB}{13,0,180} 
\definecolor{lightblue}{RGB}{0.9,0.9,1}

\newcommand{\EEE}{\color{black}}

\newcommand{\res}{\mathop{\hbox{\vrule height 7pt width .5pt depth
0pt\vrule height .5pt width 6pt depth0pt}}\nolimits}

\usepackage{hyperref}
\hypersetup{
    colorlinks=true, 
    linktoc=all,     
    linkcolor=blue_links,  
    citecolor=blue_links,
    urlcolor=blue_links,
}

\newcommand{\N}{\mathbb{N}}

\newcommand{\coloneq }{\hspace{1pt}\raisebox{0.74pt}{\scalebox{0.8}{:}}\hspace{-2.2pt}=}

\newcommand{\eeta}{\boldsymbol{\eta}}

\makeatletter
\@namedef{subjclassname@2020}{%
  \textup{2020} Mathematics Subject Classification}
\makeatother

\title[Optimal control problems with additive noise]{Optimal control problems in transport dynamics with additive noise}

\author[S. Almi]{Stefano Almi}
\address[Stefano Almi]{Dipartimento di Matematica e Applicazioni ``R.~Caccioppoli'',
Universit\`a di Napoli Federico II, via Cintia, 80126 Napoli, Italy.}
\email{stefano.almi@unina.it}

\author[M. Morandotti]{Marco Morandotti}
\address[Marco Morandotti]{Dipartimento di Scienze Matematiche ``G.~L.~Lagrange'',
Politecnico di Torino, Corso Duca degli Abruzzi 24,
10129 Torino, Italy.}
\email{marco.morandotti@polito.it}

\author[F. Solombrino]{Francesco Solombrino}
\address[Francesco Solombrino]{Dipartimento di Matematica e Applicazioni ``R.~Caccioppoli'',
Universit\`a di Napoli Federico II, via Cintia, 80126 Napoli, Italy.}
\email{francesco.solombrino@unina.it}

\date{\today}

\keywords{Mean-field optimal control, optimal control with SDE constraints, population dynamics, $\Gamma$-convergence, superposition principle}

\begin{document}
\subjclass[2020]{49N80, 
35Q93, 
(
49J45, 
60H10, 
49M41, 
93E20)} 

\begin{abstract}
Motivated by the applications, a class of optimal control problems is investigated, where the goal is to influence the behavior of a given population through another controlled one interacting with the first. Diffusive terms accounting for randomness in the evolution are taken into account. A well-posedness theory under very low regularity of the control vector fields is developed, as well as a rigorous derivation from stochastic particle systems.
\end{abstract}

\maketitle

\tableofcontents

\bigskip

\section{Introduction}
\label{s:intro}

Many evolutionary models in population dynamics are usually formulated in the form of a Fokker--Planck-type equation for the time-dependent population density~$\mu_t$
\begin{equation}\label{e:FP}
\partial_t \mu_t+\dive(v \mu_t)-\sigma\Delta \mu_t=0
\end{equation}
coupling a transport dynamics (encoded by the divergence term and driven by the velocity field~$v_t$) with a diffusion term (encoded by the laplacian).
The probabilistic counterpart of such equation is the stochastic differential equation
\begin{equation}\label{e:SL}
\de X(t)=v(t,X(t))\,\de t+\sqrt{2\sigma}\,\de W(t),
\end{equation}
where~$X$ is a random variable and~$W$ denotes the Brownian motion, which can be interpreted as the Lagrangian formulation of the Eulerian problem~\eqref{e:FP}. 

In many interesting situations, the vector field $v(t,x)\eqqcolon v_{\mu_t}(t,x)$ may itself depend on the population density~$\mu_t$ in a rather general way. A standard situation is, for instance,
\begin{equation}\label{e:vel_field}
v(t, x) = (K *\mu_t)(t, x) + f (t, x),
\end{equation}
where $f(t,x)$ is an external velocity field and $K(t,x)$ is a self-interaction kernel. This gives both a non-linear and non-local character to equation~\eqref{e:FP}.
In many models, velocity terms of the kind~\eqref{e:vel_field} account for elementary attraction and repulsion forces between the members of the population.

If we neglect the diffusion term in~\eqref{e:FP}, the corresponding evolution of~$\mu_t$ obeys a transport-like dynamics driven by a non-local continuity equation. 
For this kind of dynamics, a general class of optimal control problems has been addressed in~\cite{bongini2016optimal}, where the goal is to modify the behavior of the population $\mu_t$  through the interaction with a selected population of leaders, whose density is~$\nu_t$.  
There, two scenarios with increasing complexity were explored:
\begin{itemize}
\item the evolution of~$\nu_t$ is determined by the optimization of a cost functional~$\mathcal{J}(\mu,\nu)$ and~$\mu$ and~$\nu$ are coupled by a velocity field of the form 
\begin{equation}\label{e:ev_v}
v(t, x) = (K *\mu_t)(t, x) + (H *\nu_t)(t, x),
\end{equation}
(see~\cite[Problem 1]{bongini2016optimal});
\item the evolution of~$\nu_t$ obeys itself a non-local continuity equation driven by a velocity field~$w(t,x)+u(t,x)$, where~$w$ has the structure as in~\eqref{e:ev_v} and the control term~$u(t,x)$ optimizes a control cost $\mathcal{J}(\mu,\nu,u)$ (see~\cite[Problem~2]{bongini2016optimal}).
\end{itemize}
For both problems above, general conditions on the control cost~$\mathcal{J}$ and on the class of admissible controls were introduced in order to provide well posedness.
We remark, in particular, that in the context of \cite[Problem 2]{bongini2016optimal} the admissible control~$u(t,\cdot)$ must satisfy a Lipschitz condition with prescribed Lipschitz constant~$L$.

The aim of the present paper is to extend the results of~\cite{bongini2016optimal} in two directions which are, in our opinion, interesting both from the point of view of modeling and of mathematical analysis.

First of all, for the dynamics of the population~$\mu_t$, we consider equation~\eqref{e:FP} with the presence of the diffusive term. This term is actually reminiscent of the agent-based interpretation of equation~\eqref{e:FP}, which can be seen as an effective limit model for a particle dynamics with a very large number of agents.
In this approximation, the inevitable loss in accuracy is taken into account by adding some white noise to the system.
Furthermore, such a term may also express the fact that individuals of the population~$\mu_t$ can exhibit some random behavior, despite being driven by interactions with other agents or with the leaders.
When coming to our analogue to \cite[Problem 2]{bongini2016optimal}, we find instead natural to postulate that the action of the leaders is completely determined by a policy-maker through the control vector field~$u$ without inserting a diffusive term in the dynamics.

The precise formulation of the two control problems we propose is given in Sections~\ref{s:general} and~\ref{s:2pop}, respectively.
In particular, Problem~2 (see \eqref{e:min-mean}) is the generalization of the optimal control problem analyzed in~\cite{ASC_22} for a discrete fixed number~$m$ of leaders (see also \cite{CGS2017,gorlando} for related problems in piracy control and maritime crime prevention).
Instead, in our formulation, this restriction is lifted and an effective macroscopic model also for~$\nu_t$ is considered.

The second novelty of our approach is that in Problem~2 we allow for a large class of admissible controls with very low regularity, namely the vector fields we consider are of the form 
$$u(t,x)=f (t, x) g(\mu_{t}) \,,$$
where the function~$f$ is only of class~$L^\infty$ in both space and time.
The presence of the term~$g(\mu_t)$ is an additional modeling possibility allowing the policy maker to tune the control action on the actual state of the system.
The above class of control vector fields is, in principle, the one considered in \cite{AMOP_2022,ASC_22,CLOS_2022,Flos,gorlando}.
In \cite{ASC_22,gorlando}, however, well posedness of the optimal control problems was considerably simplified by the assumption that the leaders' population remain discrete.
In our setting, we have instead to resort to the superposition principle \cite[Theorem~5.2]{AmbForMorSav18}, \cite[Chapter~8]{AmbGigSav08}, \cite{AmbTre14}, as it will be clear from the proof of Theorem~\ref{p:existence-mean} below. 
This tool has already proven to be crucial in connection with the problems considered in \cite{AMOP_2022,CCDMP2021, CLOS_2022,Flos}, where however no diffusive terms were present in the state equations. In our setting its use has therefore to be combined with some a priori estimates for equation~\eqref{e:FP} which are recovered by looking at its stochastic Lagrangian counterpart~\eqref{e:SL} and employing some fixed point argument.

We stress that our control problem has a different formulation from that of mean-field games, introduced in~\cite{HMC,LL}. 
While, there, the decentralized control rules are embedded inside the dynamics of~$\mu$, in our setting a control mass~$\nu$ interacts with the original population with the aim of influencing its behavior. 
For mean-field games in the context of Fokker--Planck-type equations, we refer the reader to \cite{Camilli_games,CPT_2017,Porretta}.

In the last two sections of the paper, we specifically focus on the rigorous derivation of Problem~2 (see \eqref{e:min-mean}) as the deterministic variational limit of a stochastic optimal control problem associated with a particle dynamics with additive noise.
In doing so, we adapt to our setting $\Gamma$-convergence techniques combined with the derivation of kinetic equations as the mean-field limit of agent-based systems \cite{Kac}.
The latter is a rather effective tool to overcome the \emph{curse of dimensionality} for systems with a very large number of agents. 
Indeed, kinetic approximations of multi-agent systems and mean-field optimal control problems, mostly in the deterministic setting, have been proposed in recent literature in connection with a huge number of possible applications, ranging from models for 
opinion formation \cite{BMPW2009,Toscani2006}, 
wealth distribution \cite{MR4147945,MR2551376,pareschi2014wealth}, 
traffic or pedestrian flows \cite{albi2016invisible,PR_traffic2,PR2018,PTZ2020,TZ2019}, 
herding problems \cite{AMOP_2022,albi2016invisible,Auletta2020HerdingSA,burger2020instantaneous,osti_10386637,Pierson2018ControllingNH,nature2016}, 
consensus-based optimization \cite{Herty-Pareschi,MR3804923,fornasier2022anisotropic,MR4160256}
(see also \cite{BKT21,CLOS_2022,kalise2020sparse,PiccoliRossiARMA1,PiccoliRossiARMA2} for rigorous derivations and further applications and \cite{bonnet2023measure,BF2021b,BF2021a,BR_PMP,BPT21} for optimality conditions).
In the context of multi-agent systems with stochastic noise, but without control, we also refer the reader to \cite{bolley2011stochastic}, while mean-field control problems with diffusion terms have been recently considered in \cite{albi2,carrillo2020mean}.

The particle approximation of problem \eqref{e:min-mean} is introduced in Section~\ref{s:particles}, where we couple a system of~$M$ agents (followers) driven by a stochastic dynamics as in~\eqref{e:SL} with the evolution of~$m$ selected and controlled agents (leaders).
Although the leaders' evolution is formally deterministic, the coupling with the followers' evolution (which is affected by additive noise) gives a stochastic character to the whole system.
A cost functional associated with the system, taking into account its expected behavior, is introduced in~\eqref{e:cost}.
The derivation of the state equation~\eqref{e:mean-field} as the mean-field limit of the particle system~\eqref{e:finite-particles}, as it is usual in the stochastic setting, goes through some propagation-of-chaos estimates, which we develop in Section~\ref{s:chaos}.
In particular, we prove that the initially coupled positions of the agents become independent in the limit as~$M$ becomes larger and larger uniformly with respect to~$m$.
In other words, the limit behavior of the particle system can be described by~$M$ copies of the SDE/ODE system in~\eqref{e:chaos}, where the coupling only takes place through the law of the random variable~$X$.
In the limit as $M,m\to\infty$, we eventually recover the PDE system~\eqref{e:mean-field}.

As a conclusive step, in Section~\ref{s:limit}, we recover the deterministic mean-field optimal control problem~\eqref{e:min-mean} as the $\Gamma$-limit of the stochastic optimal control problem~\eqref{e:min-finite}.
A major difficulty has to be overcome in the $\Gamma$-$\limsup$ inequality. 
We remark, indeed, that the sole integrability of the control field~$f$ is not enough to guarantee the existence of a flow map for system~\eqref{e:finite-particles}.
Hence, the construction of a recovery sequence for problem~\eqref{e:min-mean} has to combine the usual discretization arguments with the use of the superposition principle in order to detect suitable discrete trajectories converging to the mean-field evolution associated with the given control.

While the present paper is devoted to the well-posedness of a class of mean-field optimal control problems with diffusion terms coming from stochastic noise, further interesting steps concerning the numerical approximation of solutions through discrete-in-time schemes (in the spirit of \cite{AlmMorSol21}), as well as the derivation of first-order optimality conditions, will be the subject of future research.

\section{Preliminaries and notation}
\label{s:preliminaries}

For $d \in \mathbb{N}$ and~$T>0$ we denote by~$\mathcal{M}_{b} ([0, T]\times \R^{d}; \R^{d})$ the space of vector-valued bounded Radon measures on~$[0, T] \times \R^{d}$. 
For a metric space~$(E,d_E)$, the symbol~$\mathcal{P}(E)$ stands for the set of probability measures on~$E$. 
For $p \in [1, +\infty)$, we denote by~$\mathcal{P}_{p} (E)$ the set of probability measures~$\mu$ on~$E$ with finite~$p$-moment
$$m_{p} (\mu)\coloneqq\int_E d_E^p(x,x_0)\,\de\mu(x),$$
where $x_0\in E$ is a given point.
We further denote by~$\cW_{p}$ the $p$-Wasserstein distance on~$\mathcal{P}_{p} (E)$. 
Given $f\colon E\to F$ a measurable function and $\mu\in\cP(E)$, the push forward of $\mu$ through $f$ is the probability measure in $\cP(F)$ defined by $(f_{\#}\mu)(B)\coloneqq \mu(f^{-1}(B))$, for every measurable subset $B\subseteq F$.
If~$f$ is, additionally, a Lipschitz function, then the following inequality holds true:
\begin{equation}\label{e:pfLip}
\cW_1(f_{\#}\mu,f_{\#}\nu)\leq {\rm Lip}(f)\cW_1(\mu,\nu),
\end{equation}
for every $\mu,\nu\in\cP(E)$, where ${\rm Lip}(f)>0$ is the Lipschitz constant of~$f$.

Along the paper we shall suppose, without loss of generality, that all the involved random variables are supported on a fixed filtered probability space $(\Om, \F, \F_{t}, \mathbb{P})$. We denote by~$\mathbb{E}$ the expectation operator
and we use the symbol~$\mathcal{M} (\Om;E)$ to indicate the space of $E$-valued random variables. For $X \in \mathcal{M} (\Om; E)$, we set ${\rm Law} (X) \coloneqq (X)_{\#}\mathbb{P}\in\cP(E)$ the push forward of~$\mathbb{P}$ through~$X$. We will denote by~$W$ an additive white noise.

For every $t \in [0, T]$ we denote by~$\ev_{t} \colon C([0, T]; \R^{d}) \to \R^{d}$ the evaluation map at time~$t$, defined as $\ev_{t} (\gamma) \coloneqq \gamma(t)$ for every $\gamma \in C([0, T]; \R^{d})$. 

We recall that for every $\bmu\in \cP(C([0,T];\R^{d}))$, setting $\mu_t\coloneqq (\ev_t)_{\#}\bmu$, it holds that $m_{p} (\mu_{t}) \leq m_{p} (\bmu)$ for every~$t \in [0, T]$. 
The curve $t\mapsto\mu_t$ will be often denoted by~$\mu$ alone. The same symbol will be used for the (positive) measure $\mu_t\otimes\de t\in\mathcal{M}_b([0,T]\times\R^d)$.

\begin{theorem}[\cite{Fou-Gui}]
\label{t:iid}
Let $p>1$, $\mu \in \mathcal{P}_{p} (\R^{d})$, and $X_{i}$ be a sequence of i.i.d. random variable with distribution~$\mu$. For $M \in \mathbb{M}$, let $\mu^{M} \coloneqq \frac{1}{M} \sum_{i=1}^{M} \delta_{X_{i}}$. Then, there exists a constant $C = C(d, p)>0$ such that for every $M$
\begin{displaymath}
\mathbb{E} \big( \cW_{1} (\mu, \mu^{M}) \big) \leq C m_{p} (\mu) 
\left\{
\begin{array}{lll}
M^{-\frac{1}{2}} + M^{-\frac{p-1}{p}} & \text{if $d=1$ and $p\neq 2$},\\[1mm]
M^{-\frac{1}{2}} \log (1 + M) + M^{-\frac{p-1}{p}} & \text{if $d = 2$ and $p\neq 2$},\\[1mm]
M^{-\frac{1}{d}} + M^{-\frac{p-1}{p}} & \text{if $d \geq 3$ and $p\neq \frac{d}{d-1}$}.
\end{array}
\right.
\end{displaymath} 
\end{theorem}

We recall the notion of \emph{pathwise solution} to a stochastic differential equation that will be used throughout the paper.
\begin{definition}\label{d:pathwise}
We say that $X \in \mathcal{M}(\Omega; C([0,T]; \R^d))$  is a \emph{pathwise} (or \emph{strong}) \emph{solution} to the stochastic differential equation 
$$\begin{cases}
\de X(t) = v(t,X(t))\,\de t + \sqrt{2\sigma}\,\de W(t),\\
X(0) = \overline{X}_0,
\end{cases}$$
for a given initial datum $\overline{X}_0\in L^p(\Omega;\R^d)$ and Brownian motion $W$, if there holds
$$X(t) = \overline{X}_0 +\int_0^t v(\tau,X(\tau))\,\de\tau+\sqrt{2\sigma} W(t)\quad\text{for $\mathbb{P}$-a.e.~$\omega\in\Omega$ and for every $t\in[0,T]$.}$$
\end{definition}  
The explicit dependence on the stochastic variable~$\omega\in\Omega$ has been omitted above, as will be done throughout the paper when no ambiguity arises.

We point out that if $X^{1}, X^{2} \in  \mathcal{M}(\Omega; C([0,T]; \R^d))$, $\bmu^{i} = {\rm Law} (X^{i}) \in \mathcal{P} (C([0, T]; \R^{d}))$, and $\mu^{i}_{t} = (\ev_{t})_{\#} \bmu^{i}$ for $i = 1, 2$, we have the following elementary inequalities, which stem out of the definition of~$\cW_{1}$:
\begin{align}
\cW_{1} (\bmu^{1}, \bmu^{2}) & \leq \mathbb{E} \bigg( \sup_{t \in [0, T]} | X^{1}(t) - X^{2}(t)| \bigg) \label{e:bmu}\,,
\\
\cW_{1} (\mu^{1}_{t} , \mu^{2}_{t} ) & \leq \mathbb{E} \big( | X^{1}(t) - X^{2}(t)| \big) \,. \label{e:muu}  
\end{align}

Finally, we say that a function $\rho \in L^{1} (\R^{d})$ has finite entropy if $\rho > 0$ and and
\begin{displaymath}
\int_{\R^{d}} \rho(x)  \ln (\rho(x)) \, \di x <+\infty\,.
\end{displaymath}

\section{A model problem}
\label{s:general}

We introduce a model control problem for the dynamics of a population with density~$\mu$ steered by means of another population of controllers with density~$\nu$.
To this aim, for $L, R>0$ and $q \in (1, +\infty]$ we define the class of admissible measure-valued curves
\begin{equation}
\mathcal{A} (q,  L, R) \coloneqq C ([0, T] ; \mathcal{P}_{1} (\R^{d})) \times \{ \nu \in  {\rm Lip}_{L} ([0, T]; \mathcal{P}_{1} (\R^{d})): \, m_{q} (\nu_{t}) \leq R \text{ for $t \in [0, T]$} \}\,.
\end{equation}
We fix a velocity field $v \colon \R^{d} \times \mathcal{P}_{1}(\R^{d})\times \mathcal{P}_{1}(\R^{d}) \to \R^{d}$ such that the following Lipschitz condition is satisfied: there exists a constant $L_v>0$ such that 
\begin{equation}\label{e:vLip}\tag{$v$-Lip}
\lvert v(x_1,\mu_1,\nu_1)-v(x_2,\mu_2,\nu_2)\rvert \leq L_v (\lvert x_1-x_2\rvert+\cW_1(\mu_1,\mu_2)+\cW_1(\nu_1,\nu_2))\,,
\end{equation}
for every $(x_1,\mu_1,\nu_1),(x_2,\mu_2,\nu_2)\in \R^{d} \times \mathcal{P}_{1}(\R^{d})\times \mathcal{P}_{1}(\R^{d})$.
We notice that condition \eqref{e:vLip} implies that there exists a constant $M_v>0$ such that 
\begin{equation}\label{e:old_v2}
|v( x, \mu, \nu)| \leq M_{v} (1 + |x| + m_{1} (\mu) + m_{1} (\nu))\,.
\end{equation}
From now on, we use the notation $v_{\mu, \nu} (x) = v( x, \mu, \nu)$.


On the set $\mathcal{A} (q, L, R)$ we want to solve the following control problem:

\begin{mathframed}
\textbf{Problem 1}
\begin{align}
\label{e:minJ}
&\min_{(\mu, \nu) \in \mathcal{A}(q, L, R)} \, \J (\mu, \nu)\,,\\
&
\label{e:system}
\text{subject to } \begin{cases}
\partial_{t} \mu_{t} - \sigma \Delta \mu_{t} = - \dive (v_{\mu_{t}, \nu_{t}} \mu_{t})\,,\\[1mm]
\mu_{0} = \overline{\mu}_{0}\,,
\end{cases}
\end{align}\vspace{2mm}
\end{mathframed}
\noindent for a given cost functional~$\J \colon C([0,T];\cP_1(\R^d))\times C([0,T];\cP_1(\R^d)) \to \R \cup \{+\infty\}$ which is lower-semicontinuous with respect to the convergence in $C([0,T];\cP_1(\R^d)) \times C([0,T];\cP_1(\R^d))$.
Notice that this is the exact analogue of Problem 1 in~\cite{bongini2016optimal}, up to the addition of a diffusive term coming from stochastic noise for the dynamics of~$\mu$.

In order to show existence of solutions to~\eqref{e:minJ}--\eqref{e:system}, it is convenient to first study the well-posedness of the PDE~\eqref{e:system} when $\nu \in {\rm Lip}_{L} ([0, T]; \mathcal{P}_{1}(\R^{d}))$ is fixed. 
To simplify the notation, for $t \in [0, T]$, $x \in \R^{d}$, and $\mu \in \mathcal{P}_{1}(\R^{d})$, we set 
\begin{equation}
\label{e:tildev}
\tilde{v}_{\mu} (t, x) \coloneqq v_{\mu, \nu_{t}} (x)\,.
\end{equation}
In what follows, we show that~\eqref{e:system} is equivalent to the SDE
\begin{equation}
\label{e:SDE}
\left\{
\begin{array}{lll}
\di X(t) = \tilde{v}_{\mu_{t}} (t, X(t)) \,\di t + \sqrt{2\sigma} \, \di W(t)\,,\\[2mm]
X(0) = \overline{X}_{0}\,, \ {\rm Law}(\overline{X}_{0}) = \overline{\mu}_{0}\,,\\[2mm]
\bmu= {\rm Law} (X)\,, \ \mu_{t} = ({\rm ev}_{t})_{\#} \bmu\,.
\end{array}\right.
\end{equation}
To this purpose, we start by showing, in the next theorem, existence and uniqueness of solutions to~\eqref{e:SDE}, together with some estimates (notice that only continuity of the measure~$\nu$ is required).
We point out that the estimate~\eqref{e:holder} below will ensure the continuity of the solution $t\mapsto\mu_t$ to~\eqref{e:system}, and therefore grants its membership to the set~$\mathcal{A}(q,L,R)$.
From now on, we let $p>1$.

\begin{theorem}
\label{t:SDE}
Let $v \colon [0, T] \times \mathcal{P}_{1}(\R^{d}) \times \mathcal{P}_{1}(\R^{d}) \to \R^{d}$ satisfy~\eqref{e:vLip}, 
let $\nu \in C ([0, T]; \mathcal{P}_{1} (\R^{d}))$,  let~$\tilde{v}$ be defined as in~\eqref{e:tildev}, and let $\overline{X}_{0} \in L^{p} (\Om; \R^{d})$. Then, the SDE~\eqref{e:SDE} with initial condition~$\overline{X}_{0}$ admits a unique solution. Moreover, there exists $C= C(p, T, M_{v})>0$ such that
\begin{align}
\label{e:momentp}
m_{p} (\bmu) & \leq C\bigg( 1 + m_{p} (\overline{\mu}_{0}) +  \int_{0}^{T} m_{1}^{p} (\nu_{\tau}) \, \di \tau +  \bigg(  \frac{p}{p-1} \bigg)^{p}\mathbb{E} (|W(T)|^{p})\bigg) \,,\\
\cW_{1} (\mu_{t}, \mu_{s}) & \leq C \int_{s}^{t}  m_{1} (\nu_{\tau})  \, \di \tau + C | t - s| ^{\frac{1}{4}}  \label{e:holder}
\\
&
\qquad + C|t - s| \bigg( 1 +  m_{p} ( \overline{\mu}_{0})^{\frac{1}{p}} + \bigg( \int_{0}^{T} m_{1}^{p} (\nu_{\tau}) \, \di \tau \bigg)^{\frac{1}{p}} +  \mathbb{E} (|W(T)|^{p})^{\frac{1}{p}} \bigg). \nonumber
\end{align}
\end{theorem}


\begin{proof}
The existence and uniqueness of the solution follows by an adaptation of the Banach fixed point argument of~\cite[Theorem~3.1]{ASC_22}, which in turn only relies on the Lipschitz continuity of the velocity field~$v$ (see (v)) and on the fact that $\nu \in C([0, T]; \mathcal{P}_{1} (\R^{d}))$.

We now estimate the $p$-moment of $\bmu$. For $ t \in [0, T]$ and $\omega \in \Om$, by~\eqref{e:old_v2} we have that
\begin{align}
\label{e:X}
| X(t) | & \leq   |\overline{X}_{0}| + \int_{0}^{t} | \tilde{v}_{\mu_{t}} (\tau, X(\tau)) | \, \di \tau + \sqrt{2\sigma} | W(t)|
\\
&
\leq  |\overline{X}_{0}| + M_{v} \int_{0}^{t} (1 + | X(\tau)| + m_{1} (\nu_{\tau}) + m_{1} (\mu_{\tau})) \, \di \tau + \sqrt{2\sigma} | W(t)| \nonumber\,.
\end{align}
Taking the $p$-power of~\eqref{e:X} and applying Gr\"onwall inequality we get that (here $C$ is a positive constant depending on~$p$, $T$, $\sigma$, and $M_{v}$ which may vary from line to line)
\begin{align}
\label{e:X2}
|X(t)|^{p} & \leq C \bigg( 1 + | \overline{X}_{0}|^{p} + \int_{0}^{t}  (m^{p}_{1} (\nu_{\tau}) + m^{p}_{1} (\mu_{\tau})) \, \di \tau + | W(t)|^{p} \bigg)e^{M_vt}
\\
&
\leq C \bigg( 1 + | \overline{X}_{0}|^{p} + \int_{0}^{t}  (m^{p}_{1} (\nu_{\tau}) + m_{p} (\mu_{\tau})) \, \di \tau + | W(t)|^{p} \bigg) e^{M_vt}\,. \nonumber
\end{align}
Averaging~\eqref{e:X2} over $\Omega$ and applying again Gr\"onwall inequality and the Doob's maximal inequality~\cite{Revuz} we obtain for every $t \in [0, T]$
\begin{align}
\label{e:X3}
m_{p} (\mu_{t}) & \leq C e^{M_vt} \bigg( 1 + m_{p} (\overline{\mu}_{0}) +  \int_{0}^{t}  m^{p}_{1} (\nu_{\tau}) \, \di \tau +\bigg(  \frac{p}{p-1} \bigg)^{p}\mathbb{E} (|W(T)|^{p}) \bigg) e^{\frac{C}{M_v}e^{M_vt}}\,.
\\
&
\leq C e^{M_vT} \bigg( 1 + m_{p} (\overline{\mu}_{0}) +  \int_{0}^{T}  m^{p}_{1} (\nu_{\tau}) \, \di \tau +\bigg(  \frac{p}{p-1} \bigg)^{p}\mathbb{E} (|W(T)|^{p}) \bigg) e^{\frac{C}{M_v}e^{M_vT}}. \nonumber
\end{align}
Inserting~\eqref{e:X3} into~\eqref{e:X2} we may continue with
\begin{align}
\label{e:X4}
\!\! |X(t)|^{p} & \leq C \bigg( 1 + | \overline{X}_{0}|^{p} + m_{p} (\overline{\mu}_{0}) + \!\int_{0}^{T} \!\!\! m_{1}^{p} (\nu_{\tau}) \, \di \tau + |W(t) |^{p} + \bigg(  \frac{p}{p-1} \bigg)^{p}\mathbb{E} (|W(T)|^{p}) \bigg) \,.
\end{align}
Taking the supremum over $t \in [0, T]$ in~\eqref{e:X4} and applying once again Gr\"onwall and Doob's maximal inequality we infer that
\begin{align*}
m_{p} (\bmu) & \leq \mathbb{E} \Big(\sup_{t \in [0, T]}\, | X(t)|^{p}\Big) \leq C\bigg( 1 + m_{p} (\overline{\mu}_{0}) + \int_{0}^{T} m_{1}^{p} (\nu_{\tau}) \, \di \tau +  \bigg(  \frac{p}{p-1} \bigg)^{p}\mathbb{E} (|W(T)|^{p})\bigg) \,,
\end{align*}
which is exactly~\eqref{e:momentp}.

It remains to prove~\eqref{e:holder}. Let us fix $s<t \in [0, T]$ and $\omega \in \Om$. Then, by~\eqref{e:old_v2}, it holds
\begin{align}
\label{e:continuity}
| X(t) - X(s)| & \leq \int_{s}^{t} | \tilde{v}_{\mu_{\tau}} (\tau, X(\tau)) | \, \di \tau + | W(t) - W(s)|
\\
&
 \leq M_{v} \int_{s}^{t} ( 1 + | X(\tau)| + m_{1}(\mu_{\tau}) + m_{1}(\nu_{\tau}) ) \, \di \tau +  | W(t) - W(s)|\,. \nonumber
\end{align}
By H\"older inequality we have that $m_{1} (\mu_{t}) \leq m_{p} (\mu_{t}) ^{\frac{1}{p}}$. Hence, in view of~\eqref{e:momentp} we may continue in~\eqref{e:continuity} with
\begin{align}
\label{e:continuity2}
| X(t) - X(s)| & \leq C \int_{s}^{t} \big(1 + | X(\tau)|  + m_{1} (\nu_{\tau}) \big) \, \di \tau + | W(t) - W(s)| 
\\
&
\qquad + C|t - s| \bigg( m_{p} (\overline{\mu}_{0})^{\frac{1}{p}} + \bigg( \int_{0}^{T} m_{1}^{p} (\nu_{\tau}) \, \di \tau \bigg)^{\frac{1}{p}} +  \bigg(  \frac{p}{p-1} \bigg) \mathbb{E} (|W(T)|^{p})^{\frac{1}{p}} \bigg). \nonumber
\end{align}
Averaging~\eqref{e:continuity2} over $\Omega$, applying H\"older inequality for the noise term, and using~\eqref{e:muu} and~\eqref{e:momentp}, we get
\begin{align}
\cW_{1} (\mu_{t}, \mu_{s}) & \leq \mathbb{E} (|X(t) - X(s)|) \leq C \int_{s}^{t}  \big( m_{1} (\mu_{\tau}) + m_{1} (\nu_{\tau}) \big) \, \di \tau + \mathbb{E} (| W(t) - W(s)| ) 
\\
&
\qquad +  C|t - s| \bigg( 1 + m_{p} (\overline{\mu}_{0})^{\frac{1}{p}} + \bigg( \int_{0}^{T} m_{1}^{p} (\nu_{\tau}) \, \di \tau \bigg)^{\frac{1}{p}} +  \mathbb{E} (|W(T)|^{p})^{\frac{1}{p}} \bigg). \nonumber
\\
&
\leq C \int_{s}^{t} m_{1} (\nu_{\tau})  \, \di \tau + \mathbb{E} (| W(t) - W(s)|^{2} )^{\frac{1}{2}}  \nonumber
\\
&
\qquad +  C|t - s| \bigg( 1 + m_{p} ( \overline{\mu}_{0})^{\frac{1}{p}} + \bigg( \int_{0}^{T} m_{1}^{p} (\nu_{\tau}) \, \di \tau \bigg)^{\frac{1}{p}} +  \mathbb{E} (|W(T)|^{p})^{\frac{1}{p}} \bigg). \nonumber
\end{align}
Finally, by standard estimates of the Brownian motion (see, e.g.,~\cite{Evans}) we deduce that
\begin{align*}
\cW_{1} (\mu_{t}, \mu_{s}) & \leq C \int_{s}^{t}  m_{1} (\nu_{\tau})  \, \di \tau + C | t - s| ^{\frac{1}{4}}
\\
&
\qquad + C|t - s| \bigg( 1 +  m_{p} ( \overline{\mu}_{0})^{\frac{1}{p}} + \bigg( \int_{0}^{T} m_{1}^{p} (\nu_{\tau}) \, \di \tau \bigg)^{\frac{1}{p}} +  \mathbb{E} (|W(T)|^{p})^{\frac{1}{p}} \bigg). \nonumber
\end{align*}
This concludes the proof of~\eqref{e:holder} and of the theorem.
\end{proof}

We now show a continuity property of solutions to equation~\eqref{e:SDE} when varying~$\nu$.

\begin{proposition}
\label{p:12}
Let $v \colon [0, T] \times \mathcal{P}_{1}(\R^{d}) \times \mathcal{P}_{1}(\R^{d}) \to \R^{d}$ satisfy~\eqref{e:vLip}, let $\overline{X}_{0} \in L^{p} (\Om; \R^{d})$, and let~$\nu^{1}, \nu^{2} \in C([0, T]; \mathcal{P}_{1} (\R^{d}))$. 
Moreover, let $X^{1}, X^{2}  \in  \mathcal{M} (\Om;  C([0, T]; \R^{d}))$ be the corresponding solutions to~\eqref{e:SDE} with $\bmu^{i} = {\rm Law} (X^{i})$ and velocities~$\tilde{v}_{\mu^i_{t}} = v_{\mu^i_{t}, \nu^i_{t}}$ for $i =1, 2$. Then, there exists $C = C(v, T )>0$ such that for every $t \in [0, T]$
\begin{align}
\label{eee}
& \mathbb{E} \big( | X^{1} (t) - X^{2} (t)| \big) \leq C \int_{0}^{T} \cW_{1} (\nu^1_{t}, \nu^2_{t}) \, \di t \,,\\
& \label{eee-1} | X^{1} (t) - X^{2} (t)| \leq C \int_{0}^{T} \cW_{1} (\nu^1_{t}, \nu^2_{t}) \, \di t \qquad \text{for $\mathbb{P}$-a.e.~$\omega \in \Om$.}
\end{align} 
\end{proposition}

\begin{proof}
By~(v) we estimate for $\mathbb{P}$-a.e.~$\omega \in \Om$ and every $t \in [0, T]$
\begin{align}
\label{e:k1}
| X^{1}(t) - X^{2} (t)| & \leq \int_{0}^{t} | v_{\mu^1_{\tau}, \nu^1_{\tau}} (X^{1} (\tau)) - v_{\mu^2_{t}, \nu^2_{t}} (X^{2}(\tau)) | \, \di \tau
\\
&
\leq L_{v} \int_{0}^{t} \big( | X^{1} (\tau) - X^{2} (\tau) | + \cW_{1} (\mu^1_{\tau}, \mu^2_{\tau}) + \cW_{1} (\nu^1_{\tau}, \nu^2_{\tau}) \big) \di \tau \nonumber\,.
\end{align}
Applying Gr\"onwall estimate to~\eqref{e:k1} we infer that for $\mathbb{P}$-a.e.~$\omega \in \Om$ and every~$t \in [0, T]$
\begin{align}
\label{e:1000}
 | X^{1} (t) - X^{2} (t)| & \leq e^{L_{v} T } \int_{0}^{t} \big(\cW_{1} (\mu^1_{\tau}, \mu^2_{\tau}) + \cW_{1} (\nu^1_{\tau}, \nu^2_{\tau}) \big) \di \tau
 \\
 &
 \leq \int_{0}^{t} \Big( \mathbb{E} \Big( | X^{1} (\tau) - X^{2} (\tau)| \Big) + \cW_{1} (\nu^1_{\tau}, \nu^2_{\tau}) \Big) \di \tau\,. \nonumber
\end{align}
Integrating~\eqref{e:1000} over~$\Om$ and applying Gr\"onwall inequality we get~\eqref{eee}; substituting~\eqref{eee} into~\eqref{e:1000} we conclude for~\eqref{eee-1}.
\end{proof}


As a corollary of Proposition~\ref{p:12} we have the following.

\begin{corollary}
\label{p:continuity}
Let $v \colon [0, T] \times \mathcal{P}_{1}(\R^{d}) \times \mathcal{P}_{1}(\R^{d}) \to \R^{d}$ satisfy~\eqref{e:vLip}, 
let $\overline{X}_{0} \in L^{p} (\Om; \R^{d})$, and let $\nu^{k}, \nu \in C ([0, T]; \mathcal{P}_{1} (\R^{d}))$ be such that $\nu^{k} \to \nu$ in $C ([0, T]; \mathcal{P}_{1} (\R^{d}))$. Let moreover $X^{k}, X \in  \mathcal{M} (\Om;  C([0, T]; \R^{d}))$ be the corresponding solutions to~\eqref{e:SDE} with $\bmu^{k} = {\rm Law} (X^{k})$ and~$\bmu = {\rm Law} (X)$ and velocities~$\tilde{v}_{\mu^{k}_{ t}} = v_{\mu^{k}_{ t}, \nu^{k}_{ t}}$, $\tilde{v}_{\mu_{t}}= v_{\mu_{t}, \nu_{t}}$, respectively. Then, 
\begin{align}
\label{eee2}
& \lim_{k\to\infty}\, \sup_{t\in[0,T]} \mathbb{E} \big(  | X^{k} (t) - X(t) | \big) = 0\,,\\
& \label{eee3} \lim_{k \to \infty} \,  \cW_{1} (\bmu^{k} ,  \bmu ) = 0\,.
\end{align}
\end{corollary}

\begin{proof}
For fixed $k\in\N$,
rewrite~\eqref{eee} 
for 
$X^1=X^k$, $\nu^1=\nu^k$ and $X^2=X$, $\nu^2=\nu$.
Then, taking the limit as~$k \to \infty$ and relying on the uniform convergence of~$\nu^{k}$ to~$\nu$ in $\mathcal{P}(\R^{d})$, we get~\eqref{eee2}. As for the convergence in~\eqref{eee3}, we consider~\eqref{eee-1} for $X^{1} =X^{k}$ and~$X^{2} = X$; take the supremum over~$[0, T]$ and integrate over~$\Omega$ and use~\eqref{e:bmu} to obtain
\begin{equation}
\label{e:1001}
\cW_{1} (\bmu^{k}, \bmu) \leq \mathbb{E} \bigg( \sup_{t \in [0, T]} \, | X^{k} (t) - X(t)| \bigg) \leq C \int_{0}^{T} \cW_{1} (\nu^{k}_{ t}, \nu_{t}) \, \di t\,.
\end{equation}
Passing to the limit as~$k \to \infty$ in~\eqref{e:1001} we get~\eqref{eee3}.
\end{proof}

We now show that, for a given $\nu\in C ([0, T]; \mathcal{P}_{1} (\R^{d}))$ and under suitable assumptions on the initial datum~$\overline{\mu}_{0}$, the PDE~\eqref{e:system} has a unique solution, which is the one generated by the law of the unique stochastic process~$X$ that solves the SDE~\eqref{e:SDE}.

\begin{theorem}
\label{t:equivalence}
Let $v \colon [0, T] \times\mathcal{P}_{1}(\R^{d}) \times \mathcal{P}_{1}(\R^{d}) \to \R^{d}$ satisfy~\eqref{e:vLip}, 
let $\nu \in C ([0, T]; \mathcal{P}_{1} (\R^{d}))$, and let~$\tilde{v}$ be defined as in~\eqref{e:tildev}. 
Let $\overline{\mu}_{0} \in \mathcal{P}_{2} (\R^{d})$ be of the form $\overline{\mu}_{0} = \overline{\rho}_{0} \, \di x$ for~$\overline{\rho}_{0} \in L^{1} (\R^{d})$ with finite entropy.
For~$\overline{X}_{0} \in L^{2} (\Om; \R^{d})$ such that ${\rm Law} (\overline{X}_{0}) = \overline{\mu}_{0}$, let $X$ be the unique solution to~\eqref{e:SDE} with initial condition~$\overline{X}_{0}$. 
Then, the corresponding~$\mu_t$ is the unique solution to
\begin{equation}
\label{e:equivalence}
\left\{
\begin{array}{ll}
\partial_{t} \mu_{t} - \sigma \Delta \mu_{t} = - \dive (\tilde{v}_{\mu_{t}} (t) \mu_{t})\,.\\[1mm]
\mu_{0}= \overline{\mu}_{0}\,.
\end{array}
\right.
\end{equation} 
\end{theorem}

\begin{proof}
In view of Theorem~\ref{t:SDE} and It\^{o}'s formula, $\mu_t$ is a solution to~\eqref{e:equivalence}. For the readers' convenience, we recall the standard argument.
For $\varphi \in C^\infty_c (\R^d)$, we apply It\^{o}'s formula \cite[Theorem 4.2.1]{Oksendal} and we obtain that
\begin{equation*}
\de\varphi (X(t)) = \langle \nabla \varphi(X(t)), \de X(t) \rangle + \sigma \Delta\varphi(X(t)) \, \de t.
\end{equation*}
Using~\eqref{e:SDE}, we get
\begin{equation*}
\de \varphi (X(t)) =  \nabla \varphi(X(t)) \cdot  \tilde{v}_{\mu_t}(t,X(t))+ \big\langle \nabla\varphi(X(t)),\sqrt{2 \sigma} \, \de W(t) \big\rangle+ \sigma \Delta\varphi(X(t)) \, \de t.
\end{equation*}
Integrating the above expression in It\^{o}'s sense we have
\begin{equation*}
\begin{split}
\varphi (X(t)) =&\, \varphi (\overline{X}_0) + \int_0^t  \nabla \varphi(X(\tau)) \cdot \tilde{v}_{\mu_\tau}(\tau,X(\tau))  \, \de\tau\\
&\, + \sqrt{2 \sigma} \int_0^t \nabla \varphi (X(\tau)) \, \de W(\tau) + \int_0^t \sigma \Delta\varphi(X(\tau)) \, \de\tau.
\end{split}
\end{equation*}
Since $\mathbb{E} \Big( \int_0^t \nabla \varphi (X(\tau)) \,\de W(\tau) \Big) = 0$ by \cite[Theorem 3.2.1]{Oksendal}, taking the expected value we get
\begin{equation*}
\int_{\R^d} \varphi (x)\, \de\mu_t(x) = \int_{\R^d} \varphi (x) \, \de\overline{\mu}_0(x) + \int_0^t \int_{\R^d} \big[  \nabla \varphi(x) \cdot \tilde{v}_{\mu_\tau}(\tau,x) + \sigma \Delta\varphi(x) \big] \, \de\mu_\tau(x) \de \tau
\end{equation*}
as required.
Let $\hat{\mu} \in C([0, T]; \mathcal{P}_{1} (\R^{d}))$ be another solution to~\eqref{e:equivalence}. Setting $\hat{v}(t, x) \coloneqq \tilde{v}_{\hat{\mu}_{t}} (t, x)$, by~\eqref{e:old_v2} 
we have that
\begin{equation}
\label{e:sublinear}
| \hat{v}(t, x) | = |v_{\hat{\mu}_{t}, \nu_{t}} (t, x)| \leq M_{v} ( 1 + m_{1} (\nu_{t}) + m_{1} (\hat{\mu}_{t}) + |x|)\,.
\end{equation} 
By continuity of~$\hat{\mu}$ and~$\nu$, we have that $m_{1} (\nu_{t})$ and $m_{1} (\hat{\mu}_{t})$ are uniformly bounded in~$[0, T]$. Thus, we deduce from~\eqref{e:sublinear} that $\hat{v}$ has sublinear growth. Hence, $\hat{v}(t, x)$ satisfies the assumptions of~\cite[Lemma~3.6]{ASC_22} (see also~\cite[Theorem~3.3]{Bog_2007}), which implies that~$\hat{\mu}$ is the unique solution to 
\begin{equation}
\label{e:fattostandard}
\left\{
\begin{array}{ll}
\partial_{t} \mu_{t} - \sigma \Delta \mu_{t} = - \dive (\hat{v}(t) \mu_{t})\,,\\[1mm]
\mu_{0} = \overline{\mu}_{0}\,.
\end{array}
\right.
\end{equation} 

Let us consider~$\widehat{X}$ the unique pathwise solution (see \cite[Theorem~5.2.1]{Oksendal}) to the SDE
\begin{equation}
\label{e:contr}
\left\{
\begin{array}{ll}
\di Y(t) = \hat{v} (t, Y(t)) \,\di t + \sqrt{2\sigma} \, \di W(t)\,,\\[2mm]
Y(0) = \overline{X}_{0}\,
\end{array}\right.
\end{equation}
and let $\hat{\bmu}\coloneqq {\rm Law}(\widehat{X})$.
The PDE~\eqref{e:fattostandard} is the Fokker--Planck equation associated with~\eqref{e:contr}, hence it has $(\ev_t)_{\#} \hat{\bmu}$ as a solution. By uniqueness of the solution to~\eqref{e:fattostandard}, we get that $\hat{\mu}_t=(\ev_t)_{\#} \hat{\bmu}$, so that $\widehat{X}$ is actually a solution to~\eqref{e:SDE}. 
By Theorem~\ref{t:SDE}, we conclude that $X= \widehat{X}$, which in particular implies that $\mu_t=\hat{\mu}_t$ for all $t\in[0,T]$, as desired.
\end{proof}

We are now in a position to show existence of solutions to the minimum problem~\eqref{e:minJ}--\eqref{e:system}.

\begin{theorem}
\label{t:existence}
Let $L, R>0$ and $q \in (1, +\infty]$ be fixed, and let $\overline{\mu}_{0} \in \mathcal{P}_{2} (\R^{d})$ be such that $\overline{\mu}_{0} = {\rm Law} (\overline{X}_{0}) = \overline{\rho}_{0} \, \di x$ for some~$\overline{X}_{0} \in L^{2} (\Om; \R^{d})$ and~$\overline{\rho}_{0} \in L^{1} (\R^{d})$ with finite entropy. Moreover, assume that the cost functional $\J \colon C([0, T]; \mathcal{P}_{1}(\R^{d})) \times C([0, T]; \mathcal{P}_{1}(\R^{d})) \to \R \cup \{+\infty\}$ is lower-semicontinuous with respect to the convergence in $C([0, T]; \mathcal{P}_{1}(\R^{d}))$. Then, the minimum problem~\eqref{e:minJ}--\eqref{e:system} admits solution.
\end{theorem}

\begin{proof}
We apply the Direct Method. Let $(\mu^{k}, \nu^{k}) \in \mathcal{A}(q, R, L)$ be a minimizing sequence for~\eqref{e:minJ}--\eqref{e:system}. In view of Theorem~\ref{t:equivalence}, we can write $\bmu^{k} = {\rm Law}(X^{k})$ for every $k \in \mathbb{N}$, where $X^{k}$ solves the SDE
\begin{displaymath}
\left\{
\begin{array}{lll}
\di X (t) = v_{\mu_{t}, \nu^{k}_{ t}} (X(t)) \, \di t + \sqrt{2\sigma} \, \di W(t)\,,\\[1mm]
X(0) = \overline{X}_{0}\,,\\[1mm]
\bmu = {\rm Law} (X), \, \mu_{t} = ({\rm ev}_{t})_{\#} \bmu\,. 
\end{array}
\right.
\end{displaymath}
By assumption we have a uniform bound on $m_{q} (\nu^{k}_{t})$ for $k \in \mathbb{N}$ and $t \in [0, T]$, with~$q \in (1, +\infty]$, which implies that the measures $\nu_t^k$ all belong to a fixed compact subset of~$\cP_1(\R^d)$. 
Moreover, $\nu^{k}$ is equi-Lipschitz continuous. By Ascoli--Arzel\'a Theorem, there exists $\nu \in {\rm Lip}_{L} ([0, T]; \mathcal{P}_{1} (\R^{d}))$ such that, up to a subsequence, $\nu^{k} \to \nu$ in $C([0, T]; \mathcal{P}_{1} (\R^{d}))$. It follows from Corollary~\ref{p:continuity} that 
$\bmu^{k} \to \bmu$ in $\mathcal{P}_{1} \big( C([0, T]; \R^{d})\big)$. 
In particular, by~\eqref{e:pfLip}, $\mu^k\to\mu$ in $C([0,T];\cP_1(\R^d))$ and 
\begin{equation}
\label{e:semiJ}
\mathcal{J} (\mu, \nu) \leq \liminf_{k\to\infty} \, \mathcal{J} (\mu^{k}, \nu^{k})\,.
\end{equation}

To conclude that $(\mu, \nu)$ is a minimizer, we are left to show that~\eqref{e:system} is satisfied. To this purpose, we exploit Theorem~\ref{t:equivalence} again and consider the solutions~$X^{k}$ to the SDE~\eqref{e:SDE} with $\tilde{v}_{\mu^{k}_{ t}} (t, \cdot) = v_{\mu^{k}_{ t}, \nu^{k}_{ t}} (\cdot)$ and initial condition~$\overline{X}_{0}$. In particular, we have that $\bmu^{k} = {\rm Law}(X^{k})$. Let us further consider the unique solution $\widetilde{X}$ to
\begin{equation}
\label{e:widetilde}
\begin{cases}
\di \widetilde{X} (t) = v_{\widetilde{\mu}_{t}, \nu_{t}} (\widetilde{X}(t)) \, \di t + \sqrt{2\sigma} \, \di W(t)\,,\\
X(0) = \overline{X}_{0}\,, \\ 
\widetilde{\bmu} = {\rm Law} (\widetilde{X})\,, \ \widetilde{\mu}_{t} = ({\rm ev}_{t})_{\#} \widetilde{\bmu}\,.
\end{cases}
\end{equation}
Proposition~\ref{p:continuity} implies that $\bmu^{k} \to \widetilde{\bmu}$ in $\mathcal{P}_{1} \big( C([0, T]; \R^{d})\big)$, hence $\mu^k_t\to\widetilde{\mu}_t$ in $\cP_1(\R^d)$, uniformly in $[0,T]$. Thus, the curves~$\mu$ and~$\widetilde{\mu}$ coincide. By Theorem~\ref{t:equivalence} we have that $(\mu, \nu)$ solves~\eqref{e:system}. This concludes the proof of the theorem.
\end{proof}

\section{Optimal control problem for a two-population dynamics}
\label{s:2pop}

In this section, we present a prototypical example of a mean-field optimal control problem of the form~\eqref{e:minJ}--\eqref{e:system} for the case of agents divided into two populations, leaders and followers. 
The population of followers is driven by a nonlinear Fokker--Planck equation taking into account noise effects on the behavior of the agents, whereas the controlled population (that of the leaders) has a deterministic behavior driven by a non-local continuity equation, in which the control vector field appears as an additional drift term.
The evolution of these two populations is described by a pair $(\mu,\nu)\in \big(C([0,T];\cP_1(\R^d))\big)^2$, where~$\mu$ is the density of the uncontrolled population subject to the additive noise and~$\nu$ is the density of the controlled one.
This is a more refined variant of Problem 1, where the population~$\mu$ is indirectly controlled through the action of the policy maker on the selected population~$\nu$; this action is encoded by adding a suitable drift term in the evolution equation for~$\nu$.

We fix
~$\overline{\mu}_{0}\in\cP_2(\R^d)$ of the form $\overline{\rho}_{0} \, \di x$ for a suitable~$\overline{\rho}_{0} \in L^{1} (\R^{d})$ with finite entropy. We also fix $q\in(1,+\infty]$ and $\overline{\nu}_{0} \in \mathcal{P}_{q} (\R^{d})$. 
Finally, let $K \subseteq \R^{d }$ be a compact convex set with $0 \in K$; for given $\Lambda, \Delta>0$, we define
\begin{displaymath}
\mathcal{G} := \{ g \in C(\mathcal{P}_{1}(\R^{d})) : \text{$g$ is $\Lambda$-Lipschitz and $\Delta$-bounded}\}\,.
\end{displaymath}

In our optimal control problem, the curve~$\nu$ considered in~\eqref{e:system} will be replaced by the solution of a controlled continuity equation. 
We assume that the control takes a multiplicative form; indeed, given $f \in L^{\infty} ([0, T] \times \R^{d}; K)$ and~$g \in \G$, the state equation for the pair $(\mu, \nu)$ reads
\begin{equation}
\label{e:mean-field}
\begin{cases}
\partial_{t} \mu_{t} - \sigma \Delta \mu_{t} = - \dive (v_{\mu_{t}, \nu_{t}} \mu_{t})\,,\\
\partial_{t} \nu_{t} + \dive\big(  (w_{\mu_{t}, \nu_{t}} + f (t, \cdot) g(\mu_{t})  ) \nu_{t} \big) = 0\,,\\
\mu_{0} = \overline{\mu}_{0}\,, \ \nu_{0} = \overline{\nu}_{0}\,,
\end{cases}
\end{equation}
for velocity fields $v,w$ satisfying \eqref{e:vLip}.
For given initial data $(\overline{\mu}_0,\overline{\nu}_0)\in\cP_2(\R^d)\times\cP_q(\R^d)$, we define the set
\begin{align}
\label{e:set}
\!\!\!\!\!\!\mathcal{S}(\overline{\mu}_{0}, \overline{\nu}_{0}) \coloneqq \bigg\{ & (\mu, \nu, \zeta, g)  \in \big(C([0, T];  \mathcal{P}_{1} (\R^{d}))\big)^2 
\times \mathcal{M}_{b} ([0, T]\times \R^{d}; \R^{d}) \times  \G:  \\
&
\text{$\zeta \ll g(\mu_t)\nu$ and $(\mu, \nu)$ solves~\eqref{e:mean-field} with $f = \frac{\di \zeta}{\di (g(\mu_{t}) \nu)} \in L^{1}_{\nu} ([0, T]\times \R^{d}; K)$} \bigg\}, \nonumber
\end{align}
where $L^{1}_{\nu} ([0, T]\times \R^{d}; K)$ is the space of integrable functions with respect to the measure~$\nu$.
For $(\mu, \nu, \zeta, g) \in \mathcal{S} (\overline{\mu}_{0}, \overline{\nu}_{0})$, we define the cost functional
\begin{equation}
\label{e:cost-mean}
E(\mu, \nu, \zeta, g)   \coloneqq \int_{0}^{T} \mathcal{L} (\mu_{t}, \nu_{t}) \, \di t +\Phi(\zeta, g) \,,
\end{equation}
where
\begin{equation}
\label{e:Phi}
\Phi(\zeta, g)  \coloneqq \min\bigg\{ \int_{0}^{T} \!\!\int_{\R^{d}}\!\! \phi (f(t, y) , g(\mu_{t}) ) \, \di \nu_{t} (y) \, \di t:  
 f \in L^{1}_{\nu}([0, T]\times \R^{d}; K), \, \zeta = f g(\mu_{t}) \nu \bigg\}\,.
\end{equation}
In~\eqref{e:cost-mean} and~\eqref{e:Phi} we consider a uniformly continuous Lagrangian cost 
\begin{displaymath}
\mathcal{L} \colon C([0, T];  \mathcal{P}_{1} (\R^{d})) \times C([0, T]; \mathcal{P}_{1} (\R^{d})) \times \mathcal{M}_{b} ([0, T]\times \R^{d}; \R^{d}) \times  \G \to [0, +\infty)
\end{displaymath}
and a control cost $\phi \colon \R^{d} \times \R \to [0, +\infty)$ such that
\begin{itemize}
\item[$(\phi1)$] $\phi$ is continuous;
\item[$(\phi2)$] $\phi (\cdot, \xi)$  is convex and has superlinear growth uniformly with respect to $\xi \in \R$;
\item[$(\phi3)$] $\phi(0, \xi) = 0$ for every $\xi \in \R$.
\end{itemize}
The optimal control problem reads as follows:
\begin{mathframed}
\textbf{Problem 2}
\begin{equation}
\label{e:min-mean}
\min \big\{ E(\mu, \nu, \zeta, g): \, (\mu, \nu, \zeta, g) \in \mathcal{S} (\overline{\mu}_{0}, \overline{\nu}_{0}) \}\,.
\end{equation}\vspace{1mm}
\end{mathframed}

\begin{remark}\label{r:Snonvuoto}
In stating the optimal control problem above, it is understood that the minimum is equal to~$+\infty$ if $\mathcal{S} (\overline{\mu}_{0}, \overline{\nu}_{0})=\emptyset$. 
Actually, the results on the particle approximation that we are going to establish in Sections~\ref{s:sez4} and~\ref{s:limit} entail, as a byproduct, that $\mathcal{S} (\overline{\mu}_{0}, \overline{\nu}_{0}) \neq \emptyset$ for every pair $(\overline{\mu}_{0}, \overline{\nu}_{0})$ considered here.
\end{remark}
In the following remark, we discuss the relation of Problem 2 with some previous contribution on the subject.
\begin{remark}\label{r:PB2}
We point out that Problem 2 is the natural generalization of the mean-field optimal control problem analyzed in~\cite{ASC_22} where, however, the leaders' population was constrained to be discrete with a fixed number~$m$ of individuals.
The multiplicative structure of the control can, in particular, also account for a purely offline control problem (when~$g$ is constant) or for a purely feedback control problem (when~$f$ is constant). 
Notice that, if~$\sigma=0$ and~$g$ is constant, we retrieve the control Problem 2 studied in~\cite{bongini2016optimal}.
Besides the addition of noise terms, another relevant improvement with respect to \cite[Problem~2]{bongini2016optimal} is that we can allow for very low regularity of the vector field~$f$ appearing in~\eqref{e:mean-field}.
\end{remark}

The justification of the definition~\eqref{e:Phi} of~$\Phi$ as a minimum and not as an infimum is postponed to Lemma~\ref{l:lsc-cost}, which also implies the lower semicontinuity of the cost functional~$E$. Instead, we start by providing some estimates on the moments~$m_{2} (\mu_t)$ and~$m_q(\nu_t)$ for $t\in[0,T]$, and on the modulus of continuity of the curve $t \mapsto (\mu_{t}, \nu_{t})$ solution to~\eqref{e:mean-field}.

\begin{lemma}
\label{l:spt-mp}
Let $(\overline{\mu}_0,\overline{\nu}_0)\in\cP_2(\R^d)\times\cP_q(\R^d)$ be such that $\overline{\mu}_{0}= \overline{\rho}_{0}\, \di x={\rm Law} (\overline{X}_{0})$ for some $\overline{\rho}_{0} \in L^{1} (\R^{d})$ with finite entropy and some $\overline{X}_0\in L^2(\Omega;\R^d)$.
Then there exists $0<r = r(T, \Delta, \Lambda, \overline{\mu}_{0}, \overline{\nu}_{0}, K)$ such that for every $(\mu, \nu, \zeta, g) \in  \mathcal{S}(\overline{\mu}_{0}, \overline{\nu}_{0})$ there holds
\begin{equation}
\label{e:mp} 
m_{2} (\mu_t)+m_1(\nu_t) \leq r \qquad \text{for every $t \in [0, T]$}.
\end{equation}
Moreover, there exists $0<L = L(T, \Delta, \Lambda, \overline{\mu}_{0}, \overline{\nu}_{0}, K)$ such that the curve $t \mapsto \mu_{t}$ belongs to $C^{0,1/4}([0,T];\cP_1(\R^d))$ with H\"older constant~$L$ and the curve $t \mapsto \nu_{t}$ belongs to ${\rm Lip}_{L} ([0, T]; \mathcal{P}_{1}(\R^{d}))$. 
\end{lemma}

\begin{proof}
We denote by~$C$ a generic positive constant depending on~$T$, $\Delta$, $\Lambda$, $\overline{\mu}_{0}$, $\overline{\nu}_{0}$, and~$K$, and which may vary from line to line.

Since $\mu, \nu \in C([0, T];  \mathcal{P}_{1} (\R^{d})) $, there exists~$\widetilde{C}>0$ such that $m_{1} (\mu_{t}), m_{1} (\nu_{t}) \leq \widetilde{C}$ for every $t \in [0, T]$. By definition of~$K$ and since~$w$ satisfies~\eqref{e:old_v2}, we have that
\begin{align}
\label{e:400}
\int_{0}^{T} \int_{\R^{d}} | w_{\mu_{t}, \nu_{t}} (y) + f(t, y)g(\mu_{t}) | \, \di \nu_{t} \, \di t & \leq  \int_{0}^{T} \int_{\R^{d}}\big(  M_{w} \big( 1 + m_{1} (\mu_{t}) + m_{1} (\nu_{t}) \big) + C \big) \, \di \nu_{t} \, \di t 
\\
&
\leq M_{w} (T + 2\widetilde{C})  + CT\,. \nonumber
\end{align}
Hence, we are in a position to apply the superposition principle (see, e.g.,~\cite[Theorem~5.2]{AmbForMorSav18} and~\cite[Section~8.2]{AmbGigSav08}) to the curve~$t\mapsto \nu_{t}$ solving the continuity equation
\begin{displaymath}
\begin{cases}
\partial_{t} \nu_{t} + \dive (b_{t} \, \nu_{t}) = 0 \,, \\ 
\nu_{0} =\overline{\nu}_{0}\,,
\end{cases}
\end{displaymath}
with velocity field $b_{t} (y) \coloneqq w_{\mu_{t}, \nu_{t}} (y) + f(t, y) g(\mu_{t}) $. In particular, there exists $\eeta \in \mathcal{P} (C([0, T]; \R^{d}))$ supported on solutions to the Cauchy problem
\begin{equation}
\label{e:superp}
\left\{
\begin{array}{ll}
\dot{\gamma} (t) = b_{t} (\gamma(t)) = w_{\mu_{t}, \nu_{t}} (\gamma(t)) +  f(t, \gamma(t)) g(\mu_{t}) \,,\\[1mm]
\gamma(0) \in {\rm spt} (\overline{\nu}_{0})
\end{array}
\right.
\end{equation}
and such that $\nu_{t} = ({\rm ev}_{t})_{\#} \eeta$. For every $\gamma \in C([0, T]; \mathbb{R}^{d})$ solution to~\eqref{e:superp}, we may estimate, as in \eqref{e:400},
\begin{align*}
| \gamma(t)| \leq |\gamma(0)| + \int_{0}^{t} \big( M_{w} ( 1 + m_{1} (\mu_{\tau}) + m_{1} (\nu_{\tau}) + | \gamma(\tau)| ) + C \big) \, \di \tau\,. 
\end{align*}
By Gr\"onwall inequality, we infer that for $\eeta$-a.e.~$\gamma$ there holds
\begin{align}
\label{e:405}
| \gamma(t)| \leq C\bigg(  |\gamma(0)| + \int_{0}^{t} \big(  1 + m_{1} (\mu_{\tau}) + m_{1} (\nu_{\tau})  \big)\, \di \tau \bigg) \,. 
\end{align}
Integrating~\eqref{e:405} over~$C([0, T]; \R^{d})$ with respect to~$\eeta$ we get that
\begin{align}
\label{e:401}
m_{1} (\nu_{t}) \leq   C\bigg(  m_{1} (\overline{\nu}_{0}) +  \int_{0}^{t} \big( 1 + m_{1} (\mu_{\tau}) + m_{1} (\nu_{\tau})\big) \, \di \tau \bigg) \,.
\end{align}
Applying Gr\"onwall inequality again, we deduce from~\eqref{e:401} that
\begin{align}
\label{e:402}
m_{1} (\nu_{t}) \leq C 
\bigg( 1 + m_{1} (\overline{\nu}_{0}) + \int_{0}^{t} m_{1} (\mu_{\tau})\, \di \tau \bigg) \,.
\end{align}

In order to estimate~$m_{2}(\mu_{t})$, we make use of Theorem~\ref{t:SDE} with $p=2$  and Theorem~\ref{t:equivalence}, which yield, together with~\eqref{e:402}, that for every $t \in [0, T]$
\begin{align*}
m_{2} (\mu_{t}) & \leq C\bigg( 1 + m_{2} (\overline{\mu}_{0}) +  \int_{0}^{t} m_{1}^{2} (\nu_{\tau}) \, \di \tau +  4\mathbb{E} (|W(T)|^{2})\bigg)
\\
&
\leq C\bigg( 1 + m_{2} (\overline{\mu}_{0}) +m^{2}_{1} (\overline{\nu}_{0}) +  \int_{0}^{t} m_{1}^{2} (\mu_{\tau}) \, \di \tau + 4\mathbb{E} (|W(T)|^{2})\bigg) \nonumber
\\
&
\leq C\bigg( 1 + m_{2} (\overline{\mu}_{0}) +m^{2}_{1} (\overline{\nu}_{0}) +  \int_{0}^{t} m_{2} (\mu_{\tau}) \, \di \tau +  4\mathbb{E} (|W(T)|^{2})\bigg). \nonumber
\end{align*}  
Applying Gr\"onwall inequality we deduce that
\begin{align}
\label{e:403}
m_{2} (\mu_{t}) \leq C \Big( 1 + m_{2} (\overline{\mu}_{0}) +m^{2}_{1} (\overline{\nu}_{0})  +  4\mathbb{E} (|W(T)|^{2}) \Big) \qquad \text{for $t \in [0, T]$}.
\end{align}
Combining~\eqref{e:402} and~\eqref{e:403}, by H\"{o}lder inequality we obtain for $t \in [0, T]$
\begin{align}
\label{e:404}
m_{1} (\nu_{t}) \leq C \Big( 1 +  m_{1} (\overline{\nu}_{0})  + \big( m_{2} (\overline{\mu}_{0}) + 4\mathbb{E} (|W(T)|^{2}) \big)^{\frac{1}{2}} \Big).
\end{align}
Putting~\eqref{e:403} and~\eqref{e:404} together yields~\eqref{e:mp}.


By~\eqref{e:holder} we have that $t \mapsto \mu_{t}$ is H\"older continuous with exponent~$1/4$ and constant~$L>0$ only depending on~$T$,~$r$,~$M_{v}$, and~$M_{w}$. It remains to prove that, up to a redefinition of~$L$, the map $t \mapsto \nu_{t}$ is $L$-Lipschitz continuous. To this aim, we estimate for~$\eeta$-a.e.~$\gamma$ and for every $s<t \in [0, T]$
\begin{align}
\label{e:1102}
| \gamma(t) - \gamma(s)| & \leq \int_{s}^{t} \big( | w_{\mu_{\tau}, \nu_{\tau}} (\gamma(\tau)) | + | g(\mu_{\tau}) f(\tau, \gamma(\tau))| \big) \, \di \tau
\\
&
\leq \int_{s}^{t} \big( M_{w} ( 1 + m_{1} (\mu_{\tau}) + m_{1} (\nu_{\tau}) + | \gamma(\tau)| ) + C \big) \, \di \tau\,. \nonumber
\end{align}
In view of the definition of the Wasserstein distance~$\cW_{1}$, inequality~\eqref{e:1102} implies that
\begin{align*}
\cW_{1} (\nu_{t}, \nu_{s}) & \leq \int_{C([0, T]; \R^{d})} | \gamma(t) - \gamma(s)| \, \di \eeta(\gamma) 
\\
&
\leq C\int_{s}^{t} \int_{C([0, T]; \R^{d})} \big( 1 + m_{1}(\mu_{\tau}) +  m_{1} (\nu_{\tau}) + | \gamma(\tau)| \big) \, \di \eeta(\gamma) \, \di \tau
\\
&
\leq C\int_{s}^{t} \big( 1 + m_{1}(\mu_{\tau}) +  m_{1} (\nu_{\tau}) + | \gamma(\tau)| \big) \, \di \tau\,.
\end{align*} 
Then, we infer from~\eqref{e:mp} that $t\mapsto\nu_t$ is Lipschitz continuous.
\end{proof}

Estimate~\eqref{e:mp} can be improved into an estimate on the $q$-th moment of~$\nu_t$.
\begin{lemma}\label{l:lemma3.2}
Let $(\overline{\mu}_0,\overline{\nu}_0)\in\cP_2(\R^d)\times\cP_q(\R^d)$ be such that $\overline{\mu}_{0}= \overline{\rho}_{0}\, \di x={\rm Law} (\overline{X}_{0})$ for some $\overline{\rho}_{0} \in L^{1} (\R^{d})$ with finite entropy and some $\overline{X}_0\in L^2(\Omega;\R^d)$.
Then there exists $0<R = R(T, \Delta, \Lambda, \overline{\mu}_{0}, \overline{\nu}_{0}, K)$ such that for every $(\mu, \nu, \zeta, g) \in  \mathcal{S}(\overline{\mu}_{0}, \overline{\nu}_{0})$ there holds
\begin{equation}
\label{e:mpq} 
m_{2} (\mu_t)+m_q(\nu_t) \leq R \qquad \text{for every $t \in [0, T]$}.
\end{equation}
\end{lemma}
\begin{proof}
It is enough to notice that by~\eqref{e:old_v2},~\eqref{e:411}, and H\"{o}lder inequality, we can estimate
\begin{equation*}
|w_{\mu_t,\mu_t}(y)|\leq M_{w} (1 + |y| + \sqrt{m_{2} (\mu_t)} + m_{1} (\nu_t))\leq M_w\big(1+\sqrt{r} \big) + M_w(|y| + m_1(\nu_t)).
\end{equation*}
By applying \cite[Proposition~5.3]{Flos} to the field $w_{\mu_t,\nu_t}(y)$ with $A=M_w(1+\sqrt{r})$, $B=M_w$, and $\theta(|y|)=|y|^q$, we infer~\eqref{e:mpq}.
\end{proof}

In order to establish existence of solutions to~\eqref{e:min-mean},  we now discuss the lower semicontinuity of the control cost functional~$\Phi$ along converging sequences in~$\mathcal{S} (\overline{\mu}_{0}, \overline{\nu}_{0})$. The result follows by a non-autonomous extension of~\cite[Lemma~9.4.3]{AmbGigSav08}. We remark that the following lemma justifies the definition of~$\Phi$ in~\eqref{e:Phi}. 

 
\begin{lemma}
\label{l:lsc-cost}
Let $(\overline{\mu}_0,\overline{\nu}_0)\in\cP_2(\R^d)\times\cP_q(\R^d)$ be such that $\overline{\mu}_{0}= \overline{\rho}_{0}\, \di x={\rm Law} (\overline{X}_{0})$ for some $\overline{\rho}_{0} \in L^{1} (\R^{d})$ with finite entropy and some $\overline{X}_0\in L^2(\Omega;\R^d)$.
Let $\overline{\nu}^{k}_{0} \in \cP_q(\R^d)$  be such that $\overline{\nu}^{k}_{0}  \to \overline{\nu}_{0}$ narrow in~$\mathcal{P}(\R^{d})$ as $k \to \infty$ and $\sup_{k\in\N} m_q(\overline{\nu}_0^k)<+\infty$. Moreover, let $(\mu^{k}, \nu^{k}, \zeta^{k}, g^{k} ) \in \mathcal{S}(\overline{\mu}_{0}, \overline{\nu}^{k}_{0})$ and $f^{k} \in L^{1}_{\nu^{k}} ([0, T]\times \R^{d}; K)$ satisfy $\zeta^{k} = f^{k} g^{k} (\mu^{k}_{ t}) \nu^{k}$ and
\begin{equation}
\label{e:505}
\sup_{k\in\N}  \int_{0}^{T} \int_{\R^{d}} \phi( f^{k}(t, y), g^{k}(\mu^{k}_{ t}) ) \, \di \nu^{k}_{t} (y) \, \di t < +\infty\,.
\end{equation}
Then there exists $(\mu, \nu, \zeta, g) \in \mathcal{S}(\overline{\mu}_{0}, \overline{\nu}_{0})$ such that, up to a (not relabeled) subsequence, $\mu^{k}\to \mu$ and $ \nu^{k}\to \nu$ in $C([0, T]; \mathcal{P}_{1} (\R^{d}))$, $\zeta^{k} \to \zeta$ weakly$^*$ in~$\mathcal{M}_{b} ([0, T]\times \R^{d}; \R^{d})$, and~$g^{k} \to g$ locally uniformly in $C(\mathcal{P}_{1} (\R^{d}))$. Furthermore, there exists~$f \in L^{1}_{\nu} ([0, T]\times \R^{d}; K)$ such that $\zeta = f g(\mu_{t}) \nu$ and
\begin{equation}
\label{e:lsc-cost}
\int_{0}^{T} \int_{\R^{d}} \phi( f(t, y), g(\mu_{t}) ) \, \di \nu_{t} (y) \, \di t  \leq \liminf_{k\to\infty}  \int_{0}^{T} \int_{\R^{d}} \phi( f^{k}(t, y), g^{k}(\mu^{k}_{ t}) ) \, \di \nu^{k}_{t} (y) \, \di t\,.
\end{equation}
\end{lemma}
\begin{proof}
Up to a subsequence (which we are not going to relabel), we may assume that the $\liminf$ in~\eqref{e:lsc-cost} is a limit. 
By Lemmas~\ref{l:spt-mp} and~\ref{l:lemma3.2}, there exist $L,R>0$ such that, for every $k \in \mathbb{N}$, $\nu^{k} \in {\rm Lip}_{L} ([0, T]; \mathcal{P}_{1} (\R^{d}))$, $\mu^{k} \in C^{0, 1/4}([0, T]; \mathcal{P}_{1} (\R^{d}))$ with H\"older constant~$L$, and 
$$\sup_{\substack {k\in\N\\ t\in[0,T]}} \big(m_2(\mu^k_t)+m_q(\nu^k_t)\big)\leq R.$$
In particular,~$\mu^{k}$ and~$\nu^{k}$ are sequences of equi-continuous curves with values in a compact subset of~$\mathcal{P}_{1} (\R^{d})$. 
By Ascoli--Arzel\'a Theorem, there exist $\mu, \nu \in C([0, T]; \mathcal{P}_{1} (\R^{d}))$ such that, up to a subsequence, $\mu^{k} \to \mu$ and $\nu^{k} \to \nu$ in $C([0, T]; \mathcal{P}_{1} (\R^{d}))$. Similarly, the sequence~$g^{k}$ is equi-Lipschitz and equi-bounded, and therefore admits a  subsequence (not relabeled) converging on compact subsets of $\cP_1(\R^d)$ to a limit~$g \in \G$.
Finally, $\zeta^{k}$ is bounded in~$\mathcal{M}_{b} ([0, T]\times \R^{d}; \R^{d})$. 
Thus, there exists $\zeta \in \mathcal{M}_{b} ([0, T]\times \R^{d}; \R^{d})$ such that, up to a subsequence,~$\zeta^{k} \to \zeta$ weakly$^*$ in $\mathcal{M}_{b} ([0, T]\times \R^{d}; \R^{d})$.

Without loss of generality, we may assume that $f^{k}(t,\cdot) = 0$ whenever $g^{k} (\mu^{k}_t) = 0$ for all $t\in[0,T]$. We recall that~$f^{k}$ takes values in the compact set~$K$ and that $g^{k}, g$ are $\Delta$-bounded and $\Lambda$-Lipschitz. Let us further set $\theta^{k} \coloneqq f^{k} \nu^{k}$. It is clear that $\theta^{k} \ll  \nu^{k}$ and that $f^{k} = \frac{\di \theta^{k}}{\di \nu^{k}}$ $\nu^{k}$-a.e.~in~$[0, T] \times \R^{d}$. Hence, we may write
 \begin{equation}
 \label{e:rewrite-cost}
 \int_{0}^{T} \int_{\R^{d}} \phi(f^{k}(t, y) , g^{k}(\mu^{k}_{t}) ) \, \di \nu^{k}_{ t} (y)\, \di t = \int_{[0, T] \times \R^{d}} \phi \bigg( \frac{\di \theta^{k}}{\di  \nu^{k}} (t, y) , g^{k} (\mu^{k}_{ t})\bigg) \, \di \nu^{k}(t, y) \,. 
 \end{equation} 
By the bounds on~$f^{k}$ and~$\nu^{k}$, we have that $\theta^{k}$ weakly$^*$ converges, up to a not relabeled subsequence, to some measure $\theta \in \mathcal{M}_{b} ([0, T]\times \R^{d}; \R^{d})$.
 
 Let us denote by $\omega_{\phi}$ a modulus of continuity of~$\phi$ on the compact set $K \times [-\Delta,\Delta]$. For $J \in \mathbb{N}$ and $j = 0, \ldots, J$ let us set $t_{j} \coloneqq \frac{jT}{J}$. Then, by~\eqref{e:rewrite-cost} and by the $\Lambda$-Lipschitz continuity of~$g^{k}$ we have that
  \begin{align}
 \label{e:cost-1}
 \int_{0}^{T} \int_{\R^{d}}&  \phi(f^{k}(t, y) , g^{k}(\mu^{k}_{ t}) ) \, \di \nu^{k}_{ t} (y)\, \di t 
 \\
 &
 \geq \sum_{j=1}^{J} \int_{[t_{j-1}, t_{j}]\times \R^{d}} \phi \bigg( \frac{\di \theta^{k}}{\di  \nu^{k}} (t, y) , g^{k} (\mu^{k}_{ t_{j-1}})\bigg) \, \di \nu^{k}(t, y) 
 - \frac{T}{J}\sum_{j=1}^{J} \omega_{\phi} \big( \Lambda \cW_{1} (\mu^{k}_{t_{j-1}}, \mu^{k}_{ t_{j}})\big) \,.  \nonumber
 \end{align} 
Arguing as in~\eqref{e:cost-1} we continue 
with
\begin{align}
\label{e:501.1}
 \int_{0}^{T} \int_{\R^{d}}&  \phi(f^{k}(t, y) , g^{k}(\mu^{k}_{ t}) )  \, \di \nu^{k}_{ t} (y)\, \di t 
 \\
 &
 \geq \sum_{j=1}^{J} \int_{[t_{j-1}, t_{j}]\times \R^{d}} \phi \bigg( \frac{\di \theta^{k}}{\di  \nu^{k}} (t, y) , g (\mu_{t_{j-1}})\bigg) \, \di \nu^{k}(t, y) -  \frac{T}{J}\sum_{j=1}^{J} \omega_{\phi} \big( \Lambda \cW_{1} (\mu^{k}_{ t_{j-1}}, \mu^{k}_{ t_{j}})\big) \nonumber
 \\
 &
\qquad   - \frac{T}{J} \sum_{j=1}^{J} \omega_{\phi} \big( | g^{k} (\mu^{k}_{ t_{j-1}}) - g (\mu_{t_{j-1}}) | \big)\,.  \nonumber
\end{align}
Let us fix~$\varepsilon>0$. Since $\mu^{k}$ and $\mu$ are equi-uniformly continuous in~$[0, T]$, there exists~$\overline{J} \in \mathbb{N}$ such that for every $J \geq \overline{J}$, every $j = 1, \ldots, J$, and every $k\in\N$
\begin{equation}
\label{e:500}
\omega_{\phi} \big( \Lambda \cW_{1} (\mu_{ t_{j-1}}, \mu_{ t_{j}}) \big) , \, \omega_{\phi} \big( \Lambda \cW_{1} (\mu^{k}_{ t_{j-1}}, \mu^{k}_{ t_{j}}) \big) \leq \frac{\varepsilon}{T} . 
\end{equation}
Combining~\eqref{e:501.1} and~\eqref{e:500} we deduce that for $J \geq \overline{J}$ and $k \in\mathbb{N}$
\begin{align}
\label{e:501}
&  \int_{0}^{T} \int_{\R^{d}}  \phi(f^{k}(t, y) , g^{k}(\mu^{k}_{ t}) )   \, \di \nu^{k}_{ t} (y)\, \di t 
 \\
 &
 \geq \sum_{j=1}^{J} \int_{[t_{j-1}, t_{j}] \times \R^{d}} \phi \bigg( \frac{\di \theta^{k}}{\di  \nu^{k}} (t, y) , g (\mu_{ t_{j-1}})\bigg) \, \di \nu^{k}(t, y)     - \frac{T}{J} \sum_{j=1}^{J} \omega_{\phi} \big( | g^{k} (\mu^{k}_{ t_{j-1}}) - g (\mu_{t_{j-1}}) | \big)- \varepsilon  \nonumber
\end{align}

Thanks to assumptions~$(\phi1)$ and~$(\phi2)$, to each term in the sum on the right-hand side of~\eqref{e:501} we can apply~\cite[Lemma~9.4.3]{AmbGigSav08} (see also~\cite[Proposition~5]{AMOP_2022}). Thus, for $j =1, \ldots , J$ we infer that the limit measure $\theta$ is such that its restriction~$\theta_{j}\coloneqq \theta \res ([t_{j-1}, t_{j}] \times \R^{d})$ satisfies $\theta_{j} \ll \nu \res ([t_{j-1}, t_{j}] \times \R^{d})$ and
\begin{align}
\label{e:502}
 \int_{[t_{j-1}, t_{j}] \times \R^{d}} \phi & \bigg( \frac{\di \theta_{j}}{\di  \nu} (t, y) , g (\mu_{t_{j-1}})\bigg) \, \di \nu (t, y) 
 \\
 &
 \leq \liminf_{k\to\infty}  \int_{[t_{j-1}, t_{j}] \times \R^{d}} \phi \bigg( \frac{\di \theta^{k}}{\di \nu^{k}} (t, y) , g (\mu_{t_{j-1}})\bigg) \, \di \nu^{k}(t, y)\,. \nonumber
\end{align}
This implies that $\theta \ll \nu$ in $[0, T] \times \R^{d}$. Moreover, since $K$ is compact and convex, setting $f \coloneqq \frac{\di \theta}{\di \nu}$, we have that $f \in K$ for $\nu$-a.e.~$(t, y) \in [0, T]\times \R^{d}$. Hence, summing up~\eqref{e:502} over~$j=1, \ldots, J$ and recalling~\eqref{e:501} we have that
\begin{align}
\label{e:503}
\sum_{j=1}^{J} \int_{[t_{j-1},t_j]\times \R^{d}} \phi & ( f(t, y) , g (\mu_{t_{j-1}})) \, \di \nu (t, y)
\\
&
 \leq \liminf_{k\to\infty}  \int_{0}^{T} \int_{\R^{d}} \phi(f^{k}(t, y) , g^{k}(\mu^{k}_{ t}) )   \, \di \nu^{k}_{ t} (y)\, \di t + \varepsilon\,. \nonumber
\end{align}
Repeating the argument of~\eqref{e:cost-1} we deduce from~\eqref{e:503} that
\begin{align*}
\int_{0}^{T}\int_{ \R^{d}} \phi ( f(t, y) , g (\mu_{t})) \, \di \nu_t(y)\,\de t \leq \liminf_{k\to\infty}  \int_{0}^{T} \int_{\R^{d}} \phi(f^{k}(t, y) , g^{k}(\mu^{k}_{ t}) )   \, \di \nu^{k}_{ t} (y)\, \di t + 2\varepsilon\,.
\end{align*}
Hence, we conclude for~\eqref{e:lsc-cost} by letting~$\varepsilon \to 0$. Since $\zeta^{k} = g^{k}(\mu^{k}_{ t}) \theta^{k}$, $g^{k} (\mu^{k}_{ t}) \to g (\mu_{t})$ in $C([0, T])$, and $\theta^{k}$ weakly$^*$ converges to~$\theta$ in $\mathcal{M}_{b} ([0, T]\times \R^{d}; \R^{d})$, we get that $\zeta = g(\mu_{t}) \theta = f g(\mu_{t}) \nu \ll \nu$. 

It remains to show that $(\mu, \nu, \zeta, g) \in \mathcal{S}(\overline{\mu}_{0}, \overline{\nu}_{0})$. In particular, we only have to prove that~$(\mu, \nu)$ solves system~\eqref{e:mean-field}. 
Since $(\mu^k,\nu^k,\zeta^k,g^k)\in\mathcal{S}(\overline{\mu}_0,\overline{\nu}_0^k)$, in view of the convergences proved above and of Corollary~\ref{p:continuity}, we obtain that $\mu$ is the unique solution to
\begin{equation}
\label{e:414.1}
\begin{cases}
\partial_{t} \mu_{t} - \sigma \Delta \mu_{t} = - \dive (v_{\mu_{t}, \nu_{t}} \mu_{t}) \,, \\
\mu_{0} = \overline{\mu}_{0}\,.
\end{cases}
\end{equation}
For every test function~$\varphi \in C^{1}_{c} ([0, T]; \R^{d})$ we have that
\begin{align}
\label{e:412}
\int_{\R^{d}} \varphi(T, x) \, \di \nu^{k}_{ T} (x) & - \int_{\R^{d}} \varphi(0, x) \, \di \overline{\nu}_{0} (x) = \int_{0}^{T} \int_{\R^{d}} \partial_{t} \varphi (t, x) \, \di \nu^{k}_{t} (x) \, \di t 
\\
&
+ \int_{0}^{T} \int_{\R^{d}} \nabla \varphi(t, x) \cdot ( w_{\mu^{k}_{ t}, \nu^{k}_{ t}} (t, x) + f^{k} (t, x)g^{k} (\mu^{k}_{ t}))\, \di \nu^{k}_{ t} (x) \, \di t\,. \nonumber
\end{align}
By~\eqref{e:vLip} and by the uniform convergence of~$\nu^{k}$ and~$\mu^{k}$ we deduce that
\begin{align}
\label{e:413}
\lim_{k\to\infty} \, \bigg| \int_{0}^{T}&  \int_{\R^{d}} \nabla \varphi(t, x) \cdot \big( w_{\mu^{k}_{ t}, \nu^{k}_{ t}} (x) - w_{\mu_{t}, \nu_{t}} ( x) )  \big)\, \di \nu^{k}_{ t} (x) \, \di t \bigg|
\\
&
\leq \lim_{k\to\infty}\,  \| \nabla \varphi\|_{\infty} \int_{0}^{T} \big( \cW_{1} (\mu_{t}, \mu^{k}_{ t}) + \cW_{1} (\nu_{t}, \nu^{k}_{ t}) \big)\,\de t =0\,. \nonumber
\end{align} 
Combining~\eqref{e:412},~\eqref{e:413}, the uniform convergence of~$\mu^{k}$ and~$\nu^{k}$ to~$\mu$ and~$\nu$, respectively, and the weak$^*$ convergence of~$\zeta^{k} = f^{k} g^{k} (\mu^{k}_{ t})  \nu^{k}$ to~$\zeta= f g(\mu_{t}) \nu$, we infer that~$\nu$ solves
\begin{equation}
\label{e:414}
\begin{cases}
\partial_{t} \nu_{t} + \dive \big( (w_{\mu_{t}, \nu_{t}} + f(t, \cdot) g(\mu_{t})) \nu_{t}\big) = 0\,, \\
\nu_{0} = \overline{\nu}_{0}\,.
\end{cases}
\end{equation}
Therefore,~\eqref{e:414.1} and~\eqref{e:414} imply that $(\mu, \nu,\zeta,  f) \in \mathcal{S}(\overline{\mu}_{0}, \overline{\nu}_{0})$ and the proof is concluded.
 \end{proof}

We are now in a position to prove existence of solutions to~\eqref{e:min-mean}.

\begin{theorem}
\label{p:existence-mean}
Let $(\overline{\mu}_0,\overline{\nu}_0)\in\cP_2(\R^d)\times\cP_q(\R^d)$ be such that $\overline{\mu}_{0}= \overline{\rho}_{0}\, \di x={\rm Law} (\overline{X}_{0})$ for some $\overline{\rho}_{0} \in L^{1} (\R^{d})$ with finite entropy and some $\overline{X}_0\in L^2(\Omega;\R^d)$.
Then there exists a solution to~\eqref{e:min-mean}. 
\end{theorem}

\begin{proof}
We apply the Direct Method. 
In view of Remark~\ref{r:Snonvuoto}, we can assume that $\mathcal{S} (\overline{\mu}_{0}, \overline{\nu}_{0})\neq\emptyset$.
Let $(\mu^{k}, \nu^{k}, \zeta^{k}, g^{k}) \in \mathcal{S} (\overline{\mu}_{0}, \overline{\nu}_{0})$ be a minimizing sequence for~\eqref{e:min-mean}; in particular, we may assume that
\begin{equation}
\label{e:411}
\sup_{k\in\N} \, E(\mu^{k}, \nu^{k}, \zeta^{k}, g^{k}) <+\infty\,.
\end{equation} 
Applying Lemma~\ref{l:lsc-cost} to each~$(\mu^{k}, \nu^{k}, \zeta^{k}, g^{k})$ we deduce that for every $k \in \mathbb{N}$ there exists $f^{k}\in L^{1}_{\nu^{k}} ([0, T]\times  \R^{d}; K)$ such that $\zeta^{k} = f^{k} g^{k}(\mu^{k}_{ t}) \nu^{k}$ and
\begin{align*}
& \Phi (\zeta^{k}, g^{k}) = \int_{0}^{T} \int_{\R^{d}} \phi( f^{k} (t, y) , g^{k} (\mu^{k}_{ t})) \, \di \nu^{k}_{ t} (y) \, \di t\,,
\\
&
 \sup_{k\in\N} \, \Phi(\zeta^{k}, g^{k}) <+\infty\,.
\end{align*}
Again by Lemma~\ref{l:lsc-cost} we have that there exists~$(\mu, \nu, \zeta, g) \in \mathcal{S}(\overline{\mu}_{0}, \overline{\nu}_{0})$ such that $\mu^{k} \to \mu$ and~$\nu^{k} \to \nu$ in $C([0, T]; \mathcal{P}_{1} (\R^{d}))$, $\zeta^{k} \to \zeta$ weakly$^*$ in~$\mathcal{M}_{b} ([0, T]\times \R^{d}; \R^{d})$, and $g^{k} \to g $ locally uniformly in $C(\cP_1(\R^d))$. Furthermore, there exists $f \in L^{1}_{\nu} ([0, T]\times \R^{d}; K)$ such that $\zeta = f g(\mu_{t}) \nu$ and 
\begin{align}
\label{e:411.1}
 \Phi(\zeta, g) & \leq \int_{0}^{T} \int_{\R^{d}} \phi( f (t, y) , g (\mu_{ t})) \, \di \nu_{t} (y) \, \di t 
 \\
 &
 \leq \liminf_{k\to\infty} \int_{0}^{T} \int_{\R^{d}} \phi( f^{k} (t, y) , g^{k} (\mu^{k}_{ t})) \, \di \nu^{k}_{ t} (y) \, \di t = \liminf_{k\to\infty} \,  \Phi(\zeta^{k}, g^{k})\,. \nonumber
 \end{align}
 By the continuity properties of the Lagrangian~$\mathcal{L}$ we have that
 \begin{equation}
 \label{e:lag-cont}
\int_{0}^{T}  \mathcal{L} (\mu_{t}, \nu_{t}) \,  \di t = \lim_{k\to\infty} \int_{0}^{T} \mathcal{L} (\mu^{k}_{ t}, \nu^{k}_{ t}) \, \di t \,.
 \end{equation}
Finally, from~\eqref{e:411.1} and~\eqref{e:lag-cont} we conclude that
\begin{displaymath}
E(\mu, \nu, \zeta, g) \leq \liminf_{k\to\infty} \, E(\mu^{k}, \nu^{k}, \zeta^{k}, g^{k}) 
\end{displaymath}
Hence, $(\mu, \nu, \zeta, g)$ is a solution to~\eqref{e:min-mean}.
\end{proof}

\section{Finite-particle control problems}\label{s:sez4}

In this section, we are going to present a finite-particle approximation of~\eqref{e:min-mean}. 
In Section~\ref{s:particles}, we discuss the setting for the stochastic evolution of two groups of agents, one which is directly subject to additive noise and one which is controlled. 
Although the controls are deterministic, the coupling of the dynamics of the two populations induces stochasticity also in the evolution of the controlled agents.
In Section~\ref{s:chaos}, we prove propagation of chaos, yielding a deterministic controlled equation.

The convergence of the finite-particle problem to Problem 2 (see~\eqref{e:min-mean}) will be left to Section~\ref{s:limit}, where it is studied in terms of $\Gamma$-convergence (see Theorem~\ref{t:gamma-lim} and Corollary~\ref{c:minima} below).

\subsection{Particles system approximation}
\label{s:particles}

Let us fix $M, m \in \mathbb{N}$. Given the initial conditions~$\overline{\bX}_{0} = (\overline{X}_{0, 1}, \ldots, \overline{X}_{0, M}) \in (L^{p} (\Om; \R^{d}))^{M}$ and $\overline{\y}_{0}:= (\overline{y}_{0, 1}, \ldots, \overline{y}_{0, m}) \in (\R^{d})^{m}$, the controls $\uu\coloneqq (u_1,\ldots,u_m)\in L^{1} ([0, T]; K^m)$, and $g \in \mathcal{G}$, we consider the finite-particle system
\begin{equation}
\label{e:finite-particles}
\left\{
\begin{array}{llll}
\di X_{i}(t) = v_{\mu^{M, m}_{t}, \nu^{M, m}_{t}} (X_{i}(t)) \, \di t + \sqrt{2\sigma}\, \di  W(t)\,, \qquad \text{for $i = 1, \ldots, M$,}\\[2mm] 
\dot{y}_{j} (t) = w_{\mu^{M, m}_{t}, \nu^{M, m}_{t}} (y_{j} (t)) + u_{j} (t) g (\mu^{M, m}_{t}) \,, \qquad \text{for $j = 1, \ldots, m$,}\\[3mm]
X_{i} (0) = \overline{X}_{0, i}, \, y_{j} (0) = \overline{y}_{0, j}\,,\\[1mm]
\displaystyle \mu^{M, m}_{t} := \frac{1}{M} \sum_{i=1}^{M} \delta_{X_{i}(t)}\,, \ \nu^{M, m}_{t} := \frac{1}{m} \sum_{j=1}^{m} \delta_{y_{j} (t)}\,.
\end{array}
\right.
\end{equation}
Notice that, for fixed $\uu$ and~$g$, the existence of a unique pathwise solution $(\bX,\y)$ to~\eqref{e:finite-particles} is a standard result in SDE theory, under the assumption that both~$v$ and~$w$ satisfy~\eqref{e:vLip}. The notion of pathwise solution to~\eqref{e:finite-particles} is given analogously to Definition~\ref{d:pathwise}, with the obvious modifications.

The cost functional associated with~\eqref{e:finite-particles} is defined by
\begin{align}
\label{e:cost}
\E  (\bX, \y, \uu, g ) & := \mathbb{E} \bigg( \int_{0}^{T} \mathcal{L} ( \mu^{M, m}_{t}, \nu^{M, m}_{t} ) \, \di t + \frac{1}{m} \sum_{j=1}^{m} \int_{0}^{t} \phi\big( u_{j}, g ( \mu^{M, m}_{t})\big) \, \di t  \bigg) \,.
\end{align}
In the next proposition we state the existence of the finite-particle optimal control problem
\begin{align}
\label{e:min-finite}
\min\,  \big\{ \E(\bX, \y, \uu, g) : \ & (\bX, \y) \text{ solves~\eqref{e:finite-particles} with initial condition~$(\overline{\bX}_{0}, \overline\y_{0})$}
\\
&
\text{and controls $\uu \in L^{1} ([0, T]; K^{m}), \, g \in \G$} \big\}\,. \nonumber
\end{align}

\begin{proposition}
\label{p:finite-optimal}
Let $M, m \in \mathbb{N}$ and $(\overline{\bX}_{0}, \overline{\y}_{0} ) \in (L^{p} (\Om; \R^{d}))^{M} \times (\R^{d})^{m}$ be fixed. 
Then there exists a solution to~\eqref{e:min-finite}.
\end{proposition}

The proof of Proposition~\ref{p:finite-optimal} relies on the following convergence result.

\begin{lemma}
\label{l:finite-particles}
Let $M, m \in \mathbb{N}$ and $(\overline{\bX}_{0}, \overline{\y}_{0} ) \in (L^{p} (\Om; \R^{d}))^{M} \times (\R^{d})^{m}$ be fixed. Let $\uu^{k}, \uu \in L^{1}([0, T]; K^{m})$ and $g^{k}, g \in \mathcal{G}$ be such that $\uu^{k} \rightharpoonup \uu$ weakly$^*$ in $L^{\infty} ([0, T] ; (\R^{d})^{m})$ and $g^{k} \to g$ locally uniformly in~$C(\mathcal{P}_{1} (\R^{d}))$. Moreover, let $(\bX^{k}, \y^{k})$ and $(\bX, \y)$ be the corresponding solutions to~\eqref{e:finite-particles}. Then, 
\begin{equation}
\label{e:conv}
\lim_{k \to \infty} \, \mathbb{E} \bigg( \max_{\substack{i=1, \ldots, M, \\ t \in [0, T]}} | X^{k}_{i} (t) - X_{i} (t) | + \max_{\substack{ j=1, \ldots, m, \\ t \in [0, T]}} | y_{j}^{k} (t) - y_{j} (t) | \bigg) = 0\,.
\end{equation}
\end{lemma}

\begin{proof}
Along the proof we denote by~$C$ a generic positive constant depending on~$v$,~$w$,~$T$,~$K$, $\Lambda$, $\Delta$,~$\overline{\bX}_{0}$, and~$\overline{\y}_{0}$, but not on~$k$.
Since~$M$ and~$m$ are fixed, we will drop them to keep the notation lighter.

For $\omega \in \Om$, $t \in [0, T]$, and $j = 1, \ldots, m$, we estimate by \eqref{e:vLip} and by definition of~$\G$ and~$K$
\begin{align}
\label{e:105}
| y^{k}_{j} (t) - y_{j} (t) | &  \leq \int_{0}^{t} | w_{\mu^{k}_{\tau}, \nu^{k}_{\tau}} (y^{k}_{j} (\tau)) - w_{\mu_{\tau}, \nu_{\tau}} (y_{j}(\tau)) | \, \di \tau 
\\
& 
\qquad + \bigg| \int_{0}^{t} \big ( u^{k}_{j} (\tau) g^{k} (\mu^{k}_{\tau}) - u_{j} (\tau) g (\mu_{\tau}) \big) \, \di \tau\bigg| \nonumber
\\
&
\leq L_{w} \int_{0}^{t}  \big( \cW_{1} ( \mu^{k}_{\tau}, \mu_{\tau}) + \cW_{1} (\nu^{k}_{\tau}, \nu_{\tau}) + | y^{k}_{j} (\tau) - y_{j} (\tau) | \big) \di \tau \nonumber
\\
&
\qquad + \bigg| \int_{0}^{t}  ( u^{k}_{j} (\tau) - u_{j} (\tau) ) \,g (\mu_{\tau})  \, \di \tau\bigg|  + \bigg| \int_{0}^{t}  u^{k}_{j} (\tau) \big( g^{k} (\mu^{k}_{\tau}) - g^{k} (\mu_{\tau}) \big)  \di \tau\bigg| \nonumber
\\
&
\qquad +  \bigg| \int_{0}^{t}  u^{k}_{j} (\tau) \big( g^{k} (\mu_{\tau}) - g (\mu_{\tau}) \big)  \di \tau\bigg| \nonumber
\\
&
\leq  C\int_{0}^{t}   \cW_{1} ( \mu^{k}_{\tau}, \mu_{\tau})\, \di \tau  + C \int_{0}^{t} \max_{j = 1 , \ldots, m} | y^{k}_{j} (\tau) - y_{j} (\tau)|  \di \tau + \mathcal{R}^{k} _{j} (t) \nonumber\,,
\end{align}
where we have set
\begin{equation}\label{e:Rkj}
\mathcal{R}^{k}_{j}(t)  \coloneqq \bigg| \int_{0}^{t}  ( u^{k}_{j} (\tau) - u_{j} (\tau) ) \,g (\mu_{\tau})  \, \di \tau\bigg|  + \bigg| \int_{0}^{t}  u^{k}_{j} (\tau) \big( g^{k} (\mu_{\tau}) - g (\mu_{\tau}) \big)  \di \tau\bigg|\,.
\end{equation}
Taking the maximum over $j = 1, \ldots, m$ in~\eqref{e:105}  and applying Gr\"onwall inequality, we get that
\begin{align}
\label{e:106}
\max_{ j=1, \ldots, m} | y^{k}_{j} (t) - y_{j} (t) | &  \leq C \Big(  \int_{0}^{t}   \cW_{1} ( \mu^{k}_{\tau}, \mu_{\tau})\, \di \tau + \sum_{j=1}^{m} \mathcal{R}^{k}_{j} (t) \Big) 
\\
&
\leq \nonumber C \Big( \int_{0}^{t}   \max_{i=1, \ldots, M} | X^{k}_{i} (\tau) - X_{i} (\tau)| \, \di \tau + \sum_{j=1}^{m} \mathcal{R}^{k}_{j} (t) \Big) \,.
\end{align}

In a similar way, we now estimate $|X^{k}_{i} (t) - X_{i} (t) |$:
\begin{align}
\label{e:107}
|X^{k}_{i} (t) - X_{i} (t) | & \leq L_{v} \int_{0}^{t} \big( \cW_{1} (\mu^{k}_{\tau}, \mu_{\tau}) + \cW_{1} (\nu^{k}_{\tau}, \nu_{\tau}) + | X^{k}_{i} (\tau) - X_{i} (\tau)|\big) \di \tau
\\
&
\leq  2L_{v} \int_{0}^{t} \Big( \max_{i=1, \ldots, M} | X^{k}_{i} (\tau) - X_{i} (\tau)| + \max_{j=1, \ldots, m} | y^{k}_{j} (\tau) - y_{j} (\tau) | \Big) \di \tau \,.\nonumber
\end{align}
Inserting~\eqref{e:106} into~\eqref{e:107} and taking the maximum over $i=1, \ldots, M$ we obtain
\begin{align}
\label{e:108}
\max_{i=1, \ldots, M} |X^{k}_{i} (t) - X_{i} (t) | & \leq C \bigg(  \int_{0}^{t}  \max_{i=1, \ldots, M} | X^{k}_{i} (\tau) - X_{i} (\tau)|\, \di \tau 
\\
&
\qquad  +   \int_{0}^{t} \int_{0}^{\tau} \max_{i=1, \ldots, M}   | X^{k}_{i} (s) - X_{i} (s)| \, \di s\, \di \tau +   \sum_{j=1}^{m} \int_{0}^{t}   \mathcal{R}^{k}_{j} (\tau)\, \di \tau \bigg)  \nonumber
\\
&
\leq  C  \int_{0}^{t} \max_{i=1, \ldots, M} | X^{k}_{i} (\tau) - X_{i} (\tau)|\, \di \tau +C \sum_{j=1}^{m} \int_{0}^{t}  \mathcal{R}^{k}_{j} (\tau)\, \di \tau\,. \nonumber
\end{align}
By Gr\"onwall inequality, we deduce from~\eqref{e:108} that
\begin{align}
\label{e:109}
\max_{\substack{ i=1, \ldots, M, \\ t \in [0, T]}} |X^{k}_{i} (t) - X_{i} (t) | & \leq C \sum_{j=1}^{m} \int_{0}^{T} \mathcal{R}^{k}_{j} (\tau) \, \di \tau\,.
\end{align}
We notice that $\mathcal{R}^{k}_{j} (t) \to 0$ as $k \to \infty$ for $\mathbb{P}$-a.e.~$\omega \in \Om$ and a.e.~$t \in [0, T]$.
Indeed, $u^{k}_{j}$ converges weakly$^*$ to $u_{j}$ for $j = 1, \ldots, m$, $g (\mu_{t})$ is an $L^{\infty}$-function, and $g^{k}(\mu_\tau)$ converges uniformly to $g(\mu_\tau)$. Moreover, each $\mathcal{R}^{k}_{j}$ is equi-Lipschitz continuous in~$[0, T]$, since $g_{k}$ is $\Delta$-bounded and $u^{k}_{j}, u_{j}$ are uniformly bounded. By Ascoli-Arzel\'a Theorem, this implies that 
$\mathcal{R}^{k}_{j} \to 0$ uniformly in~$[0, T]$ for $\mathbb{P}$-a.e.~$\omega \in \Om$. By dominated convergence, we further infer that
\begin{equation}
\label{e:110}
\lim_{k \to \infty} \, \mathbb{E} \bigg( \max_{t \in [0, T] } \sum_{j=1}^{m} \mathcal{R}^{k}_{j} (t) \bigg) =0\,.
\end{equation}

Combining~\eqref{e:109} and~\eqref{e:110} we get that
\begin{equation}
\label{e:111}
\lim_{k \to \infty} \, \mathbb{E} \bigg( \max_{\substack{ i=1, \ldots, M, \\ t \in [0, T]}} | X^{k}_{i} (t) - X_{i} (t) | \bigg) = 0\,.
\end{equation}
Finally, combining~\eqref{e:106},~\eqref{e:110}, and~\eqref{e:111} we infer that
\begin{displaymath}
\lim_{k\to\infty} \, \mathbb{E} \bigg( \max_{\substack{ j=1, \ldots, m, \\ t \in [0, T]}} | y^{k}_{j} (t) - y_{j} (t) | \bigg) = 0\,.
\end{displaymath}
This concludes the proof of~\eqref{e:conv} and of the lemma.
\end{proof}

We are now in a position to prove Proposition~\ref{p:finite-optimal}.

\begin{proof}[Proof of Proposition~\ref{p:finite-optimal}]
We proceed by the Direct Method. Since $M, m \in \mathbb{N}$ are fixed, we drop them in our notation. Let $(\uu^{k}, g^{k}) \in L^{1}([0, T]; K^{m}) \times \G$ with corresponding solutions $(\bX^{k}, \y^{k})$ to~\eqref{e:finite-particles} be a minimizing sequence for~\eqref{e:min-finite}. By definition of~$K$ and~$\G$, there exists $(\uu, g) \in L^{1}([0, T]; K^{m}) \times \G$ such that, up to a subsequence, $\uu^{k} \rightharpoonup \uu$ weakly$^*$ in $L^{\infty} ([0, T]; (\R^{d})^{m})$ and $g^{k} \to g$ locally uniformly in $C(\mathcal{P}_{1} (\R^{d}))$. Let $(\bX, \y)$ be the solution to~\eqref{e:finite-particles} with controls~$(\uu, g)$ and initial datum $(\overline{\bX}_{0}, \overline{\y}_{0})$. In view of Lemma~\ref{l:finite-particles} we have that, along the same subsequence,
\begin{equation}
\label{e:305}
\lim_{k \to \infty} \, \mathbb{E} \bigg(  \max_{\substack{ i=1, \ldots, M\\ t \in [0, T]}} | X_{i}^{k} (t) - X_{i}(t)| + \max_{\substack{j=1, \ldots, m,\\ t \in [0, T]}} | y^{k}_{ j} (t) - y_{j} (t)|\bigg) = 0\,.
\end{equation}
In particular,~\eqref{e:305} implies that $\cW_{1} (\mu^{k}_{ t}, \mu_{t}) \to 0$ and $\cW_{1} (\nu^{k}_{t}, \nu_{t}) \to 0$ uniformly in~$t \in [0, T]$ as $k \to \infty$.

We now prove the lower-semicontinuity of the cost. Let us denote by~$\omega_{\mathcal{L}}$ a concave modulus of continuity of~$\mathcal{L}$. Then, by Jensen we estimate
\begin{align*}
\label{e:306}
\bigg| \mathbb{E} \bigg(  \int_{0}^{T}  \mathcal{L} (\mu^{k}_{ t}, \nu^{k}_{t}) \, \di t -  \int_{0}^{T} \mathcal{L} (\mu_{t}, \nu_{t}) \, \di t  \bigg) \bigg|
&\leq \int_{0}^{T} \mathbb{E} \Big(  \omega_{\mathcal{L}} \big( \cW_{1} (\mu^{k}_{ t}, \mu_{t}) + \cW_{1} (\nu^{k}_{t}, \nu_{t})\big) \Big) \, \di t 
\\
&
\leq \int_{0}^{T}  \omega_{\mathcal{L}} \Big( \mathbb{E} \Big(  \cW_{1} (\mu^{k}_{ t}, \mu_{t}) + \cW_{1} (\nu^{k}_{t}, \nu_{t})\Big) \Big) \, \di t  \,.  \nonumber
\end{align*} 
In view of~\eqref{e:305} we have that
\begin{equation}
\label{e:307}
\lim_{k\to \infty} \, \int_{0}^{T}  \mathcal{L} (\mu^{k}_{t}, \nu^{k}_{t}) \, \di t = \int_{0}^{T} \mathcal{L} (\mu_{t}, \nu_{t}) \, \di t \,.
\end{equation}

As for the control part of the functional~$\mathcal{E}$, we simply rewrite
\begin{align}
\label{e:308}
\mathbb{E} \bigg( \frac{1}{m}\sum_{j=1}^{m}&  \int_{0}^{T} \phi\big(u^{k}_{ j} (t) , g^{k} (\mu^{k}_{t})  \big) \, \di t \bigg) 
 = \mathbb{E} \bigg( \frac{1}{m} \sum_{j=1}^{m} \int_{0}^{T} \phi\big(u^{k}_{ j} (t) , g (\mu_{t})  \big) \, \di t  \bigg)
 \\
 &
 \qquad + \mathbb{E} \bigg( \frac{1}{m} \sum_{j=1}^{m} \int_{0}^{T} \Big( \phi\big(u^{k}_{ j} (t) , g^{k} (\mu^{k}_{ t})  \big) -\phi\big(u^{k}_{ j} (t) , g (\mu_{t})  \big) \Big)  \, \di t \bigg)  \nonumber
\\
&
=: I_{k} + II_{k} \,. \nonumber
\end{align}
By~$(\phi2)$ and the weak$^*$ convergence of~$\uu^{k}$ to~$\uu$ we have that
\begin{equation}
\label{e:309}
\liminf_{k\to \infty} \, I_{k} \geq \mathbb{E} \bigg( \frac{1}{m} \sum_{j=1}^{m} \int_{0}^{T} \phi\big(u_{j} (t) , g (\mu_{t})  \big) \, \di t \bigg)  \,.
\end{equation}
Since $\uu^{k}_{t} \in K^{m}$ for a.e.~$t \in [0, T]$, for every $k \in \mathbb{N}$, $g^{k}, g$ are $\Delta$-bounded, and $\phi$ is continuous (see $(\phi1)$), denoting by~$\omega_{\phi}$ a concave modulus of continuity of~$\phi$ in the compact set $K \times [-\Delta, \Delta]$ we further have that
\begin{align*}
|II_{k}| &  \leq \int_{0}^{T} \mathbb{E} \bigg( \frac{1}{m} \sum_{j=1}^{m} \omega_{\phi} \big( | g^{k} (\mu^{k}_{t}) - g (\mu_{t}) | \big) \bigg) \, \di t
\leq \int_{0}^{T}  \frac{1}{m} \sum_{j=1}^{m} \mathbb{E} \bigg( \omega_{\phi} \big( \Lambda \cW_{1} (\mu^{k}_{ t} , \mu_{t} )  \big) \bigg)  \, \di t
\\
&
\leq  \int_{0}^{T}  \frac{1}{m} \sum_{j=1}^{m}  \omega_{\phi} \Big( \Lambda \mathbb{E} \big(\cW_{1} (\mu^{k}_{ t} , \mu_{t} )  \big) \Big)  \, \di t\,.
\end{align*}
By the uniform convergence $\mathbb{E} \big(\cW_{1} (\mu^{k}_{ t} , \mu_{t} )  \big)\to 0$ in~$[0, T]$, we infer that
\begin{equation}
\label{e:310}
\lim_{k\to \infty} \, II_{k} = 0\,.
\end{equation}
Thus, combining~\eqref{e:308}--\eqref{e:310} we finally obtain
\begin{displaymath}
\liminf_{k\to \infty} \, \mathcal{E}(\bX^{k}, \y^{k}, \uu^{k}, g^{k}) \geq \mathcal{E}(\bX, \y, \uu, g)\,.
\end{displaymath}
This concludes the proof of existence of solutions to~\eqref{e:min-finite}.
\end{proof}

\subsection{Auxiliary estimates}
\label{s:chaos}

This section is devoted to an intermediate step towards the mean-field limit of problem~\eqref{e:min-finite}. Namely, we analyze here the propagation of chaos for system~\eqref{e:system} and the corresponding convergence of the cost functional~$\mathcal{E}$.

Let us fix $\overline{\y}_{0} = (\overline{y}_{0, 1}, \ldots, \overline{y}_{0, m}) \in (\R^{d})^{m}$, $\overline{X}_{0}  \in L^{p} (\Om; \R^{d})$, $\uu \in L^{1}([0, T]; K^{m})$, and $g \in \G$ and consider the system
\begin{equation}
\label{e:chaos}
\left\{
\begin{array}{llll}
\di X(t) = v_{\mu^{m}_{t}, \nu^{m}_{t}} (X(t)) \, \di t + \sqrt{2\sigma} \, \di W(t) \,,\\[2mm]
\dot{y}_{j} (t) = w_{\mu^{m}_{t}, \nu^{m}_{t}} (y_{j} (t)) + u_{j} (t) g (\mu^{m}_{t}) \,, \qquad j = 1, \ldots, m\,,\\[2mm]
X(0) = \overline{X}_{0}\,, \ y_{j} (0) = \overline{y}_{0, j}\,,\\[2mm]
\displaystyle \bmu^{m} = {\rm Law} (X) \,, \ \mu^{m}_{t} = ({\rm ev}_{t}) _{\#} \bmu^{m}\,,\ \nu^{m}_{t} = \frac{1}{m} \sum_{j=1}^{m} \delta_{y_{j} (t)}\,.
\end{array}
\right.
\end{equation}
To simplify the notation, we further set $\overline{\mu}_{0} \coloneqq {\rm Law} (\overline{X}_{0})$ and $\overline{\nu}_{0}^{m} \coloneq \frac{1}{m} \sum_{j=1}^{m} \delta_{\overline{y}_{0, j}}$. 
Associated with system~\eqref{e:chaos}, we introduce the cost functional
\begin{equation}
\label{e:cost-chaos}
\mathfrak{E}( X, \y, \uu, g) = \int_{0}^{T} \mathcal{L} ( \mu^{m}_{t}, \nu^{m}_{t}) \, \di t + \frac{1}{m} \sum_{j=1}^{m} \int_{0}^{T} \phi(u_{j} (t) , g (\mu^{m}_{t}))  \, \di t\,.
\end{equation}

We first discuss the existence and uniqueness of solutions to~\eqref{e:chaos}. 
\begin{proposition}
\label{p:existence-chaos}
Let $m \in \mathbb{N}$, $(\overline{X}_{0} , \overline{\y}_{0}) \in L^{p} (\Om; \R^{d}) \times (\R^{d})^{m}$, $\uu \in L^{1}([0, T]; K^{m})$, and $g \in \mathcal{G}$. Then, there exists a unique solution $(X, \y) \in \mathcal{M} (\Om ; C([0, T]; \R^{d}))  \times C([0, T]; (\R^{d})^{m})$ to~\eqref{e:chaos} with initial conditions~$(\overline{X}_{0}, \overline{\y}_{0})$ and controls~$(\uu, g)$. Moreover, for $q \in (1, +\infty]$ there exists a constant $D = D(p,q, T, M_{v}, M_{w}, K, \Delta)>0$ such that 
\begin{align}
\label{e:1100}
m_{p} (\bmu^{m}) & \leq D \big( 1 +  m_{p} (\overline{\mu}_{0}) + m_{1}^{p} (\overline{\nu}^{m}_{0}) \big) \,,\\
\max_{t \in [0, T]} \,  m_{q} (\nu^{m}_{t})  & \leq D \big( 1  + m_{p}^{q} (\overline{\mu}_{0}) + m_{q} (\overline{\nu}^{m}_{0}) \big) \,. \nonumber
\end{align}
\end{proposition}

\begin{proof}
Concerning existence and uniqueness, we can set up the fixed point result of~\cite[Theorem~3.1]{ASC_22}. Indeed, the precise structure of the velocity fields used in~\cite{ASC_22} is not needed and only the Lipschitz continuity assumed in~\eqref{e:vLip} is necessary.

For the estimates on the moments, we may observe that the measures~$\nu_t^m$ are indeed solutions to the continuity equation
$$\partial_t\nu_t^m+\dive (w_{\mu_t^m,\nu_t^m} \nu_t^m + \zeta_t^m)=0,$$
where 
$$\zeta_t^{m} \coloneqq \frac{1}{m} \sum_{j=1}^{m} u_{j} (t) g (\mu^{m}_{t}) \delta_{y_{j} (t)} \quad \in \mathcal{M}_{b} (\R^{d}; \R^{d})\,,
$$
and repeat the computations of Lemmas~\ref{l:spt-mp} and~\ref{l:lemma3.2}. 
\end{proof}

\EEE

\begin{remark}
\label{r:bdd}
We notice that the estimate in~\eqref{e:1100} implies that $m_{p} (\mu^{m}_{t})$ and $m_{p} (\bmu^{m})$ are uniformly bounded with respect to~$m \in \mathbb{N}$ if we assume that
\begin{equation}
\label{e:nu0bound}
\sup_{m \in \mathbb{N}}  \, m_{q} (\overline{\nu}^{m}_{0}) <+\infty\,.
\end{equation}
\end{remark}

The next proposition can be proved following the arguments of~\cite[Lemma~3.8]{ASC_22}.
\begin{proposition}
\label{p:iid}
Let $\overline{\bX} = (\overline{X}_{0, 1}, \ldots, \overline{X}_{0, M}) \in ( L^{p} (\Om; \R^{d}))^{M}$,  $\overline{\y}_{0} = (\overline{y}_{0, 1}, \ldots, \overline{y}_{0, m}) \in (\R^{d})^{m}$, $\uu \in L^{1} ([0, T]; K^{m})$, and~$g \in \G$, and assume that $\overline{X}_{0, i}$ are i.i.d. Then, the solutions~$(X_{i} , \y_{i} )$ of~\eqref{e:chaos} with initial conditions $(\overline{X}_{0, i} , \overline{\y}_{0})$ satisfy the following:
\begin{itemize}
\item[$(i)$] $X_{i}$ are i.i.d.;
\item[$(ii)$] $\y_{i} = \y_{1}$ for $i = 1, \ldots, M$.
\end{itemize} 
\end{proposition}


\begin{remark}
In the notation of Proposition~\ref{p:iid}, since $X_{i}$ are i.i.d.~and $\y_{i} = \y_{1}$ for $i=1, \ldots, M$, we have that the cost functional~$\mathfrak{E}$ satisfies
$$
\mathfrak{E}(X_{i}, \y_{i} , \uu, g) = \mathfrak{E}(X_{j}, \y_{j}, \uu, g)  \qquad \text{for $i, j=1, \ldots, M$}.
$$
\end{remark}

In the next result we show that, by propagation of chaos,~\eqref{e:finite-particles} and~\eqref{e:chaos} are equivalent in the limit as~$M \to \infty$, uniformly with respect to~$m \in \mathbb{N}$.

\begin{theorem}[Propagation of chaos]
\label{t:chaos}
Let $m \in \mathbb{N}$, $\overline{\y}_{0} = (\overline{y}_{0, 1}, \ldots, \overline{y}_{0, m}) \in (\R^{d})^{m}$, $\uu \in L^{1}( [0, T]; K^{m})$, and $g \in \G$. For every $M \in \mathbb{N}$, let $\overline{\bX}_{0} \in (L^{p} (\Om; \R^{d}))^{M}$ be such that $\overline{X}_{0, i}$ are i.i.d., and let~$(\bX, \y)$ and $(\widetilde{\bX}, \widetilde{\y})$ be the solutions to~\eqref{e:finite-particles} and to~\eqref{e:chaos}, respectively. Then, there exists a positive constant~$C(\overline{\mu}_{0}, \overline{\nu}^{m}_{0}, M)$ such that $C(\overline{\mu}_{0}, \overline{\nu}^{m}_{0}, M) \to 0$ as~$M \to \infty$ and
\begin{align}
\label{e:chaos-prop}
& \mathbb{E} \bigg( \max_{\substack{i=1, \ldots, M, \\ t \in [0, T]}} | X_{i} (t) - \widetilde{X_{i}} (t)| + \max_{\substack{i=1, \ldots, m,\\ t \in [0, T]}} | y_{j} (t) - \widetilde{y}_{j} (t) | \bigg) \leq C(\overline{\mu}_{0}, \overline{\nu}^{m}_{0}, M) \,,
\\
& \label{e:chaos-prop-2}
 \big| \mathcal{E} (\bX, \y, \uu, g) - \mathfrak{E}( \widetilde{X}_{i}, \widetilde{\y}_{i} , \uu, g) \big| \leq C(\overline{\mu}_{0}, \overline{\nu}^{m}_{0}, M)  \qquad \text{for $i = 1, \ldots, M$}\,.
\end{align}
Moreover, if~\eqref{e:nu0bound} is satisfied, then $C(\overline{\mu}_{0}, \overline{\nu}^{m}_{0}, M)$ only depends on~$\overline{\mu}_0$ and~$M$.
\end{theorem}

\begin{proof}
We denote by $C$ a generic positive constant, which may vary from line to line. 

We start by estimating $y_{j} - \tilde{y}_{j}$. By \eqref{e:vLip} and the assumptions on the controls~$\uu$ and $g$ we have that
\begin{align*}
| y_{j} (t) - \tilde{y}_{j} (t) | & \leq L_{w} \int_{0}^{t} \big( \cW_{1} (\mu^{M, m}_{\tau}, \mu^{m}_\tau) + \cW_{1} (\nu^{M, m}_{\tau}, \nu^{m}_{\tau} ) + | y_{j} (\tau) - \tilde{y}_{j} (\tau) | \big) \, \di \tau 
\\
&
\leq  L_{w} \int_{0}^{t} \big( \cW_{1} (\mu^{M, m}_{\tau}, \mu^{m}_\tau) + 2 \max_{j = 1, \ldots, m} | y_{j} (\tau) - \tilde{y}_{j} (\tau) | \big) \, \di \tau \,.\nonumber
\end{align*}
Taking the supremum over $j$ and applying Gr\"onwall inequality we get
\begin{align}
\label{e:205.1}
\max_{j = 1, \ldots, m} | y_{j} (t) - \tilde{y}_{j} (t) | & \leq C  \int_{0}^{t}  \cW_{1} (\mu^{M, m}_{\tau}, \mu^{m}_\tau) \, \di \tau  \,.
\end{align}
Denoting by $\widetilde{\mu}^{M, m}_{t} \coloneqq \frac{1}{M} \sum_{i=1}^{M} \delta_{\widetilde{X}_{i} (t)}$, by triangle inequality we obtain
\begin{align}
\label{e:205}
\max_{j = 1, \ldots, m} | y_{j} (t) - \tilde{y}_{j} (t) | & \leq C \bigg( \int_{0}^{t}  \cW_{1} (\mu^{M, m}_{\tau}, \widetilde{\mu}^{M, m}_\tau) + \cW_{1} (\widetilde{\mu}^{M, m}_{\tau}, \mu^{m}_\tau) \, \di \tau  \bigg)\,.
\end{align}

In the same way we estimate $| X_{i} (t) - \widetilde{X}_{i}(t)|$:
\begin{align}
\label{e:206}
| X_{i} (t) - \widetilde{X}_{i} (t)| & \leq L_{v} \int_{0}^{t} \big( \cW_{1} (\mu^{M, m}_{\tau}, \mu^{m}_\tau) + \cW_{1} (\nu^{M, m}_{\tau}, \nu^{m}_{\tau} ) + | X_{i} (\tau) - \widetilde{X}_{i} (\tau) | \big) \, \di \tau
\\
&
\leq  L_{v} \int_{0}^{t} \big( \cW_{1} (\mu^{M, m}_{\tau}, \mu^{m}_\tau) + \max_{j=1, \ldots, m} | y_{j} (\tau) - \tilde{y}_{j} (\tau)|  + | X_{i} (\tau) - \widetilde{X}_{i} (\tau) | \big) \, \di \tau \,. \nonumber
\end{align}
Combining~\eqref{e:205.1} and~\eqref{e:206} we get that
\begin{align*}
\max_{i=1, \ldots, M} | X_{i} (t) - \widetilde{X}_{i}(t)| & \leq C \int_{0}^{t} \bigg( \cW_{1} (\mu^{M, m}_{\tau}, \mu^{m}_\tau) +  \int_{0}^{\tau} \cW_{1} (\mu^{M, m}_{s} , \mu^{m}_{s}) \, \di s  
\\
&
\qquad \qquad + \max_{i=1, \ldots, M} | X_{i} (\tau) - \widetilde{X}_{i} (\tau) | \bigg) \, \di \tau \nonumber
\\
&
\leq C \int_{0}^{t} \big( \cW_{1} (\mu^{M, m}_{\tau}, \mu^{m}_\tau)  + \max_{i=1, \ldots, M} \, | X_{i} (\tau) - \widetilde{X}_{i} (\tau) | \big) \, \di \tau \,. \nonumber
\end{align*}
By Gr\"onwall inequality and by the triangle inequality we deduce that
\begin{align}
\label{e:207}
\max_{i=1, \ldots, M} | X_{i} (t) - \widetilde{X}_{i} (t)| & \leq C \int_{0}^{t}  \cW_{1} (\mu^{M, m}_{\tau}, \mu^{m}_\tau)  \, \di \tau
\\
&
\leq C \int_{0}^{t} ( \cW_{1} (\mu^{M, m}_{\tau}, \widetilde{\mu}^{M, m}_\tau) + \cW_{1} (\widetilde{\mu}^{M, m}_{\tau}, \mu^{m}_\tau) ) \, \di \tau \,.\nonumber
\end{align}
Integrating~\eqref{e:207} over~$\Om$, by definition of~$\mu^{M, m}_{t}$ and of~$\widetilde{\mu}^{M, m}_{t}$ and by Gr\"onwall inequality we infer that
\begin{align}
\label{e:208}
\mathbb{E} \Big( \cW_{1} (\mu^{M, m}_{t}, \widetilde{\mu}^{M, m}_t) \Big)  \leq C \int_{0}^{t}  \mathbb{E} \Big( \cW_{1} (\widetilde{\mu}^{M, m}_{\tau}, \mu^{m}_\tau) \Big)  \, \di \tau\,.
\end{align}
In view of Proposition~\ref{p:iid}, we have that $\widetilde{X}_{i}$ are i.i.d.~with~$\bmu^{m} = {\rm Law} (\widetilde{X}_{i})$ for $i = 1, \ldots, M$. Hence, by Theorem~\ref{t:iid} we have that
\begin{align}
\label{e:gui}
\mathbb{E} \Big( \cW_{1} (\widetilde{\mu}^{M, m}_{t}, \mu^{m}_t) \Big)  \leq C_{M}\,  m_{p} (\bmu^{m}) \qquad \text{for $t \in [0, T]$}\,,
\end{align}
for some positive constant~$C_{M}$ independent of~$m$ and~$t$ and such that $C_{M} \to 0$ as~$M \to \infty$. In view of Proposition~\ref{p:existence-chaos} and of Remark~\ref{r:bdd}, we have that $m_{p} (\bmu^{m})$ is bounded in terms of~$m_{p} (\overline{\mu}_{0})$ and of~$m_{q} (\overline{\nu}^{m}_{0})$. Hence, by \eqref{e:208} and \eqref{e:gui} we have that for $t \in [0, T]$
\begin{equation}
\label{e:209}
\mathbb{E}  \Big( \cW_{1} (\mu^{M, m}_{t}, \widetilde{\mu}^{M, m}_t)  \Big) + \mathbb{E}  \Big( \cW_{1} (\widetilde{\mu}^{M, m}_{t}, \mu^{m}_t) \Big)  \leq C'(\overline{\mu}_{0}, \overline{\nu}^{m}_{0}, M) \to 0 \qquad \text{as $M \to \infty$}\,.
\end{equation}
Combining~\eqref{e:205},~\eqref{e:207}, and~\eqref{e:209}, we conclude~\eqref{e:chaos-prop}. 

As for~\eqref{e:chaos-prop-2}, by the uniform continuity of~$\mathcal{L}$ we deduce that there exists a concave modulus of continuity~$\omega_{\mathcal{L}}$ such that
\begin{align}
\label{e:210}
\mathbb{E}  \bigg( \int_{0}^{T} \big(\mathcal{L} & (\mu^{M, m}_{t}, \nu^{M, m}_{t} ) - \mathcal{L} (\mu^{m}_{t}, \nu^{m}_{t} ) \big)  \, \di t \bigg) 
\\
&
\leq \mathbb{E} \bigg( \int_{0}^{T} \omega_{\mathcal{L}} \big( \cW_{1} (\mu^{M, m}_{t}, \mu^{m}_{t} ) + \cW_{1} ( \nu^{M, m}_{t}, \nu^{m}_{t}) \big) \, \di t\bigg) \nonumber
\\
&
\leq \int_{0}^{T} \omega_{\mathcal{L}} \bigg( \mathbb{E} \Big( \cW_{1} (\mu^{M, m}_{t}, \mu^{m}_{t} ) + \cW_{1} ( \nu^{M, m}_{t}, \nu^{m}_{t}) \Big) \bigg) \nonumber .
\end{align}
Let us further denote by~$\omega_{\phi}$ a concave modulus of continuity of~$\phi$ on~$K \times [-\Delta, \Delta]$. Since each~$g$ is $\Lambda$-Lipschitz continuous, we have that
\begin{align}
\label{e:211}
\mathbb{E} &  \bigg( \int_{0}^{T} \frac{1}{m} \sum_{j=1}^{m}  \big( \phi(u_{j} (t) , g (\mu^{M, m}_{t}) ) - \phi(u_{j} (t) , g (\mu^{m}_{t}) ) \big) \, \di t   \bigg) 
\\
&
\leq \mathbb{E} \bigg( \int_{0}^{T} \frac{1}{m}\sum_{j=1}^{m} \omega_{\phi} \big( | g(\mu^{M, m}_{t}) - g (\mu^{m}_{t}) | \big)\, \di t \bigg)  
\leq \mathbb{E} \bigg( \int_{0}^{T} \frac{1}{m} \sum_{j=1}^{m} \omega_{\phi} \big( \Lambda \,  \cW_{1}  ( \mu^{M, m}_{t}, \mu^{m}_{t}) \big) \, \di t \bigg) \nonumber
\\
&
\leq \int_{0}^{T} \omega_{\phi} \bigg( \Lambda \, \mathbb{E} \Big(  \cW_{1}  ( \mu^{M, m}_{t}, \mu^{m}_{t}) \Big) \bigg) \, \di t \,. \nonumber
\end{align}
We conclude for~\eqref{e:chaos-prop-2} by combining~\eqref{e:chaos-prop},~\eqref{e:210}, and~\eqref{e:211}.

Finally, we notice that if~\eqref{e:nu0bound} is satisfied, then the constant~$C'(\overline{\mu}_{0}, \overline{\nu}^{m}_{0}, M)$ in~\eqref{e:209} can be made independent of~$m \in \mathbb{N}$. This yields that also the constant~$C(\overline{\mu}_{0}, \overline{\nu}^{m}_{0}, M)$ can be taken independent of~$m$. 
\end{proof}

Since in Section~\ref{s:limit} we are interested in working with i.i.d.~initial conditions~$\{ \overline{X}_{0, i}\}_{i=1}^{M}$, in view of Proposition~\ref{p:iid} and of Theorem~\ref{t:chaos} we consider from now on the optimal control problem
\begin{equation}
\label{e:min-chaos}
\min\bigg\{ \mathfrak{E}(X,  \y, \uu, g) : (\uu,g) \in L^{1}([0, T]; K^{m}) \times \mathcal{G}, \, (X, \y) \text{ solves~\eqref{e:chaos}}  \bigg\}\,,
\end{equation}
with a single initial condition~$\overline{X}_{0} \in L^{p} (\Om; \R^{d})$. For $(X, \y)$ solution to~\eqref{e:chaos}, we recall the notation $\bmu^{ m}= {\rm Law} (X)$, $\mu^{ m}_{t} = ({\rm ev}_{t})_{\#} \bmu^{m}$, $\nu^{m}_{t}= \frac{1}{m} \sum \delta_{y_{j} (t)}$, $\overline{\mu}_{0} = {\rm Law} (\overline{X}_{0})$, and~$\overline{\nu}^{m}_{0} = \frac{1}{m} \sum \delta_{\overline{y}_{0, j}}$.

The well posedness of~\eqref{e:min-chaos} only comes at the expenses of slight modifications to the proof of Proposition~\ref{p:finite-optimal} and Lemma~\ref{l:finite-particles}, since now $\mu_t^m$ are no longer defined as empirical measures. We report the statements and a sketch of the proof.

\begin{lemma}
\label{l:cont-chaos}
Let $m \in \mathbb{N}$, $(\overline{X}_{0}, \overline{\y}_{0}) \in L^{p} (\Om; \R^{d}) \times (\R^{d})^{m}$, $(\uu^{k}, g^{k}), (\uu, g)  \in L^{1}([0, T];K^{m}) \times \G$, and $(X^{k}, \y^{k})$, $(X, \y)$ be the corresponding solutions to~\eqref{e:chaos}. Assume that $\uu^{k} \rightharpoonup \uu$ weakly$^*$ in $L^{\infty} ([0, T]; K)$ and that $g^{k} \to g$ locally uniformly in $C(\mathcal{P}_1 (\R^{d}))$. Then,
\begin{equation}
\label{e:299}
\lim_{k\to \infty} \, \mathbb{E} \bigg( \max_{\substack{ t \in [0, T]}} | X^{k} (t) - X(t)| + \max_{\substack{j=1, \ldots, m,\\ t \in [0, T]}} | y^{k}_{ j} (t) - y_{j} (t)| \bigg) = 0\,.
\end{equation}
\end{lemma}

\begin{proof}
We proceed following the lines of Lemma~\ref{l:finite-particles}. Here, we denote by~$C$ a positive constant independent of~$m$, which may vary from line to line. Since $m \in \mathbb{N}$ is fixed, we drop it in the notation of the measures $\mu^{k}$, $\mu$, $\nu^{k}$, and~$\nu$. 
Arguing as in~\eqref{e:106}, we deduce 
that
\begin{equation}
\label{e:301}
\max_{j=1, \ldots, m} \, | y^{k}_{ j} (t) - y_{j} (t)|  \leq C  \bigg(  \int_{0}^{t}  \cW_{1} ( \mu^{k}_{ \tau}, \mu_{\tau}) \, \di \tau + \sum_{j=1}^{m} \mathcal{R}^{k}_{j} (t)\bigg) \,,
\end{equation}
where $\mathcal{R}_{j}^{k}$ is defined in~\eqref{e:Rkj}.

By~\eqref{e:vLip} and by~\eqref{e:301} we have that
\begin{align*}
| X^{k} (t) - X(t)| & \leq L_{v} \int_{0}^{t} \big( \cW_{1} (\mu^{k}_{ \tau}, \mu_{\tau}) + \cW_{1} (\nu^{k}_{ \tau}, \nu_{\tau}) + | X^{k} (\tau) - X (\tau)| \big)\, \di \tau
\\
&
\leq L_{v} \int_{0}^{t} \big( \cW_{1} (\mu^{k}_{ \tau}, \mu_{\tau}) + \max_{j=1, \ldots, m} \, | y^{k}_{ j} (\tau) - y_{j} (\tau)|  + | X^{k} (\tau) - X (\tau)| \big)\, \di \tau \nonumber 
\\
&
\leq C\int_{0}^{t} \bigg( \cW_{1} (\mu^{k}_{ \tau}, \mu_{\tau})  + \sum_{j=1}^{m}\mathcal{R}^{k}_{j} (\tau) + | X^{k} (\tau) - X(\tau)| \bigg) \di \tau  \nonumber\,.
\end{align*}
By Gr\"onwall inequality we infer that
\begin{align}
\label{e:302}
| X^{k} (t) - X(t)| & \leq C \int_{0}^{t} \bigg( \cW_{1} (\mu^{k}_{ \tau}, \mu_{\tau})  + \sum_{j=1}^{m} \mathcal{R}^{k}_{j} (\tau) \bigg) \di \tau  \,.
\end{align}
Integrating~\eqref{e:302} over~$\Om$ we get, by definition of~$\mu^{k}_{ t}$ and of~$\mu_{t}$,
\begin{equation*}
\cW_{1} (\mu^{k}_{ t}, \mu_{t}) \leq \mathbb{E} \big(| X^{k} (t) - X(t)| \big)  \leq C\bigg(  \int_{0}^{t}  \cW_{1} (\mu^{k}_{ \tau}, \mu_{\tau}) \, \di \tau  + \mathbb{E} \bigg( \sum_{j=1}^{m}  \int_{0}^{t} \mathcal{R}^{k}_{j} (\tau) \, \di \tau \bigg)\bigg) \,.
\end{equation*}
Hence, by Gr\"onwall inequality we deduce that
\begin{equation}
\label{e:303}
\cW_{1} (\mu^{k}_{ t}, \mu_{t}) \leq C  \mathbb{E} \bigg( \sum_{j=1}^{m}  \int_{0}^{t} \mathcal{R}^{k}_{j} (\tau) \, \di \tau \bigg)\,.
\end{equation}
Arguing as in~\eqref{e:110}, we can show that the right-hand side of~\eqref{e:303} tends to $0$ as $k \to \infty$ uniformly in~$[0, T]$, so that 
\begin{equation}
\label{e:304}
\lim_{k\to \infty} \, \max_{t \in [0, T]}\, \cW_{1} (\mu^{k}_{ t}, \mu_{t})  = 0\,.
\end{equation}
Taking the maximum of~\eqref{e:301} and~\eqref{e:302} over $t \in [0, T]$, integrating over~$\Om$, and summing up yield
\begin{equation*}
\mathbb{E} \bigg(\max_{t \in [0, T]} \, | X^{k} (t)  - X(t) | +\max_{\substack{j=1, \ldots, m, \\ t \in [0, T]}} | y^{k}_{ j} (t) - y_{j} (t) |  \bigg) \leq C \max_{t \in [0, T]} \, \cW_{1} (\mu^{k}_{t}, \mu_{t}) + \sum_{j=1}^{m}  \int_{0}^{T}\mathcal{R}^{k}_{j} (t) \, \di t\,.
\end{equation*}
Hence, passing to the limit as $k \to \infty$ and recalling~\eqref{e:304} we get~\eqref{e:299}.
\end{proof}

\begin{proposition}
\label{p:cont-chaos}
Let $m \in \mathbb{N}$, $(\overline{X}_{0}, \overline{\y}_{0}) \in L^{p} (\Om; \R^{d}) \times (\R^{d})^{m}$. Then there exists a solution to the minimum problem~\eqref{e:min-chaos}.
\end{proposition}

\begin{proof}
The proof follows the argument of Proposition~\ref{p:finite-optimal}, simply replacing the use of Lemma~\ref{l:finite-particles} with Lemma~\ref{l:cont-chaos}.
\end{proof}

\section{Mean-field optimal control}
\label{s:limit}

This section is devoted to the study of the relation between the finite-particle control problems~\eqref{e:min-finite} and~\eqref{e:min-chaos} and the mean-field control problem~\eqref{e:min-mean}. In particular, we aim at showing that~\eqref{e:min-mean} is the mean-field limit of~\eqref{e:min-finite} and~\eqref{e:min-chaos}. This is the content of our main result, Theorem~\ref{t:gamma-lim} below. We point out that the theorem is stated only in terms of the cost functionals~$E$ and~$\mathfrak{E}$, since the propagation of chaos result in Theorem~\ref{t:chaos} already guarantees a $\Gamma$-convergence type of relation between~\eqref{e:min-finite} and~\eqref{e:min-chaos}.

We briefly recall some notation. Given the initial conditions~$(\overline{X}_{0}, \overline{\y}^{m}_{0}) \in L^{p} (\Om; \R^{d}) \times (\R^{d})^{m}$ and the controls $(\uu^{m}, g^{m}) \in L^{1} ([0, T]; K^{m}) \times \G$, to the solution~$(X^{m}, \y^{m})$ to the corresponding system~\eqref{e:chaos} we associate the measures $\bmu^{m} = {\rm Law} (X^{m})$, $\mu^{m}_{t} = ({\rm ev}_{t})_{\#} \bmu^{m}$, $\nu^{m}_{t} = \frac{1}{m} \sum \delta_{y^{m}_{j} (t)}$, and
\begin{equation}
\label{e:rho}
\zeta^{m} \coloneqq \frac{1}{m} \sum_{j=1}^{m} u_{j}^{m} (t) g^{m} (\mu^{m}_{t}) \delta_{y^{m}_{j} (t)} \otimes \di \mathcal{L}^{1} \res [0, T] \quad \in \mathcal{M}_{b} ([0, T]\times \R^{d}; \R^{d})\,. 
\end{equation}
We further define $\overline{\nu}^{m}_{0} \coloneqq \frac{1}{m} \sum_{j=1}^{m} \delta_{\overline{y}^{m}_{0, j}}$.

\begin{theorem}
\label{t:gamma-lim}
Let $q \in (1, +\infty]$ and let $\overline{\mu}_{0} \in \mathcal{P}_{2} (\R^{d})$ be such that $\overline{\mu}_{0}= \overline{\rho}_{0}\, \di x={\rm Law} (\overline{X}_{0})$ for some $\overline{\rho}_{0} \in L^{1} (\R^{d})$ with finite entropy and some $\overline{X}_0\in L^2(\Omega;\R^d)$.
Then the following facts hold:

\noindent \emph{\textbf{$\Gamma$-liminf:}} for every sequence $(\uu^{m}, g^{m}, \overline{\y}^{m}_{0}) \in L^{1}([0, T]; K^{m}) \times \G \times (\R^{d})^{m}$ such that
\begin{equation}
\label{e:4.23}
\sup_{m \in \mathbb{N}} m_{q} (\overline{\nu}^{m}_{0}) < +\infty\,,
\end{equation}
let $(X^{m}, \y^{m}) \in \mathcal{M} (\Om ; C([0, T]; \R^{d})) \times C([0, T]; (\R^{d})^{m})$ be the solution to~\eqref{e:chaos} with controls $(\uu^{m}, g^{m})$ and initial conditions~$(\overline{X}_{0}, \overline{\y}^{m}_{0})$. 
Then there exist $\overline{\nu}_{0} \in \mathcal{P}_{q} ( \R^{d})$ and $(\mu, \nu, \zeta, g) \in \mathcal{S}(\overline{\mu}_{0}, \overline{\nu}_{0})$ such that, up to a not relabeled subsequence, $\overline{\nu}_{0}^{m} \to \overline{\nu}_{0}$ narrow in~$\mathcal{P}(\R^{d})$, $\mu^{m} \to \mu$ and $\nu^{m} \to  \nu$ in $C([0, T]; \mathcal{P}_{1} (\R^{d}))$, $g^{m} \to g$ locally uniformly in $C(\mathcal{P}_{1} (\R^{d}))$, and $\zeta^{m} \to \zeta$ weakly$^{*}$ in $\mathcal{M}_{b} ([0, T] \times \R^{d}; \R^{d})$, as $m\to\infty$. Moreover,
\begin{equation}
\label{gamma_liminf}
E(\mu, \nu, \zeta, g) \leq \liminf_{m \to \infty} \, \mathfrak{E} ( X, \y^{m}, \uu^{m}, g^{m})\,.
\end{equation}

\noindent \emph{\textbf{$\Gamma$-limsup:}}
for every $\overline{\nu}_{0} \in \mathcal{P}_{q} (\R^{d})$, every $(\mu, \nu, \zeta, g) \in \mathcal{S}(\overline{\mu}_{0}, \overline{\nu}_{0})$, and every sequence of initial data $\overline{\y}_0^m \in (\R^{d})^{m}$ such that $\cW_1(\overline{\nu}^{m}_{0},\overline{\nu}_{0})\to0$, there exists a sequence $(\uu^{m}, g^{m} ) \in L^{1}([0, T]; K^{m}) \times \G$ with corresponding solutions $(X^{m}, \y^{m})$ to~\eqref{e:chaos} such that $\mu^{m} \to \mu$, $\nu^{m} \to \nu$ in $C([0, T]; \mathcal{P}_{1} (\R^{d}))$, $\zeta^{m} \to \zeta$ weakly$^{*}$ in $\mathcal{M}_{b} ([0, T]\times \R^{d}; \R^{d})$, and
\begin{equation}
\label{gamma_limsup}
E (\mu, \nu, \zeta, g) \geq  \limsup_{m \to \infty} \, \mathfrak{E} ( X^{m}, \y^{m}, \uu^{m}, g^{m}) \,.
\end{equation} 
\end{theorem}

For the proof of the theorem, we need some preparatory work. We start from the following lemma.

\begin{lemma}
\label{l:control1}
Let $m \in \mathbb{N}$, let $(\uu, g) \in L^{1} ([0, T]; K^{m}) \times \G$, let $\overline{X}_{0} \in L^{p} (\Om; \R^{d})$ and $\overline{\y}^{m}_{0} \in (\R^{d})^{m}$, and let $(X, \y^{m}) \in \mathcal{M} (\Om; C([0, T]; \R^{d})) \times C([0, T]; (\R^{d})^{m})$ be the solution to~\eqref{e:chaos} with initial condition~$(\overline{X}_{0}, \overline{\y}^{m}_{0})$ and controls~$(\uu, g)$. Let us further set\begin{equation}
\label{e:theta}
\theta^{m}_{t} \coloneq \frac{1}{m} \sum_{j=1}^{m} u_{i} (t) \delta_{y_{j}(t)} \,, \qquad \theta^{m} \coloneq \theta^{m}_{t} \otimes \mathcal{L}^{1}\res[0,T] \,.
\end{equation}
If, for $j=1, \ldots, m$, we have 
\begin{equation}\label{e:spieghiamo}
u_{j}(t) = 0 \quad\text{whenever} \quad g(\mu^{m}_{t})= 0,
\end{equation}
then for a.e.~$t \in [0,T]$ it holds
\begin{equation}
\label{e:aux2}
\frac{1}{m} \sum_{j=1}^{m} \phi(u_{j}(t), g(\mu^{m}_{t})) = \int_{\R^{d}} \phi \bigg( \frac{\di \theta^{m}_{t}}{\di \nu^{m}_{t}} (y) , g(\mu^{m}_{t}) \bigg) \, \di \nu^{m}_{t} (y) \,.
\end{equation} 
\end{lemma}
\begin{proof}
The proof of~\eqref{e:aux2}  can be obtained by combining the proof of~\cite[Lemma~6.2]{Flos} with the proof of~\cite[Lemma~1, formula~(38)]{AMOP_2022}.
The only difference is that the map $(t, u) \mapsto \phi (u, g(\mu_{t}^{m}))$ is non-autonomous, as it explicitly depends on time. However, the argument of~\cite[Lemma~6.2]{Flos} does not change, as it works for fixed time~$t \in [0, T]$. 
Also notice that, to conclude the argument, one needs that $u_i(t)=u_j(t)$ whenever $\dot y_i(t)=\dot y_j(t)$ and $y_i(t)=y_j(t)$. This is granted by~\eqref{e:spieghiamo}.
\end{proof}

In the construction of a recovery sequence we will use the following lemma.
\begin{lemma}
\label{l:lsc-F}
Let $\mu, \nu \in C([0, T]; \mathcal{P}_{1} (\R^{d}))$, $f \in L^{1}_{\nu} ([0, T] \times \R^{d}; K)$, and $g \in \G$, be such that $f(t, \cdot) = 0$ if $g(\mu_{t}) = 0$, and let us set
\begin{align}
 \Xi & \coloneq \big\{ \gamma \in C([0, T]; \R^{d}): \, \text{$\dot{\gamma}(t) = w_{\mu_{t}, \nu_{t}} (\gamma(t) ) + f(t, \gamma(t)) g(\mu_{t})$, $\gamma(0) \in \spt(\nu_{0})$} \big\}\,, \label{e:def-Xi}\\
 \mathcal{F} (\gamma) & \coloneq \int_{0}^{T} \phi (f(t, \gamma(t)), g(\mu_{t}) ) \, \di t \qquad \text{$\gamma \in C([0, T]; \R^{d})$}. \label{e:def-F}
\end{align} 
Then $\mathcal{F}$ is lower semicontinuous on~$\Xi$ with respect to the convergence in $C([0, T]; \R^{d})$. Moreover, if $\gamma_{j}, \gamma \in \Xi$ are such that $\gamma_{j} \to \gamma$ in $C([0, T]; \R^{d})$ and $\mathcal{F} (\gamma_{j}) \to \F(\gamma)$, then $f(\cdot, \gamma_{j} (\cdot))  \to f(\cdot, \gamma (\cdot))$ in $L^{p} ([0, T] ; \R^{d})$ for every $p <+\infty$.  
\end{lemma}

\begin{proof}
Let~$\gamma_{j}, \gamma \in \Xi$ be such that $\gamma_{j} \to \gamma$ with respect to the convergence in~$C([0, T]; \R^{d})$. Since~$f$ takes values in the compact set~$K$ we immediately deduce that~$f(\cdot, \gamma_{j}(\cdot))$ is bounded in~$L^{\infty}([0,T]; \R^{d})$, and therefore converges weakly$^*$, up to a subsequence, to some $h \in L^{\infty}([0,T]; \R^{d})$ and, by convexity of~$\widetilde{\phi} (t, \cdot) = \phi (\cdot, g(\mu_{t}))$,
\begin{displaymath}
\int_{0}^{T} \phi(h(t), g(\mu_{t})) \, \di t \leq \liminf_{j\to \infty}  \, \F(\gamma_{j}) \,.
\end{displaymath}
Since~$\gamma_{j} \in \Xi$ for every $j \in \mathbb{N}$, for $s<t \in [0,T]$ we can write
\begin{displaymath}
\gamma_{j}(t) - \gamma_{j}(s) = \int_{s}^{t} \big(  w_{\mu_{\tau}, \nu_{\tau}} (\gamma_{j} (\tau)) + 
f ( \tau , \gamma_{j} (\tau )) g(\mu_{\tau}) \big) \, \di \tau\,.
\end{displaymath}
By~\eqref{e:vLip}, passing to the limit as $j\to \infty$ in the previous equality we get that
\begin{displaymath}
\gamma (t) - \gamma (s) = \int_{s}^{t} \big( w_{\mu_{\tau}, \nu_{\tau}} (\gamma (\tau)) + h(\tau) g(\mu_{\tau}) \big) \, \di \tau\,.
\end{displaymath}
Since~$\gamma \in \Xi$ we have that
\begin{displaymath}
\gamma (t) - \gamma (s) = \int_{s}^{t}  \big( w_{\mu_{\tau}, \nu_{\tau}} (\gamma (\tau)) + f(\tau, \gamma(\tau)) g(\mu_{\tau}) \big) \, \di \tau\,,
\end{displaymath}
which implies, by the arbitrariness of~$s$ and~$t$, that $h(\tau) g(\mu_{\tau}) = f(\tau, \gamma(\tau)) g(\mu_{\tau})$ for a.e.~$\tau \in [0,T]$. Hence, $h(t) = f(t, \gamma(t))$ for a.e.~$t \in \{ s \in [0,T] : \, g(\mu_{s})  \neq 0\}$, while $f(t, \gamma(t)) = 0$ for $t \in \{ s \in [0,T]: \,g(\mu_{s})  = 0 \}$. Since~$\phi \geq 0$ and $\phi (0, r) = 0$ for every $r \in \R$, we finally obtain
\begin{displaymath}
\F(\gamma) \leq \int_{0}^{T} \phi(h(t), g(\mu_{t})) \, \di t \leq \liminf_{j\to \infty} \F(\gamma_{j}) \,.
\end{displaymath}

Since $\phi(\cdot, r)$ is superlinear uniformly with respect to~$r \in \R$, the convergence $\F(\gamma_{j}) \to \F(\gamma)$ implies that $f(\cdot, \gamma_{j} (\cdot)) \to f(\cdot, \gamma(\cdot))$ as $j \to \infty$ in $L^{1} ([0, T]; \R^{d})$, and hence in $L^{p}([0, T]; \R^{d})$ for every $p <+\infty$ by dominated convergence.
\end{proof} 

We are in a position to prove Theorem~\ref{t:gamma-lim}. 

\begin{proof}[Proof of Theorem~\ref{t:gamma-lim}] We divide the proof into two steps.

\noindent{\bf $\Gamma$-liminf:} We may assume that the liminf in~\eqref{gamma_liminf} is a limit and is finite, otherwise there is nothing to show. Noticing that $\nu^{m}$ solves
\begin{displaymath}
\begin{cases}
\partial_{t} \nu^{m}_{t} + \dive (w_{\mu^{m}_{t}, \nu^{m}_{t}}  \nu^{m}_{t}+ \zeta^{m}_{t}) = 0 \,,\\
\nu^{m}_{0} = \overline{\nu}^{m}_{0}\,,
\end{cases}
\end{displaymath} 
Theorem~\ref{t:equivalence} implies that we have that $(\mu^{m}, \nu^{m}, \zeta^{m}, g^{m}) \in \mathcal{S}(\overline{\mu}_{0}, \overline{\nu}_{0}^{m})$ for every $m \in \mathbb{N}$. Thanks to Lemma~\ref{l:control1} we have that
\begin{align}
\label{e:600}
E(\mu^{m}, \nu^{m}, \zeta^{m}, g^{m}) &= \int_{0}^{T} \mathcal{L} (\mu_{t}^{m}, \nu^{m}_{t}) \, \di t + \Phi(\zeta^{m}, \nu^{m}) 
\\
&
 \leq \int_{0}^{T} \mathcal{L} (\mu_{t}^{m}, \nu^{m}_{t}) \, \di t + \int_{0}^{T} \int_{\R^{d}} \phi\bigg( \frac{\di \theta^{m}_{t}}{\di \nu^{m}_{t}} (y) , g^{m} (\mu^{m}_{t}) \bigg)\, \di \nu^{m}_{t} (y) \, \di t  \nonumber
\\
&
= \vphantom{\int_{\Gamma}} \mathfrak{E} (X^{m}, \y^{m}, \uu^{m}, g^{m})\,, \nonumber
\end{align}
where we have set
\begin{displaymath}
\theta^{m}_{t} \coloneqq \frac{1}{m} \sum_{j=1}^{m} u_{j}^{m} (t) \delta_{y^{m}_{j} (t)} \qquad \theta^{m} \coloneqq \theta^{m}_{t} \otimes \mathcal{L}^{1}\res[0, T]\,.
\end{displaymath}
Hence, by Lemma~\ref{l:lsc-cost}
there exists $(\mu, \nu, \zeta, g)\in \mathcal{S}(\overline{\mu}_{0}, \overline{\nu}_{0})$ such that, up to a subsequence, $\mu^{m} \to \mu$ and~$\nu^{m} \to \nu$ in $C([0, T]; \mathcal{P}_{1} (\R^{d}))$, $\zeta^{m} \to \zeta$ weakly$^*$ in $\mathcal{M}_{b} ([0, T]\times \R^{d}; \R^{d})$, and~$g^{m} \to g$ locally uniformly in~$C(\mathcal{P}_{1} (\R^{d}))$.

Since $\zeta^{m} = g^{m} (\mu^{m}_{t}) \theta^{m}$, we deduce from~\eqref{e:600} and from the definition of~$\Phi$ in~\eqref{e:Phi} that
\begin{equation}
\label{e:601}
\Phi(\zeta^{m}, \nu^{m} ) \leq \frac{1}{m} \sum_{j=1}^{m} \int_{0}^{T}\phi(u^{m}_{j} (t) ,  g^{m} (\mu^{m}_{t}) ) \,.
\end{equation}
Hence, Lemma~\ref{l:lsc-cost} and~\eqref{e:601} yield that
\begin{equation}
\label{e:602}
\Phi(\zeta, \nu) \leq \liminf_{m \to \infty} \,\Phi(\zeta^{m}, \nu^{m} )\leq  \liminf_{m \to \infty}  \frac{1}{m} \sum_{j=1}^{m} \int_{0}^{T}\phi(u^{m}_{j} (t), g^{m} (\mu^{m}_{t}) )\,.
\end{equation}
From the continuity of the Lagrangian cost and from~\eqref{e:602} we deduce~\eqref{gamma_liminf}.

\noindent{\bf $\Gamma$-limsup:}
From now on, we denote by~$C$ any positive constant independent of~$m$, which may vary from line to line.

 Let~$(\mu, \nu, \zeta, g) \in \mathcal{S}(\overline{\mu}_{0}, \overline{\nu}_{0})$ and let $f \in L^{1}_{\nu}([0,T]\times \R^{d}; K)$ be such that $\zeta= f g(\mu_{t}) \nu$ and 
 \begin{displaymath}
 \Phi(\zeta, \nu) = \int_{[0, T] \times \R^{d}} \phi( f(t, y) , g(\mu_{t}) ) \, \di \nu(t, y) \,.
 \end{displaymath}
Since $\phi(0, \xi) = 0$ for every $\xi \in \R$ (see~$(\phi2)$), we may assume that $f (t, \cdot)= 0$ whenever $ g(\mu_{t}) = 0$. We recall that~$\nu$ solves the continuity equation
\begin{equation}
\label{e:cont-nu}
\begin{cases}
\partial_{t} \nu_{t} + \dive \big( (w_{\mu_{t}, \nu_{t}} + f (t, \cdot) g(\mu_{t}) ) \nu_{t} \big) = 0 \\
 \nu_{0} = \overline{\nu}_{0}\,,
\end{cases}
\end{equation}
while, by Theorem~\ref{t:equivalence}, we can write $\mu_{t} =(\ev_{t})_{\#}\bmu$ with~$\bmu = {\rm Law} (X)$ and~$X \in \mathcal{M} (\Om; C([0, T]; \R^{d}))$ being the unique solution to
\begin{equation}
\label{e:X-mu}
\begin{cases}
\di X(t) = v_{\mu_{t}, \nu_{t}} (X(t)) \, \di t + \sqrt{2\sigma} \, \di W(t)\,,\\
 X(0) = \overline{X}_{0}\,.
 \end{cases}
\end{equation}

Since $\overline{\nu}_{0} \in \mathcal{P}_{q} (\R^{d})$, by Lemma~\ref{l:lemma3.2} there exists $R = R(\overline{\mu}_{0}, \overline{\nu}_{0} , \Lambda, \Delta, T, K)>0$ such that 
\begin{displaymath}
\max_{t \in [0, T]}\, m_{2} (\mu_{t}) + m_{q} (\nu_{t}) \leq R\,.
\end{displaymath} 
This, together with the boundedness of~$g$ and of~$f$, allows us to apply the superposition principle~\cite[Section~8.2]{AmbGigSav08} to the continuity equation~\eqref{e:cont-nu}. Hence, defining~$\Xi$ as in~\eqref{e:def-Xi} and setting for brevity $\Gamma \coloneq C([0,T]; \R^{d})$,
there exists $\eeta \in \mathcal{P}(\Gamma)$ supported on~$\Xi$ such that $\nu_{t} = ({\rm ev}_{t})_{\#} \eeta$ for every $t \in [0, T]$. We further define $\F \colon \Gamma \to [0, +\infty)$ as in~\eqref{e:def-F}.
%
In particular, we notice that
\begin{align}
\label{e:799}
\int_{\Gamma} \F(\gamma) \, \di \eeta(\gamma) & = \int_{\Gamma} \int_{0}^{T} \phi(f(t, \gamma(t)), g(\mu_{t})) \, \di t \, \di \eeta(\gamma) 
\\
&
=  \int_{0}^{T}\int_{\Gamma} \phi(f(t, \ev_{t}(\gamma)), g(\mu_{t})) \, \di \eeta(\gamma)\, \di t \nonumber
\\ 
&
=  \int_{0}^{T}\int_{\R^{d}}  \phi(f(t,y) , g(\mu_{t}) ) \, \di \nu_{t}(y) \, \di t = \Phi  (\zeta, \nu)\,. \nonumber
\end{align}
By Lemma~\ref{l:lsc-F} we have that $\F$ is lower-semicontinuous in~$\Xi$. By Lusin approximation, we find an increasing sequence of compact subsets $\Xi_{k} \Subset \Xi_{k+1} \Subset \Xi$ such that $\F$ is continuous on~$\Xi_{k}$ for every $k \in \mathbb{N}$ and $\eeta (\Xi \setminus \Xi_{k}) \to 0$ as $k \to \infty$. We set
\begin{displaymath}
\overline{\eeta}_{k} \coloneqq \frac{1}{\eeta(\Xi_{k})} \, \eeta \res \Xi_{k} \quad \in \mathcal{P}(\Gamma)\,,
\end{displaymath}
which satisfies 
\begin{equation}
\label{e:800}
\lim_{k\to \infty}\, \cW_{1} (\overline{\eeta}_{k}, \eeta) = 0 \qquad \text{and} \qquad \lim_{k \to \infty} \int_{\Gamma} \F(\gamma) \, \di \overline{\eeta}_{k} (\gamma) = \int_{\Gamma} \F(\gamma)\, \di \eeta(\gamma)\,.
\end{equation}

Given an at most countable set~$D = \{ \varphi_{\ell}\}_{\ell \in \mathbb{N}}$ dense in~$C_{c} ([0, T]\times \R^{d}; \R^{d})$, reasoning as in~\cite[formulas (52)--(56)]{AMOP_2022} we can construct a strictly increasing sequence~$n(k) \in \mathbb{N}$ and $\{\overline{\eeta}^{n}_{k}\}_{n} \in \mathcal{P}(\Gamma)$ such that $\overline{\eeta}^{n}_{k} = \frac{1}{n} \sum_{j=1}^{n} \delta_{\gamma^{j}_{k}} $ for suitable $y^{j}_{k} \in \Xi$ and for every $k$ and every $n \geq n(k)$ it holds
\begin{align}
& \label{e:801}  \vphantom{\int_{\Gamma}}  \cW_{1} (\overline{\eeta}^{n}_{k}, \overline{\eeta}_{k} ) \leq \frac{1}{k}\,,\\
& \label{e:802}\bigg|  \int_{\Gamma} \F(\gamma) \, \di \overline{\eeta}^{n}_{k} (\gamma) -  \int_{\Gamma} \F(\gamma) \, \di \overline{\eeta}_{k} (\gamma) \bigg| \leq \frac{1}{k}\,,\\
& \label{e:803} \bigg|  \int_{\Gamma} \int_{0}^{T} \varphi_{\ell} (t, \gamma(t)) \cdot f(t, \gamma(t)) g(\mu_{t}) \, \di t \, \di (\overline{\eeta}^{n}_{k} - \overline{\eeta}_{k} ) \bigg|  \leq \frac{1}{k} \qquad \text{for $\ell \leq k$}\,.
\end{align}
Then, for every $m \in [n(k), n(k+1))$ we set $\eeta^{m} \coloneqq \overline{\eeta}^{m}_{k}$. From~\eqref{e:799}--\eqref{e:802} it follows that
\begin{align}
&\label{e:804}\vphantom{\int_{\Gamma}} \lim_{m\to \infty} \, \cW_{1} (\eeta^{m}, \eeta) = 0 \,,\\
&\label{e:805} \lim_{m\to \infty} \int_{\Gamma} \F(\gamma) \, \di \eeta^{m}(\gamma) = \Phi(\zeta, \nu)\,.
\end{align}

We now construct the controls~$(\uu^{m}, g^{m}) \in L^{1} ([0, T]; K^{m}) \times \G$. First, we simply set~$g^{m}\coloneq g$ for every $m$. We introduce the auxiliary curves of measures $\lambda^{m}_{t} \coloneqq (\ev_{t})_{\#} \eeta^{m}$ for $t \in [0, T]$ and denote by $\z^{m} = (z_{1}^{m}, \ldots, z^{m}_{m}) \in \Xi$ the curves on which~$\eeta^{m}$ is concentrated, so that $\lambda^{m}_{t} = \frac{1}{m} \sum_{j} \delta_{z^{m}_{j} (t)}$. Then, we set $u^{m}_{j}(t) \coloneqq f(t, z^{m}_{j} (t))$. In particular, we notice that $z^{m}_{j}$ solves
\begin{equation}
\label{e:806}
\dot{z}_{j}^{m} (t) = w_{\mu_{t}, \nu_{t}} (z^{m}_{j}(t)) + u_{j}^{m} (t) g(\mu_{t}) \,.
\end{equation}
with initial condition~$z_{j}^{m}(0) \in \spt( \overline{\nu}_{0})$. Hence, we have to modify~$\z^{m}$ and~$X$ since the ODEs~\eqref{e:806} and the SDE~\eqref{e:X-mu} still account for the limit curves~$\mu$ and~$\nu$. For later convenience, we further notice that, by construction of~$\lambda^{m}$ and by~\eqref{e:804}, we have that
\begin{equation}
\label{e:807}
\lim_{m \to \infty} \, \cW_{1} (\lambda^{m}_{t}, \nu_{t}) = 0 \qquad \text{uniformly in~$[0, T]$}.
\end{equation}
Moreover, setting~$\alpha^{m} \coloneqq \frac{1}{m} \sum_{j=1}^{m} u^{m}_{j} (t) g(\mu_{t}) \delta_{z^{m}_{j} (t)} \otimes \mathcal{L}^{1} \res[0, T]$, it holds
\begin{equation}
\label{e:808}
\alpha^{m} \longrightarrow \zeta \qquad \text{weakly$^*$ in $\mathcal{M}_{b} ([0, T]\times \R^{d}; \R^{d})$.}
\end{equation}
Indeed, for every $\varphi \in C_{c} ([0, T]\times \R^{d}; \R^{d})$ and every $\varepsilon>0$, we fix~$\ell \in \mathbb{N}$ such that $\| \varphi - \varphi_{\ell}\|_{\infty} \leq \varepsilon$. Then, by a direct computation and by definition of~$\alpha^{m}$ and of~$\eeta^{m}$ we have that
\begin{align}
\label{e:808.1}
& \bigg| \int_{0}^{T} \int_{\R^{d}}  \varphi(t, y) \, \di (\alpha^{m} - \zeta) (t, y) \bigg| 
\\
&
 \leq  \int_{0}^{T} \int_{\R^{d}}| \varphi(t, y) - \varphi_{\ell}(t, y) |  \, \di |\alpha^{m} - \zeta | (t, y)  + \bigg| \int_{0}^{T} \int_{\R^{d}} \varphi_{\ell}(t, y) \, \di (\alpha^{m} - \zeta) (t, y) \bigg| \nonumber
\\
&
\leq C \varepsilon + \bigg|  \frac{1}{m} \sum_{j=1}^{m} \int_{0}^{T} \varphi_{\ell}(t, z_{j}^{m}(t)) \cdot f(t, z_{j}^{m}(t)) g ( \mu_{t})  \, \di t \nonumber
- \int_{0}^{T}\int_{\R^{d}} \varphi_{\ell}(t, y) f(t, y)   g ( \mu_{t})  \, \di \nu_{t}(y) \, \di t \bigg| \nonumber
\\
&
= C \varepsilon + \bigg|  \int_{0}^{T} \int_{\R^{d}} \varphi_{\ell} (t, y) \cdot f(t, y) g ( \mu_{t})  \, \di \lambda^{m}_{t}(y) \,\di t \nonumber
- \int_{0}^{T}\int_{\R^{d}} \varphi_{\ell}(t, y) \cdot f(t, y)   g ( \mu_{t})  \, \di \nu_{t}(y) \, \di t \bigg| \nonumber
\\
&
= C \varepsilon + \bigg| \int_{\Gamma} \int_{0}^{T} \varphi_{\ell}(t, \gamma (t)) \cdot f(t, \gamma (t)) g ( \mu_{t})  \, \di t \, \di \eeta^{m}(\gamma) \nonumber
- \int_{\Gamma} \int_{0}^{T} \varphi_{\ell}(t, \gamma (t)) \cdot f(t, \gamma (t))g ( \mu_{t})  \, \di t \, \di \eeta (\gamma)  \bigg| .\nonumber
\end{align}
For $m \in [n(k), n(k+1))$ we recall the definition $\eeta^{m}= \overline{\eeta}^{n(k)}_{k}$ and continue in~\eqref{e:808.1} by triangle inequality with
\begin{align}
\label{e:808.2}
& \bigg| \int_{0}^{T} \int_{\R^{d}}  \varphi(t, y) \, \di (\alpha^{m} - \zeta) (t, y) \bigg| 
\\
&
= C \varepsilon + \bigg| \int_{\Gamma} \int_{0}^{T} \varphi_{\ell}(t, \gamma (t)) \cdot f(t, \gamma (t)) g ( \mu_{t})  \, \di t \, \di (\overline{\eeta}^{n(k)}_{k} - \overline{\eeta}_{k}) (\gamma) \nonumber
\\
&
\qquad \,\,\, + \int_{\Gamma} \int_{0}^{T} \varphi_{\ell}(t, \gamma (t)) \cdot f(t, \gamma (t))g ( \mu_{t})  \, \di t \, \di (  \overline{\eeta}_{k} - \eeta)  (\gamma)  \bigg| \nonumber
\\
&
\leq C \varepsilon + \frac{1}{k} + \bigg| \int_{\Gamma} \int_{0}^{T} \varphi_{\ell}(t, \gamma (t)) \cdot f(t, \gamma (t))g ( \mu_{t})  \, \di t \, \di (  \overline{\eeta}_{k} - \eeta)  (\gamma)  \bigg| \nonumber\,,
\end{align}
where, in the last inequality, we have used~\eqref{e:803}. Passing to the limit in~\eqref{e:808.2} as $k \to \infty$ we deduce~\eqref{e:808} from the arbitrariness of~$\varepsilon$.

Since the cost functional~$\mathfrak{E}$ in~\eqref{e:cost-chaos} is invariant under permutations of controls and trajectories, we may assume that
\begin{equation}
\label{e:808.3}
\cW_{1} (\overline{\nu}^{m}_{0}, \lambda^{m}_{0}) = \frac{1}{m} \sum_{j=1}^{m} | \overline{y}^{m}_{0, j} - z^{m}_{j} (0) |\,.
\end{equation}
We define $(X^{m}, \y^{m} ) \in \mathcal{M} (\Om; C([0, T]; \R^{d})) \times C([0, T]; (\R^{d})^{m})$ as the unique solution to~\eqref{e:chaos} with controls~$(\uu^{m}, g^{m})$ and initial data~$(\overline{X}_{0}, \overline{\y}^{m}_{0})$. We finally recall the definition of~$\zeta^{m}$ in~\eqref{e:rho}.

We claim that $(X^{m}, \y^{m}, \uu^{m}, g^{m})$ is a recovery sequence for~$(\mu, \nu, \zeta, g)$. To this purpose, we first show the convergences
\begin{align}
& \label{e:809} \lim_{m \to \infty} \, \cW_{1} (\nu^{m}_{t}, \nu_{t}) = 0  \qquad \text{uniformly in~$[0, T]$},\\
& \label{e:810} \lim_{m \to \infty}\, \cW_{1} (\mu^{m}_{t}, \mu_{t}) = 0 \qquad \text{uniformly in~$[0, T]$},\\ 
& \label{e:811} \zeta^{m} \longrightarrow \zeta \qquad \text{weakly$^*$ in $\mathcal{M}_{b} ([0, T]\times \R^{d}; \R^{d})$}.
\end{align}
In view of~\eqref{e:807}, we notice that to conclude for~\eqref{e:809} it is enough to prove that
\begin{equation}
\label{e:812.1}
\lim_{m\to \infty} \, \cW_{1} (\lambda^{m}_{t}, \nu^{m}_{t}) = 0  \qquad \text{uniformly in~$[0, T]$}.
\end{equation}
We start by estimating the distance between the single trajectories $y^{m}_{j}$ and~$z^{m}_{j}$. By~\eqref{e:vLip}, by the definition of~$\G$, and by triangle inequality we have that
\begin{align}
\label{e:812}
|y^{m}_{j}(t) - z^{m}_{j} (t)|&  \leq | \overline{y}^{m}_{0, j} - z^{m}_{j}(0)| + \int_{0}^{t} | w_{\mu^{m}_{\tau}, \nu^{m}_{\tau}} (y^{m}_{j} (\tau) ) - w_{\mu_{\tau}, \nu_{\tau}} (z^{m}_{j} (\tau)) | \, \di \tau 
\\
&
\qquad + \int_{0}^{\tau} |u^{m}_{j} (\tau) | \, | g (\mu^{m}_{\tau} ) - g (\mu_{\tau})| \, \di \tau \nonumber
\\
&
\nonumber \leq | \overline{y}^{m}_{0, j} - z^{m}_{j}(0)| + C \int_{0}^{t} \Big( \cW_{1} (\mu^{m}_{\tau}, \mu_{\tau}) + \cW_{1} (\nu^{m}_{\tau}, \nu_{\tau})  + |y^{m}_{j} (\tau) - z^{m}_{j} (\tau) | \Big) \, \di \tau  \nonumber
\\
&
\nonumber \leq  | \overline{y}^{m}_{0, j} - z^{m}_{j}(0)| + C \int_{0}^{t} \Big( \cW_{1} (\mu^{m}_{\tau}, \mu_{\tau})  + \cW_{1} (\lambda^{m}_{\tau}, \nu_{\tau}) 
\\
&
\qquad+\cW_{1} (\nu^{m}_{\tau}, \lambda^{m}_{\tau}) + |y^{m}_{j} (\tau) - z^{m}_{j} (\tau) | \Big) \, \di \tau  \nonumber
\\
&
\leq  | \overline{y}^{m}_{0, j} - z^{m}_{j}(0)| + C \int_{0}^{t} \Big( \cW_{1} (\mu^{m}_{\tau}, \mu_{\tau})  + \cW_{1} (\lambda^{m}_{\tau}, \nu_{\tau}) \nonumber
\\
&
\qquad + \frac{1}{m} \sum_{\ell=1}^{m} |y^{m}_{\ell} (\tau) - z^{m}_{\ell} (\tau) | + |y^{m}_{j} (\tau) - z^{m}_{j} (\tau) | \Big) \, \di \tau \nonumber \,.
\end{align}
Summing up over $j=1, \ldots, m$ and applying Gr\"onwall inequality we deduce from~\eqref{e:812} that
\begin{align}
\label{e:814}
\cW_{1} (\nu^{m}_{t}, \lambda^{m}_{t}) & \leq \frac{1}{m} \sum_{j=1}^{m}  |y^{m}_{j}(t) - z^{m}_{j} (t)| 
\\
&
\leq  C \bigg( \cW_{1} (\overline{\nu}^{m}_{0}, \lambda^{m}_{0}) + \int_{0}^{t}  \big( \cW_{1} (\mu^{m}_{\tau}, \mu_{\tau}) +  \cW_{1} (\lambda^{m}_{\tau}, \nu_{\tau}) \big)  \di \tau \bigg) \nonumber
\\
&
\leq C \bigg( \cW_{1} (\overline{\nu}^{m}_{0}, \lambda^{m}_{0}) + \int_{0}^{t}  \big( \mathbb{E} (| X^{m} (\tau) - X(\tau)| ) +  \cW_{1} (\lambda^{m}_{\tau}, \nu_{\tau}) \big)  \di \tau \bigg) \nonumber .
\end{align}
where, in the last inequality, we have used~\eqref{e:muu} together with the equalities~$\bmu^{m} = {\rm Law} (X^{m})$ and~$\bmu = {\rm Law} (X)$. By Proposition~\ref{p:12} (see~\eqref{eee}), we may further estimate for every $t \in [0, T]$
\begin{align}
& \label{e:816} \mathbb{E} ( | X^{m} (t) - X(t)| ) \leq C\int_{0}^{t} \cW_{1} (\nu^{m}_{\tau}, \nu_{\tau}) \, \di \tau \leq  C\int_{0}^{t}  \big(\cW_{1} (\nu^{m}_{\tau}, \lambda^{m}_{\tau}) + \cW_{1} (\lambda^{m}_{\tau},  \nu_{\tau})  \big)\, \di \tau\,.
\end{align}
Combining~\eqref{e:814} and~\eqref{e:816} we get
\begin{align}
\label{e:817}
\cW_{1} (\nu^{m}_{t}, \lambda^{m}_{t}) & \leq \frac{1}{m} \sum_{j=1}^{m}  |y^{m}_{j}(t) - z^{m}_{j} (t)| 
\\
&
\leq  C \bigg( \cW_{1} (\overline{\nu}^{m}_{0}, \lambda^{m}_{0}) + \int_{0}^{t}  \big( \cW_{1} (\nu^{m}_{\tau}, \lambda^{m}_{\tau})  +  \cW_{1} (\lambda^{m}_{\tau}, \nu_{\tau}) \big)  \di \tau \bigg) \nonumber .
\end{align}
Relying once again on Gr\"onwall inequality, we infer from~\eqref{e:817} that
\begin{align}
\label{e:818}
\cW_{1} (\nu^{m}_{t}, \lambda^{m}_{t}) & \leq C \bigg( \cW_{1} (\overline{\nu}^{m}_{0}, \lambda^{m}_{0}) + \int_{0}^{t}   \cW_{1} (\lambda^{m}_{\tau}, \nu_{\tau}) \, \di \tau \bigg) .
\end{align}
Since~\eqref{e:807} holds, inequality~\eqref{e:818} yields~\eqref{e:812.1} and thus~\eqref{e:809}. Finally,~\eqref{e:809} and~\eqref{e:816} imply~\eqref{e:810}. We further notice that combining~\eqref{e:809},~\eqref{e:812}, and~\eqref{e:817}, we deduce the auxiliary uniform limit
\begin{align}
\label{e:819}
\lim_{m \to \infty} \, \frac{1}{m} \sum_{j=1}^{m} | y^{m}_{j} (t) - z^{m}_{j} (t) | = 0 \qquad \text{uniformly in~$[0, T]$.}
\end{align}

We now show that
\begin{align}
\label{e:820}
\zeta^{m} - \alpha^{m} \longrightarrow 0 \qquad \text{weakly$^*$ in~$\mathcal{M}_{b} ([0, T]\times \R^{d}; \R^{d})$}.
\end{align}
We notice that~\eqref{e:820}, together with~\eqref{e:808}, implies~\eqref{e:811}. For every $\varphi \in C_{c}([0,T]\times  \R^{d} ; \R^{d})$, using the definition of~$\zeta^{m}$, of~$\alpha^{m}$, and of the controls~$\uu^{m}$, we have that, 
\begin{align}
\label{e:limsup131}
 \bigg|\int_{0}^{T}  \int_{ \R^{d}} & \varphi (t, y) \, \di \zeta^{m}(t, y) - \int_{0}^{T} \int_{ \R^{d}} \varphi(t, y) \, \di \alpha^{m}(t, y) \bigg|
\\
= &\,  \bigg| \frac{1}{m} \sum_{j=1}^{m} \int_{0}^{T} \big( \varphi(t, y^{m}_{j}(t))  g(\mu^{m}_{t})  - \varphi(t, z^{m}_{j}(t)) g(\mu_{t}) \big) \, f(t, z^{m}_{j}(t)) \, \di t \bigg| \nonumber
\\
\leq &\,
  \frac{1}{m} \sum_{j=1}^{m} \int_{0}^{T} \big| \varphi(t, {y}^{m}_{j}(t)) - \varphi(t, z^{m}_{j}(t)) \big| \cdot  \big| g(\mu^{m}_{t}) f(t, z^{m}_{j} (t)) \big| \, \di t \nonumber
\\
&
\, +  \frac{1}{m} \sum_{j=1}^{m} \int_{0}^{T} \big| g(\mu^{m}_{t}) - g(\mu_{t}) \big| \cdot \big|  \varphi(t, z^{m}_{j} (t)) f(t, z^{m}_{j}(t)) \big| \, \di t \,.\nonumber
\end{align}
Relying on the $\Lambda$-Lipschitz continuity of~$g$, on the boundedness of~$f$, and on the uniform continuity of~$\varphi$, we can continue in~\eqref{e:limsup131} with
\begin{align}
\label{e:900}
\bigg|\int_{0}^{T}  \int_{ \R^{d}} & \varphi (t, y) \, \di \zeta^{m}(t, y) - \int_{0}^{T} \int_{ \R^{d}} \varphi(t, y) \, \di \alpha^{m}(t, y) \bigg|
\\
&
 \leq  C \sum_{j=1}^{m} \frac{1}{m} \int_{0}^{T} \omega_{\varphi} \big( | y^{m}_{j} (t) - z^{m}_{j} (t)| \big) \, \di t + C \| \varphi\|_{\infty} \int_{0}^{T} \cW_{1} (\mu^{m}_{t} , \mu_{t}) \, \di t \nonumber
 \\
 &
 \leq C  \int_{0}^{T} \omega_{\varphi} \bigg( \frac{1}{m} \sum_{j=1}^{m} | y^{m}_{j} (t) - z^{m}_{j} (t)| \bigg) \, \di t + C \| \varphi\|_{\infty} \int_{0}^{T} \cW_{1} (\mu^{m}_{t} , \mu_{t}) \, \di t\,, \nonumber
\end{align} 
where~$\omega_{\varphi}$ denotes a concave modulus of continuity of~$\varphi$. By~\eqref{e:810} and~\eqref{e:819}, we can pass to the limit as~$m \to \infty$ in~\eqref{e:limsup131} and deduce~\eqref{e:820}.

We are left to show that~\eqref{gamma_limsup} holds. In view of~\eqref{e:809} and~\eqref{e:810} and of the continuity of the Lagrangian cost, we have that
\begin{equation}
\label{e:901}
\lim_{m \to \infty} \int_{0}^{T} \mathcal{L} (\mu^{m}_{t}, \nu^{m}_{t}) \, \di t = \int_{0}^{T} \mathcal{L} (\mu_{t}, \nu_{t}) \, \di t
\end{equation}
As for the control part of the cost~$\mathfrak{E}$, denoting by~$\omega_{\phi}$ a modulus of continuity of~$\phi$ on~$K \times [-\Delta, \Delta]$ and recalling that $g$ is $\Lambda$-Lipschitz continuous, we estimate
\begin{align}
\label{e:902}
\frac{1}{m} \sum_{j=1}^{m} \int_{0}^{T} \!\!\!\phi(u^{m}_{j} (t), g (\mu^{m}_{t}) ) \, \di t & \leq \frac{1}{m} \sum_{j=1}^{m} \int_{0}^{T}\!\! \phi(u^{m}_{j} (t), g (\mu_{t})) \, \di t + \int_{0}^{T} \!\! \omega_{\phi} \big(\Lambda \cW_{1} (\mu^{m}_{t}, \mu_{t}) \big) \, \di t\,.
\end{align} 
By definition of the controls~$u^{m}_{j} (t)$ and of the measures~$\lambda^{m}_{t}$ and~$\eeta^{m}$, we may continue in~\eqref{e:902} with
\begin{align}
\label{e:903}
\frac{1}{m}  \sum_{j=1}^{m} \int_{0}^{T}  & \phi(u^{m}_{j} (t), g (\mu^{m}_{t}) ) \, \di t 
\\
&
\leq \frac{1}{m} \sum_{j=1}^{m} \int_{0}^{T} \phi( f(t, z^{m}_{j} (t)) , g (\mu_{t}) ) \, \di t +  \int_{0}^{T} \omega_{\phi} \big(\Lambda \cW_{1} (\mu^{m}_{t}, \mu_{t})  \big) \, \di t \nonumber
\\
&
= \int_{0}^{T} \int_{\R^{d}} \phi(f(t, y),  g(\mu_{t}))  \,\di \lambda^{m}_{t} (y)  \di t +  \int_{0}^{T} \omega_{\phi} \big( \Lambda \cW_{1} (\mu^{m}_{t}, \mu_{t})  \big)\, \di t \nonumber
\\
&
= \int_{\Gamma} \int_{0}^{T} \phi(f(t, \gamma), g(\mu_{t}) ) \, \di t \, \di \eeta^{m} (\gamma) +  \int_{0}^{T} \omega_{\phi} \big(\Lambda \cW_{1} (\mu^{m}_{t}, \mu_{t}) \big) \, \di t \nonumber
\\
&
= \int_{\Gamma} \F(\gamma) \, \di \eeta^{m} (\gamma) +  \int_{0}^{T} \omega_{\phi} \big(\Lambda \cW_{1} (\mu^{m}_{t}, \mu_{t}) \big) \, \di t \nonumber\,.
\end{align}
Thanks to~\eqref{e:805}, to~\eqref{e:809}, and to the choice $g^{m}=g$, we pass to the limsup in~\eqref{e:903} and infer that
\begin{equation}
\label{e:904}
\limsup_{m \to \infty} \, \frac{1}{m} \sum_{j=1}^{m} \int_{0}^{T} \phi(u^{m}_{j} (t), g^{m}(\mu^{m}_{t}) ) \, \di t  \leq \Phi(\zeta, \nu)\,.
\end{equation}
Combining~\eqref{e:901} and~\eqref{e:904} we infer~\eqref{gamma_limsup}. This concludes the proof of the theorem.
\end{proof}

As a consequence of Theorem~\ref{t:gamma-lim} we have the following results on the convergence of minima and minimizers of the control problems~\eqref{e:min-mean} and~\eqref{e:min-chaos}.

\begin{corollary}
\label{c:minima}
Let $q \in (1, +\infty]$, $\overline{\nu}_{0} \in \mathcal{P}_{q} (\R^{d})$, and~$\overline{\mu}_{0} \in \mathcal{P}_{2} (\R^{d})$ be such that $\overline{\mu}_{0} = \overline{\rho}_{0}\, \di x = {\rm Law} (\overline{X}_{0})  $ for some~$\overline{\rho}_{0} \in L^{1} (\R^{d})$ with finite entropy and some $\overline{X}_{0} \in L^{2} (\Om; \R^{d})$. For $m \in \mathbb{N}$, let $\overline{\y}^{m}_{0} \in (\R^{d})^{m}$ be such that~\eqref{e:4.23} is satisfied and~$\overline{\nu}^{m}_{0}\to \overline{\nu}_{0}$ narrow in~$\mathcal{P}(\R^{d})$. Then, for every sequence~$(X^{m}, \y^{m}, \uu^{m}, g^{m}) \in \mathcal{M}(\Om; C([0, T]; \R^{d})) \times C([0, T]; (\R^{d})^{m}) \times L^{1}([0, T]; K^{m}) \times \G$ of solutions to~\eqref{e:min-chaos}, there exists $(\mu, \nu, \zeta, g) \in \mathcal{S}(\overline{\mu}_{0}, \overline{\nu}_{0})$ solution to~\eqref{e:min-mean} such that, up to a subsequence, $\mu^{m} \to \mu$ and $\nu^{m} \to \nu$ in $C([0, T]; \mathcal{P}_{1} (\R^{d}))$, $\zeta^{m} \to \zeta$ weakly$^{*}$ in $\mathcal{M}_{b} ([0, T]\times \R^{d}; \R^{d})$,~$g^{m} \to g$ locally uniformly in~$C(\mathcal{P}_{1} (\R^{d}))$, and
\begin{displaymath}
E(\mu, \nu, \zeta, g) = \lim_{m \to \infty} \, \mathfrak{E} (X^{m}, \y^{m}, \uu^{m}, g^{m})\,.
\end{displaymath}
\end{corollary}

\begin{proof}
The thesis follows by standard arguments of $\Gamma$-convergence, invoking the compactness and convergence results of Theorem~\ref{t:gamma-lim}.
\end{proof}

\begin{corollary}\label{c:5.5}
Let $q \in (1, +\infty]$, $\overline{\nu}_{0} \in \mathcal{P}_{q} (\R^{d})$, and~$\overline{\mu}_{0} \in \mathcal{P}_{2} (\R^{d})$ be such that $\overline{\mu}_{0} = \overline{\rho}_{0}\, \di x = {\rm Law} (\overline{X}_{0})  $ for some~$\overline{\rho}_{0} \in L^{1} (\R^{d})$ with finite entropy and some $\overline{X}_{0} \in L^{2} (\Om; \R^{d})$. For $m \in \mathbb{N}$, let $\overline{\y}^{m}_{0} \in (\R^{d})^{m}$ be such that~\eqref{e:4.23} is satisfied and~$\overline{\nu}^{m}_{0}\to \overline{\nu}_{0}$ narrow in~$\mathcal{P}(\R^{d})$. Then
\begin{equation*}
\begin{split}
\min &\,  \big\{ E(\mu, \nu, \zeta, g): \, (\mu, \nu, \zeta, g) \in \mathcal{S} (\overline{\mu}_{0}, \overline{\nu}_{0}) \}\\
=&\, \lim_{m\to\infty} \min\bigg\{ \mathfrak{E}(X,  \y, \uu, g) : (\uu,g) \in L^{1}([0, T]; K^{m}) \times \mathcal{G}, \, (X, \y) \text{ solves~\eqref{e:chaos}}  \bigg\}\,.
\end{split}
\end{equation*}
\end{corollary}
\begin{proof}
The proof is an immediate consequence of Corollary~\ref{c:minima}.
\end{proof}

\bigskip
\noindent\textbf{Acknowledgments}
The work of SA was partially funded by the Austrian Science Fund through the projects ESP-61 and P-35359.
The work of MM was partially supported by the \emph{Starting grant per giovani ricercatori} of Politecnico di Torino, by the MIUR grant Dipartimenti di Eccellenza 2018-2022 (E11G18000350001), and by the PRIN 2020 project \emph{Mathematics for industry 4.0 (Math4I4)} (2020F3NCPX) financed by the Italian Ministry of University and Research.
The work of FS was partially supported by the project \emph{Variational methods for stationary and evolution problems with singularities and interfaces} PRIN 2017 (2017BTM7SN) financed by the Italian Ministry of Education, University, and Research and by the project Starplus 2020 Unina Linea~1 \emph{New challenges in the variational modeling of continuum mechanics} from the University of Naples ``Federico II'' and Compagnia di San Paolo (CUP: E65F20001630003).
MM and FS are members of the GNAMPA group of INdAM. This research fits within the scopes of the GNAMPA 2022 Project \emph{Approccio multiscala all’analisi di modelli di interazione}.
Finally, the authors acknowledge the warm hospitality of ESI, Vienna during the workshop \emph{Between Regularity and Defects: Variational and Geometrical Methods in Materials Science}, where part of this research was carried out.

\bibliographystyle{siam}
\bibliography{Stochastic_AMS.bib}

\begin{thebibliography}{10}

\bibitem{AMOP_2022}
{\sc G.~Albi, S.~Almi, M.~Morandotti, and F.~Solombrino}, {\em Mean-field
  selective optimal control via transient leadership}, Applied Math. {\&}
  Optim., 85 (2022), p.~22.

\bibitem{albi2016invisible}
{\sc G.~Albi, M.~Bongini, E.~Cristiani, and D.~Kalise}, {\em Invisible control
  of self-organizing agents leaving unknown environments}, SIAM Journal on
  Applied Mathematics, 76 (2016), pp.~1683--1710.

\bibitem{albi2}
{\sc G.~Albi, Y.-P. Choi, M.~Fornasier, and D.~Kalise}, {\em Mean field control
  hierarchy}, Appl. Math. Optim., 76 (2017), pp.~93--135.

\bibitem{AlmMorSol21}
{\sc S.~Almi, M.~Morandotti, and F.~Solombrino}, {\em A multi-step {L}agrangian
  scheme for spatially inhomogeneous evolutionary games}, J. Evol. Equ., 21
  (2021), pp.~2691--2733.

\bibitem{AmbForMorSav18}
{\sc L.~Ambrosio, M.~Fornasier, M.~Morandotti, and G.~Savar\'e}, {\em Spatially
  inhomogeneous evolutionary games}, Comm. Pure Appl. Math., 74 (2021),
  pp.~1353--1402.

\bibitem{AmbGigSav08}
{\sc L.~Ambrosio, N.~Gigli, and G.~Savar\'{e}}, {\em Gradient flows in metric
  spaces and in the space of probability measures}, Lectures in Mathematics ETH
  Z\"{u}rich, Birkh\"{a}user Verlag, Basel, second~ed., 2008.

\bibitem{AmbTre14}
{\sc L.~Ambrosio and D.~Trevisan}, {\em Well-posedness of {L}agrangian flows
  and continuity equations in metric measure spaces}, Anal. PDE, 7 (2014),
  pp.~1179--1234.

\bibitem{ASC_22}
{\sc G.~Ascione, D.~Castorina, and F.~Solombrino}, {\em Mean-field sparse
  optimal control of systems with additive white noise}, Preprint
  ar{X}iv:2204.02431,  (2022).

\bibitem{Auletta2020HerdingSA}
{\sc F.~Auletta, D.~Fiore, M.~J. Richardson, and M.~di~Bernardo}, {\em Herding
  stochastic autonomous agents via local control rules and online target
  selection strategies}, Autonomous Robots, 46 (2020), pp.~469 -- 481.

\bibitem{Bog_2007}
{\sc V.~I. Bogachev, G.~Da~Prato, M.~R{\"o}ckner, and W.~Stannat}, {\em
  Uniqueness of solutions to weak parabolic equations for measures}, Bull.
  Lond. Math. Soc., 39 (2007), pp.~631--640.

\bibitem{bolley2011stochastic}
{\sc F.~Bolley, J.~A. Ca\~{n}izo, and J.~A. Carrillo}, {\em Stochastic
  mean-field limit: non-{L}ipschitz forces and swarming}, Math. Models Methods
  Appl. Sci., 21 (2011), pp.~2179--2210.

\bibitem{bongini2016optimal}
{\sc M.~Bongini and G.~Buttazzo}, {\em Optimal control problems in transport
  dynamics}, Math. Models Methods Appl. Sci., 27 (2017), pp.~427--451.

\bibitem{bonnet2023measure}
{\sc B.~Bonnet, C.~Cipriani, M.~Fornasier, and H.~Huang}, {\em A measure
  theoretical approach to the mean-field maximum principle for training
  {N}eur{ODE}s}, Nonlinear Anal., 227 (2023), pp.~Paper No. 113161, 55.

\bibitem{BF2021b}
{\sc B.~Bonnet and H.~Frankowska}, {\em Differential inclusions in
  {W}asserstein spaces: the {C}auchy-{L}ipschitz framework}, J. Differential
  Equations, 271 (2021), pp.~594--637.

\bibitem{BF2021a}
\leavevmode\vrule height 2pt depth -1.6pt width 23pt, {\em Necessary optimality
  conditions for optimal control problems in {W}asserstein spaces}, Appl. Math.
  Optim., 84 (2021), pp.~S1281--S1330.

\bibitem{BR_PMP}
{\sc B.~Bonnet and F.~Rossi}, {\em The {P}ontryagin maximum principle in the
  {W}asserstein space}, Calc. Var. Partial Differential Equations, 58 (2019),
  pp.~Paper No. 11, 36.

\bibitem{Herty-Pareschi}
{\sc G.~Borghi, M.~Herty, and L.~Pareschi}, {\em Constrained
  {C}onsensus-{B}ased {O}ptimization}, SIAM J. Optim., 33 (2023), pp.~211--236.

\bibitem{BKT21}
{\sc M.~Burger, L.~M. Kreusser, and C.~Totzeck}, {\em Mean-field optimal
  control for biological pattern formation}, ESAIM Control Optim. Calc. Var.,
  27 (2021), pp.~Paper No. 40, 24.

\bibitem{BPT21}
{\sc M.~Burger, R.~Pinnau, C.~Totzeck, and O.~Tse}, {\em Mean-field optimal
  control and optimality conditions in the space of probability measures}, SIAM
  J. Control Optim., 59 (2021), pp.~977--1006.

\bibitem{burger2020instantaneous}
{\sc M.~Burger, R.~Pinnau, C.~Totzeck, O.~Tse, and A.~Roth}, {\em Instantaneous
  control of interacting particle systems in the mean-field limit}, Journal of
  Computational Physics, 405 (2020), p.~109181.

\bibitem{CCDMP2021}
{\sc F.~Camilli, G.~Cavagnari, R.~De~Maio, and B.~Piccoli}, {\em Superposition
  principle and schemes for measure differential equations}, Kinet. Relat.
  Models, 14 (2021), pp.~89--113.

\bibitem{Camilli_games}
{\sc F.~Camilli, S.~Duisembay, and Q.~Tang}, {\em Approximation of an optimal
  control problem for the time-fractional {F}okker-{P}lanck equation}, J. Dyn.
  Games, 8 (2021), pp.~381--402.

\bibitem{CPT_2017}
{\sc P.~Cardaliaguet, A.~Porretta, and D.~Tonon}, {\em A segregation problem in
  multi-population mean field games}, in Advances in dynamic and mean field
  games, vol.~15 of Ann. Internat. Soc. Dynam. Games, Birkh\"{a}user/Springer,
  Cham, 2017, pp.~49--70.

\bibitem{MR3804923}
{\sc J.~A. Carrillo, Y.-P. Choi, C.~Totzeck, and O.~Tse}, {\em An analytical
  framework for consensus-based global optimization method}, Math. Models
  Methods Appl. Sci., 28 (2018), pp.~1037--1066.

\bibitem{carrillo2020mean}
{\sc J.~A. Carrillo, E.~A. Pimentel, and V.~K. Voskanyan}, {\em On a mean field
  optimal control problem}, Nonlinear Anal., 199 (2020), pp.~112039, 14.

\bibitem{CLOS_2022}
{\sc G.~Cavagnari, S.~Lisini, C.~Orrieri, and G.~Savar\'{e}}, {\em Lagrangian,
  {E}ulerian and {K}antorovich formulations of multi-agent optimal control
  problems: equivalence and gamma-convergence}, J. Differential Equations, 322
  (2022), pp.~268--364.

\bibitem{CGS2017}
{\sc G.~M. Coclite, M.~Garavello, and L.~V. Spinolo}, {\em A mathematical model
  for piracy control through police response}, NoDEA Nonlinear Differential
  Equations Appl., 24 (2017), pp.~Paper No. 48, 22.

\bibitem{PR_traffic2}
{\sc M.~L. Delle~Monache, B.~Piccoli, and F.~Rossi}, {\em Traffic regulation
  via controlled speed limit}, SIAM J. Control Optim., 55 (2017),
  pp.~2936--2958.

\bibitem{MR4147945}
{\sc G.~Dimarco, L.~Pareschi, G.~Toscani, and M.~Zanella}, {\em Wealth
  distribution under the spread of infectious diseases}, Phys. Rev. E, 102
  (2020), pp.~022303, 14.

\bibitem{BMPW2009}
{\sc B.~D\"{u}ring, P.~Markowich, J.-F. Pietschmann, and M.-T. Wolfram}, {\em
  Boltzmann and {F}okker-{P}lanck equations modelling opinion formation in the
  presence of strong leaders}, Proc. R. Soc. Lond. Ser. A Math. Phys. Eng.
  Sci., 465 (2009), pp.~3687--3708.

\bibitem{MR2551376}
{\sc B.~D\"{u}ring, D.~Matthes, and G.~Toscani}, {\em Kinetic equations
  modelling wealth redistribution: a comparison of approaches}, Phys. Rev. E
  (3), 78 (2008), pp.~056103, 12.

\bibitem{Evans}
{\sc L.~C. Evans}, {\em An introduction to stochastic differential equations},
  American Mathematical Society, Providence, RI, 2013.

\bibitem{fornasier2022anisotropic}
{\sc M.~Fornasier, H.~Huang, L.~Pareschi, and P.~S\"{u}nnen}, {\em Anisotropic
  diffusion in consensus-based optimization on the sphere}, SIAM J. Optim., 32
  (2022), pp.~1984--2012.

\bibitem{Flos}
{\sc M.~Fornasier, S.~Lisini, C.~Orrieri, and G.~Savar\'{e}}, {\em Mean-field
  optimal control as {G}amma-limit of finite agent controls}, European J. Appl.
  Math., 30 (2019), pp.~1153--1186.

\bibitem{Fou-Gui}
{\sc N.~Fournier and A.~Guillin}, {\em On the rate of convergence in
  wasserstein distance of the empirical measure}, Probab. Theory Relat. Fields,
  162 (2015), pp.~707--738.

\bibitem{HMC}
{\sc M.~Huang, R.~P. Malham\'{e}, and P.~E. Caines}, {\em Nash equilibria for
  large-population linear stochastic systems of weakly coupled agents}, in
  Analysis, control and optimization of complex dynamic systems, vol.~4 of
  GERAD 25th Anniv. Ser., Springer, New York, 2005, pp.~215--252.

\bibitem{Kac}
{\sc M.~Kac}, {\em Foundations of kinetic theory}, in Proceedings of the
  {T}hird {B}erkeley {S}ymposium on {M}athematical {S}tatistics and
  {P}robability, 1954--1955, vol. {III}, University of California Press,
  Berkeley-Los Angeles, Calif., 1956, pp.~171--197.

\bibitem{kalise2020sparse}
{\sc D.~Kalise, K.~Kunisch, and Z.~Rao}, {\em Sparse and switching infinite
  horizon optimal controls with mixed-norm penalizations}, ESAIM: Control,
  Optimisation and Calculus of Variations, 26 (2020), p.~61.

\bibitem{LL}
{\sc J.-M. Lasry and P.-L. Lions}, {\em Mean field games}, Jpn. J. Math., 2
  (2007), pp.~229--260.

\bibitem{osti_10386637}
{\sc G.~C. Maffettone, A.~Boldini, M.~Di~Bernardo, and M.~Porfiri}, {\em
  Continuification control of large-scale multiagent systems in a ring}, IEEE
  Control Systems Letters, 7 (2023).

\bibitem{Oksendal}
{\sc B.~{\O}ksendal}, {\em Stochastic differential equations}, Universitext,
  Springer-Verlag, Berlin, sixth~ed., 2003.
\newblock An introduction with applications.

\bibitem{gorlando}
{\sc G.~Orlando}, {\em Mean-field optimal control in a multi-agent interaction
  model for prevention of maritime crime}, Preprint ar{X}iv:2212.05341,
  (2022).

\bibitem{pareschi2014wealth}
{\sc L.~Pareschi and G.~Toscani}, {\em Wealth distribution and collective
  knowledge: a boltzmann approach}, Philosophical Transactions of the Royal
  Society A: Mathematical, Physical and Engineering Sciences, 372 (2014),
  p.~20130396.

\bibitem{PiccoliRossiARMA1}
{\sc B.~Piccoli and F.~Rossi}, {\em Generalized {W}asserstein distance and its
  application to transport equations with source}, Arch. Ration. Mech. Anal.,
  211 (2014), pp.~335--358.

\bibitem{PiccoliRossiARMA2}
\leavevmode\vrule height 2pt depth -1.6pt width 23pt, {\em On properties of the
  generalized {W}asserstein distance}, Arch. Ration. Mech. Anal., 222 (2016),
  pp.~1339--1365.

\bibitem{PR2018}
\leavevmode\vrule height 2pt depth -1.6pt width 23pt, {\em Measure-theoretic
  models for crowd dynamics}, in Crowd dynamics. {V}ol. 1, Model. Simul. Sci.
  Eng. Technol., Birkh\"{a}user/Springer, Cham, 2018, pp.~137--165.

\bibitem{PTZ2020}
{\sc B.~Piccoli, A.~Tosin, and M.~Zanella}, {\em Model-based assessment of the
  impact of driver-assist vehicles using kinetic theory}, Z. Angew. Math.
  Phys., 71 (2020), pp.~Paper No. 152, 25.

\bibitem{Pierson2018ControllingNH}
{\sc A.~Pierson and M.~Schwager}, {\em Controlling noncooperative herds with
  robotic herders}, IEEE Transactions on Robotics, 34 (2018), pp.~517--525.

\bibitem{Porretta}
{\sc A.~Porretta}, {\em Weak solutions to {F}okker-{P}lanck equations and mean
  field games}, Arch. Ration. Mech. Anal., 216 (2015), pp.~1--62.

\bibitem{Revuz}
{\sc D.~Revuz and M.~Yor}, {\em Continuous martingales and Brownian motion},
  vol.~293 of Grundlehren der mathematischen Wissenschaften [Fundamental
  Principles of Mathematical Sciences], Springer-Verlag, 1999.

\bibitem{Toscani2006}
{\sc G.~Toscani}, {\em Kinetic models of opinion formation}, Commun. Math.
  Sci., 4 (2006), pp.~481--496.

\bibitem{TZ2019}
{\sc A.~Tosin and M.~Zanella}, {\em Kinetic-controlled hydrodynamics for
  traffic models with driver-assist vehicles}, Multiscale Model. Simul., 17
  (2019), pp.~716--749.

\bibitem{MR4160256}
{\sc C.~Totzeck and M.-T. Wolfram}, {\em Consensus-based global optimization
  with personal best}, Math. Biosci. Eng., 17 (2020), pp.~6026--6044.

\bibitem{nature2016}
{\sc J.~Zhang, Z.~Huang, Z.~Wu, R.~Su, and Y.-C. Lai}, {\em Controlling herding
  in minority game systems}, Sci Rep, 6 (2016), p.~20925.

\end{thebibliography}

\end{document}